\documentclass[11pt,a4paper]{amsart}

\usepackage{amsmath,amsfonts,amssymb,amsthm}
\usepackage{txfonts}
\usepackage{latexsym,bm,graphicx}
\usepackage{mathrsfs}
\usepackage{color}

\title[Lipschitz continuity of harmonic maps]{Lipschitz continuity of harmonic maps between Alexandrov spaces}
\author{Hui-Chun Zhang}
\address{Department of Mathematics\\  Sun Yat-sen University\\ Guangzhou 510275\\ \newline E-mail address: zhanghc3@mail.sysu.edu.cn}
\author{Xi-Ping Zhu}
\address{Department of Mathematics\\  Sun Yat-sen University\\ Guangzhou 510275\\ \newline E-mail address: stszxp@mail.sysu.edu.cn}
\newtheorem{thm}{Theorem}[section]
\newtheorem{prop}[thm]{Proposition}
\newtheorem{lem}[thm]{Lemma}
\newtheorem{slem}[thm]{Sublemma}
\newtheorem{cor}[thm]{Corollary}
\newtheorem{fact}[thm]{Fact}

\newtheorem{conj}[thm]{Conjecture}

\theoremstyle{definition}
\theoremstyle{remark}

\newtheorem{defn}[thm]{Definition}
\newtheorem{rem}[thm]{Remark}

\numberwithin{equation}{section}

\newcommand{\ls}{\leqslant}
\newcommand{\gs}{\geqslant}
\newcommand{\wa}{\widetilde\angle}

\newcommand{\ip}[2]{\left<{#1},{#2}\right>}
\newcommand{\rv}{{\rm vol}}

\newcommand{\R}{\mathbb{R}}
\newcommand{\M}{\mathbb{M}}

\newcommand{\doueq}{=\!=}

\begin{document}


\begin{abstract}
In 1997, J. Jost \cite{jost97} and F. H. Lin \cite{lin97}, independently proved  that every energy minimizing harmonic map from an Alexandrov space with curvature bounded from below to an Alexandrov space with non-positive curvature is locally H\"older continuous.
In \cite{lin97}, F. H. Lin proposed an open problem: Can the H\"older continuity be improved to Lipschitz continuity?\  J. Jost also asked a similar problem about Lipschitz regularity of harmonic maps between singular spaces (see Page 38 in \cite{jost98}).
The main theorem of this paper gives a complete resolution to it.\\[6pt]
Mathematics Subject Classification (2010)\ \ \ \  58E20
\end{abstract}

\maketitle

\section{Introduction}

Given a map $u:M^n\to N^k$ between smooth Riemannian manifolds of dimension $n$ and $k$, there is a natural concept of energy associated to $u$.
The minimizers, or more general critical points of such an energy functional, are called harmonic maps. If $n=2$,
the regularity of  energy minimizing harmonic maps
was established by C. Morrey \cite{morrey48} in 1948. For energy minimizing harmonic maps defined on a higher dimensional Riemannian manifold, a well-known
 regularity theory has been developed by R. Schoen and K. Uhlenbeck  \cite{sch-u82} in 1982.
 In particular, in the case where the target space $N^k$ has non-positive sectional curvature,
  it has been proved that any energy minimizing harmonic map is smooth (see also \cite{Hilder77}). However,
 without any restriction on the target space $N^k$, an energy minimizing
 map might not be even continuous.

\subsection{Harmonic maps between singular spaces and H\"older continuity}$\  $

M. Gromov and R. Schoen \cite{gro-sch92} in 1992 initiated to study the theory of harmonic maps into singular spaces,
motivated by the $p$-adic superrigidity for lattices in groups of rank one. Consider a map $u:M\to Y$. If $Y$ is not a smooth manifold,
 the energy of $u$ can not be defined via its differential. A natural idea is to consider an energy concept as a limit of suitable difference quotients.
  The following concept of approximating energy for maps between metric spaces was introduced by N. Korevaar and R. Schoen in  \cite{ks93}.

 Let $(M,d_M)$, $(Y,d_Y)$ be two metric spaces and let $\Omega$ be a domain of $M$, equipped with a Radon measure ${\rm vol}$ on $M$. Given $p\gs1$, $\epsilon>0$ and a Borel measurable map $u:\Omega\to Y$,
 an approximating energy functional $E^u_{p,\epsilon}$ is defined on $C_0(\Omega)$, the set of continuous functions compactly supported in $\Omega$, as follows:
$$E^u_{p,\epsilon}(\phi):= c(n,p)\int_\Omega\phi(x)\int_{B_x(\epsilon)\cap\Omega}\frac{d^p_Y(u(x),u(y))}{\epsilon^{n+p}}d\rv(y)d\rv(x)$$
where $\phi\in C_0(\Omega)$ and $c(n,p)$ is a normalized constant.

 In the case where $\Omega$ is a domain of a smooth Riemannian manifold and $Y$ is
an arbitrary metric space,  N. Korevaar and R. Schoen \cite{ks93} proved that  $E^u_{p,\epsilon}(\phi)$ converges weakly,
 as a linear functional on $C_0(\Omega)$, to
some (energy) functional $E^u_p(\phi)$. The same convergence has been established for the case where $\Omega$ is replaced with one of the following:\\
{\footnotesize$\bullet$} \ a domain of a Lipschitz manifold (by G. Gregori \cite{gre98});\\
 {\footnotesize$\bullet$} \ a domain of a Rimannian polyhedron (for $p\!=\!2$, by J. Eells and B. Fuglede  \cite{ef01});\\
{\footnotesize $\bullet$} \  a domain of a singular space with certain condition,
  including Alexandrov spaces with curvature bounded from below, abbreviated by CBB for short
   (by K. Kuwae and T. Shioya  \cite{ks-sob03}).

 When $p=2$, minimizing maps, in the sense of calculus of variations, of such an energy functional $E^u_2(\phi)$ are called \emph{harmonic maps}.

  K-T. Sturm \cite{st-har} studied a generalization of the theory of harmonic maps between singular spaces via an approach of probabilistic theory.

The purpose of this paper is to study the regularity theory of harmonic
maps from a domain of an Alexandrov space with CBB into
  a complete length space of non-positive curvature in the sense of Alexandrov, abbreviated by NPC for short.
This problem was initiated by F. H. Lin \cite{lin97} and J. Jost \cite{jost96,jost97,jost98}, independently.  They established the following H\"older regularity.
 \begin{thm}[Lin \cite{lin97}, Jost\footnote{J. Jost worked on a generalized Dirichlet form on a larger class of metric spaces.} \cite{jost97}]
 \label{thm1.1}
 Let $\Omega$ be a bounded domain in an Alexandrov space with CBB, and
 let $(Y,d_Y)$ be an NPC space. Then any harmonic map $u:\Omega\to Y$ is locally H\"older continuous in $\Omega$.
  \end{thm}

The H\"older regularity of harmonic maps between singular spaces or into singular spaces has been also studied by many other authors. For example,
J. Chen \cite{chen95},  J. Eells \& B. Fuglede \cite{ef01,f03,f08}, W. Ishizuka \& C. Y. Wang \cite{iw08} and G. Daskalopoulos \& C. Mese \cite{dm08,dm10}, and others.

\subsection{Lipschitz continuity and main result}$\ $

 F. H.  Lin  \cite{lin97}  proposed an open problem: whether the H\"older continuity in the above Theorem \ref{thm1.1}
 can be improved to Lipschitz continuity?\  Precisely,
 \begin{conj}[Lin \cite{lin97}]\label{lin-conj}
 Let $\Omega, Y$ and $u$ be as in Theorem \ref{thm1.1}. Is $u$ locally Lipschitz continuous in $\Omega$?
\end{conj}

J. Jost also asked a similar problem about Lipschitz regularity of harmonic maps between singular spaces (see Page 38 in \cite{jost98}).
The Lipschitz continuity of harmonic maps is the key in establishing rigidity theorems of geometric group theory in \cite{gro-sch92,dm08,dm09}.

Up to now,
 there are only a few answers for some special cases.

  The first is the case where the target space $Y=\mathbb R$, i.e., the theory of harmonic functions.
 The Lipschitz regularity of harmonic functions on singular spaces
 has been obtained under one of the following two assumptions: (i) $\Omega$ is a domain of a metric space,
 which supports a doubling measure, a Poincar\'e inequality and a certain heat kernel
 condition (\cite{koskela03,jiang}); (ii) $\Omega$ is a domain of an Alexandrov space with
 CBB (\cite{petu96,pet-har,zz12}). Nevertheless, these proofs depend heavily on the linearity of the Laplacian on such spaces.

 It is known from \cite{c99} that the H\"older continuity always holds for any harmonic function on a metric measure space $(M,d,\mu)$
  with a standard assumption: the  measure $\mu$ is doubling  and $M$ supports a Poincar\'e inequality (see, for example, \cite{c99}).
 However, in \cite{koskela03}, a counterexample was given to show that such a standard assumption is not sufficient to guarantee the
 Lipschitz continuity of harmonic functions.

The second is the case where $\Omega$ is a domain of some smooth Riemannian manifold and $Y$ is an NPC space. N. Korevaar and R. Schoen \cite{ks93} in 1993
 established the following
  Lipschitz regularity for any harmonic map from $\Omega$ to $Y$.
 \begin{thm}[Korevaar-Schoen \cite{ks93}]
 \label{ksthm}
  Let $\Omega$ be a bounded domain of a smooth Riemannian manifold $M$,
 and let $(Y,d_Y)$ be an NPC  metric space. Then any harmonic map $u:\Omega\to Y$ is locally Lipschitz continuous in $\Omega$.
 \end{thm}
\noindent However, their Lipschitz constant in the above theorem depends on the $C^1$-norm of the metric $(g_{ij})$ of the smooth manifold $M$.
 In Section 6 of \cite{jost96}, J. Jost described a new argument for the above Korevaar-Schoen's Lipschitz regularity using intersection properties of balls. The Lipschitz
 constant given by Jost depends on the upper and lower bounds of Ricci curvature on $M$.
This does not seem to suggest a Lipschitz regularity of harmonic maps from a singular space.

The major obstacle to prove a Lipschtz continuity of harmonic maps from a singular space can be understood as follows.
For the convenience of the discussion, we consider a harmonic
map $u\!:\! (\Omega,g)\!\to\! N$  from a domain $\Omega\!\subset\!\mathbb R^n$ with a \emph{singular} Riemanian metric $ g=(g_{ij}) $ into
 a smooth non-positively curved manifold $N$,
which by the Nash embedding theorem is isometrically embedded in some Euclidean space $\mathbb R^K$. Then  $u$ is a solution of the nonlinear
elliptic system of  divergence form
\begin{equation}\label{system}
  \frac{1}{\sqrt g}\partial_i(\sqrt gg^{ij}\partial_j u_\alpha)+g^{ij}A^\alpha\big(\partial_iu,\partial_ju\big)=0,\quad\alpha=1,\cdots, K
\end{equation}
in the sense of distribution, where $g=\det(g_{ij})$, $(g^{ij})$ is the inverse matrix of $(g_{ij})$, and $ A^\alpha$ is the second fundamental form of $N$.
It is  well-known that, as a second order elliptic system, the regularity of solutions is determined by  regularity of its coefficients. If the coefficients
 $\sqrt gg^{ij}$ are
merely  bounded measurable,  Y. G. Shi \cite{shi} proved that the solution $u$ is H\"older continuous. But, a harmonic map might fail to be Lipschitz continuous, even with
assumption that the coefficients are continuous. See \cite{mazya} for a counterexample for this.

The above Lin's conjecture is about the Lipschitz continuity for harmonic maps between Alexandrov spaces. Consider $M$ to be an Alexandrov space with CBB  and let $p\in M$ be a regular point. According to \cite{os94,per-dc}, there is
 a coordinate neighborhood $U\!\ni\! p$ and a corresponding $BV_{\rm loc}$-Riemannian metric $(g_{ij})$ on $U$.
  Hence,  the coefficients $\sqrt gg^{ij}$ of  elliptic system \eqref{system} are measurable on $U$.
However, it is well-known \cite{os94} that they may not be continuous on a \emph{dense} subset of $U$ for general Alexandrov spaces with CBB. Thus, it is apparent that
 the above Lin's conjecture might not be true.

Our main result in this paper is the following affirmative resolution to the above Lin's problem, Conjecture \ref{lin-conj}.
\begin{thm}
\label{main-thm}
Let $\Omega$ be a bounded domain in an $n$-dimensional Alexandrov space $(M,|\cdot,\cdot|)$ with curvature $\gs k$ for some constant $k\ls0$, and let $(Y,d_Y)$ be an NPC space (not necessary locally compact). Assume that
   $u:\Omega\to Y$ is a harmonic map. Then, for any ball $B_q(R)$ with $B_q(2R)\subset \Omega$ and $R\ls1$, there exists a constant $C(n,k,R)$, depending only on $n,k$ and $R$, such that
  $$\frac{d_Y\big(u(x),u(y)\big)}{|xy|}\ls C(n,k,R)\cdot \bigg( \Big(\frac{ E^u_2\big(B_q(R)\big)}{\rv\big(B_q(R)\big)}\Big)^{1/2}+{\rm osc}_{\overline{B_q(R)}}u\bigg)$$
for all $x,y\in B_q(R/16),$ where $E^u_2(B_q(R))$ is the energy of $u$ on $B_q(R)$.
\end{thm}
\begin{rem}
 A curvature condition on domain space is necessary. Indeed, J. Chen \cite{chen95} constructed a harmonic function $u$ on a two-dimensional metric cone $M$
  such that $u$ is not Lipschitz continuous if $M$ has no a lower curvature bound.
\end{rem}

\subsection{Organization of the paper}$\  $

The paper is composed of six sections. In Section 2, we will provide some necessary
materials on Alexandrov spaces.
In  Section 3, we will recall basic analytic results on Alexandrov spaces, including Sobolev spaces, super-solutions of Poisson equations in the sense of distribution and super-harmonicity in the sense of Perron.
 In Section 4, we will review the concepts of energy and approximating energy, and then we will prove a point-wise convergence result  for their densities.
In Section 5, we will recall some basic results on existence and H\"older regularity of  harmonic map into NPC spaces.
We will then give an estimate for point-wise Lipschitz constants of such a harmonic map.
The Section 6 is devoted to the proof of the main Theorem \ref{main-thm}.

\noindent\textbf{Acknowledgements.} Both  authors are partially supported by NSFC 11521101.
The first author is partially supported by NSFC 111571374 and by ``National Program for Support of Top-notch Young Professionals".

\section{Preliminaries}

\subsection{Basic concepts on Alexandrov spaces with curvature $\gs k$}$\  $

Let $k\in\R$ and $l\in\mathbb N$. Denote by $\mathbb M^l_k$ the simply connected, $l$-dimensional space form of
 constant sectional curvature $k$. The space $\mathbb M^2_k$ is called $k$-plane.

Let $(M,|\cdot\cdot\ |)$ be a complete metric space. A rectifiable curve $\gamma$ connecting two points $p,q$ is called a
 \emph{geodesic} if its length is equal to $|pq|$ and it has unit speed. A metric space $M$ is called a \emph{geodesic space}
  if, for every pair points $p,q\in M$, there exists some geodesic connecting them.

Fix any $k\in\R$. Given three points $p,q,r$ in a geodesic space $M$, we can take a
  triangle $\triangle \bar p\bar q\bar r$ in $k$-plane $\M^2_k$ such that $|\bar p\bar q|=|pq|$,
  $|\bar q\bar r|=|qr|$ and $|\bar r\bar p|=|rp|$. If $k>0$, we add the assumption  $|pq|+|qr|+|rp|<2\pi/\sqrt{k}$.
  The triangle $\triangle \bar p\bar q\bar r\subset \mathbb M^2_k$ is unique up to a rigid motion.
  We let $\wa_k pqr$ denote the angle at the vertex $\bar q$ of the triangle $\triangle \bar p\bar q\bar r$, and we call it a \emph{$k$-comparison angle}.

\begin{defn}
 Let $k\in \mathbb R$. A geodesic space $M$ is called an \emph{Alexandrov space with curvature} $\gs k$ if it satisfies the
 following properties:\\
\indent(i) it is locally compact;\\
\indent(ii)  for any point $x\in M$, there exists a neighborhood $U$ of $x$ such that the following condition is satisfied:
 for any two geodesics $\gamma(t)\subset U$ and $\sigma(s)\subset U$ with $\gamma(0)=\sigma(0):=p$, the $k$-comparison angles
 $$\wa_\kappa \gamma(t)p\sigma(s)$$
is non-increasing with respect to each of the variables $t$ and $s$.
\end{defn}

It is well known that the Hausdorff dimension of an Alexandrov space with curvature $\gs k$, for some constant $k\in \mathbb R$, is always an integer or $+\infty$ (see,
 for example, \cite{BBI} or \cite{bgp92}).
In the following, the terminology of ``an ($n$-dimensional) Alexandrov space $M$" means that $M$
 is an Alexandrov space with curvature $\gs k$ for some $k\in \mathbb R$ (and that its Hausdorff dimension $=n$).
We denote by $\rv$ the $n$-dimensional Hausdorff measure on $M$.

On an $n$-dimensional Alexandrov space $M$, the angle between any two geodesics $\gamma(t)$ and $\sigma(s)$
 with $\gamma(0)=\sigma(0)\!:=p$ is well defined, as the limit
 $$\angle \gamma'(0)\sigma'(0)\!:=\lim_{s,t\to0}\wa_\kappa \gamma(t)p\sigma(s).$$
 We denote by $\Sigma'_p$ the set of equivalence classes of geodesic $\gamma(t)$ with $\gamma(0)=p$,
  where $\gamma(t)$ is equivalent to $\sigma(s)$ if $\angle\gamma'(0)\sigma'(0)=0$.
  $(\Sigma_p',\angle)$ is a metric space, and its completion is called the \emph{space of directions at} $p$, denoted by $\Sigma_p$.
  It is  known (see,
 for example, \cite{BBI} or \cite{bgp92}) that $(\Sigma_p,\angle)$ is an Alexandrov space with curvature $\gs1$ of dimension $n-1$.
 It is also known (see,
 for example, \cite{BBI} or \cite{bgp92}) that the \emph{tangent cone at} $p$, $T_p$,
   is the Euclidean cone over $\Sigma_p$. Furthermore, $T_p^k$ is the $k$-cone over $\Sigma_p$ (see Page 355 in \cite{BBI}).
  For two tangent vectors $u,v\in T_p$, their ``scalar product" is defined by (see Section 1 in \cite{pet-con})
$$\ip{u}{v}:=\frac{1}{2} (|u|^2 + |v|^2- |uv|^2).$$

Let $p\in M$. Given a direction $\xi\in \Sigma_p$, we remark that there does possibly not exists geodesic $\gamma(t)$ starting at $p$
 with $\gamma'(0)=\xi$.

 We refer to the seminar paper \cite{bgp92} or the text book \cite{BBI} for the details.

\begin{defn}[Boundary, \cite{bgp92}] The boundary of an Alexandrov space $M$ is defined inductively with respect to dimension.
 If the dimension of $M$ is one, then $M$ is a complete
Riemannian manifold and the \emph{boundary }of $M$ is defined as usual. Suppose that the dimension of $M$ is $n\gs2$.
A point $p$ is a \emph{boundary point} of $M$ if $\Sigma_p$ has non-empty boundary.
\end{defn}

\begin{center}\emph{From now on, we always consider Alexandrov spaces without boundary.}
 \end{center}

\subsection{The exponential map and second variation of arc-length}$\  $

 Let $M$ be an $n$-dimensional Alexandrov space and $p\in M$.
 For each point $x\not=p$, the symbol $\uparrow_p^x$ denotes the direction at $p$ corresponding to \emph{some} geodesic $px$.
 Denote by (\cite{os94})
 $$W_p:=\big\{x\in M\backslash\{p\} \big|\ {\rm  geodesic}\  px\ {\rm can\ be\ extended\ beyond}\ x\big\}.$$
 According to \cite{os94}, the set $W_p$ has full measure in $M$. For each $x\in W_p$, the direction
  $\uparrow_p^x$ is uniquely determined, since any geodesic in $M$ does not branch (\cite{bgp92}).
  Recall that the  map $\log_p: W_p\to T_p$ is defined by $\log_p(x):=|px|\cdot\uparrow_p^x$ (see \cite{pet-con}).
It is one-to-one from $W_p$ to its image
$$ \mathscr W_p:=\log_p(W_p)\subset T_p.$$

The inverse map of $\log_p$,
 $$\exp_p=(\log_p)^{-1}:  \mathscr W_p\to W_p,$$
 is called the \emph{exponential map at} $p$.

One of the technical difficulties in Alexandrov geometry comes from the fact that $\mathscr W_p$ may not contain any neighbourhood of the vertex of the cone $T_p$.

If $M$ has curvature $\gs k$ on $B_p(R)$, then exponential map
     $$\exp_p: B_o(R)\cap \mathscr W_p\subset T^k_p\to M$$ is a non-expending map (\cite{bgp92}), where $T^k_p$
      is the $k$-cone over $\Sigma_p$ and $o$ is the vertex of $T_p$.

In \cite{pet-para}, A. Petrunin established the notion of parallel transportation and second variation of arc-length on Alexandrov spaces.
\begin{prop}[Petrunin, Theorem 1.1. B in \cite{pet-para}]\label{para}
Let $k\in\mathbb R$ and let $M$ be an $n$-dimensional Alexandrov space with curvature $\gs k$. Suppose that points $p$ and $q$ such that
the geodesic $pq$ can be extended beyond both $p$ and $q$.

Then, for any fixed sequence $\{\epsilon_j\}_{j\in\mathbb N}$ going to $0$, there exists an isometry $T:T_p\to T_q$
and a subsequence $\{\varepsilon_j\}_{j\in\mathbb N}\subset \{\epsilon_j\}_{j\in\mathbb N}$ such that
\begin{equation}
\big|\exp_{p}(\varepsilon_j\cdot \eta)\ \exp_{q}(\varepsilon_j\cdot T\eta)\big |\ls |pq|-\frac{k\cdot|pq|}{2}|\eta|^2\cdot\varepsilon^2_j+o(\varepsilon_j^2)
\end{equation}
for any $\eta\in T_p$ such that the left-hand side is well-defined.
\end{prop}
Here and in the following, we denote by $g(s)=o(s^\ell)$ if  the function $g(s)$ satisfies $\lim_{s\to0^+}\frac{g(s)}{s^\ell}=0.$

\subsection{Singularity, regular points, smooth points and $C^\infty$-Riemannian approximations}$\ $

Let $k\in\mathbb R$ and let $M$ be an $n$-dimensional Alexandrov space with curvature $\gs k$. For any $\delta>0$, we denote
 $$ M^\delta:=\big\{x\in M:\ \rv(\Sigma_x)>(1-\delta)\cdot\rv(\mathbb S^{n-1})\big\},$$
 where $\mathbb S^{n-1}$ is the standard $(n-1)$-sphere. This is an open  set (see \cite{bgp92}).
 The set $S_\delta:= M\backslash M^\delta$ is called the $\delta$-\emph{singular set}.  Each point $p\in S_\delta$ is called a $\delta$-\emph{singular point}.
The set
$$S_M:=\cup_{\delta>0}S_\delta$$
is called \emph{singular set}. A point $p\in M$ is called a \emph{singular point} if $p\in S_M$. Otherwise it is called a \emph{regular point}. Equivalently,
 a point $p$ is regular if and only if $T_p$ is isometric to $\mathbb R^n$ (\cite{bgp92}). At a regular point $p$, we have that $T^k_p$ is isometric $\mathbb M_k^n.$
 Since we always assume that the boundary of $M$ is empty, it is proved in \cite{bgp92} that the Hausdorff dimension of $S_M$ is $\ls n-2.$
 We remark that the singular set $S_M$ might be dense in $M$ (\cite{os94}).

Some basic structures of Alexandrov spaces have been known in the following.
\begin{fact} \label{fact}
Let $k\in\mathbb R$ and let $M$ be an $n$-dimensional Alexandrov space with curvature $\gs k$.\\
(1) \ There exists a constant $\delta_{n,k}>0$ depending only on the dimension $n$ and $k$ such that for each $\delta\!\in\!(0,\delta_{n,k})$, the set
$ M^{\delta}$ forms a Lipschitz manifold (\cite{bgp92}) and has a $C^\infty$-differentiable structure (\cite{kms01}).\\
(2)\  There exists a $BV_{\rm loc}$-Riemannian metric $g$ on $ M^\delta$ such that \\
\indent $\bullet$ the metric $g$ is continuous in $M\backslash S_M$ (\cite{os94,per-dc});\\
\indent $\bullet$ the distance function on $M\backslash S_M$ induced from $g$ coincides with the original one of $M$ (\cite{os94});\\
\indent $\bullet$ the Riemannian measure on $M\backslash S_M$ induced from $g$ coincides with the Hausdorff measure of $M$ (\cite{os94}).
\end{fact}

 A point $p$ is called a \emph{smooth} point if it is regular and there exists a coordinate system $(U,\phi)$
   around $p$ such that \begin{equation}\label{eq2.2}
  |g_{ij}(\phi(x))-\delta_{ij}|=o(|px|),
  \end{equation}
   where $(g_{ij})$ is the corresponding Riemannian metric in the above Fact \ref{fact} (2) near $p$ and $(\delta_{ij})$
   is the identity $n\times n$ matrix.

 It is shown in \cite{per-dc} that the set of smooth points has full measure.
  The following asymptotic behavior of $ W_p$ around a smooth point $p$ is proved in \cite{zz12}.
 \begin{lem}[Lemma 2.1 in \cite{zz12}]\label{smooth}
  Let $p\in M$ be a smooth point.  We have
   \begin{equation*}
   \Big|\frac{d\rv(x)}{dH^n(v)}-1\Big|=o(r),\qquad \forall\ x\in W_p\cap B_p(r),\quad v=\log_p(x)
    \end{equation*}
and
\begin{equation}\label{eq2.3}
\frac{{H^n} \big(B_o(r)\cap \mathscr W_p\big)}{{H^n}\big(B_o(r)\big)}\gs1-o(r).
\end{equation}
 where $B_o(r)\subset T_p$ and $H^n$ is $n$-dimensional Hausdorff measure on $T_p\ (\overset{{\rm isom}}{\approx}\mathbb R^n)$.
 \end{lem}

 The following property on smooth approximation is contained in the proof of Theorem 6.1 in \cite{kms01}. For the convenience, we state it as a lemma.
\begin{lem}[Kuwae-Machigashira-Shioya \cite{kms01}, $C^\infty$-approximation]\label{appo}
 Let $k\in\mathbb R$ and let $M$ be an $n$-dimensional Alexandrov space with curvature $\gs k$. The constant
 $\delta_{n,k}$ is given in the above Fact \ref{fact} (1).

Let $0<\delta<\delta_{n,k}$. For any compact set $C\subset  M^\delta$,
there exists  an neighborhood $U$ of $C$ with $U\subset M^\delta$ and a $C^\infty$-Riemannian metric $g_\delta$ on $U$ such that
the distance $d_\delta$ on $U$ induced from $g_\delta$ satisfies
\begin{equation}\label{eq2.4}
\bigg|\frac{d_\delta(x,y)}{|xy|}-1 \bigg|<\kappa(\delta) \qquad {\rm for\ any}\ x,y\in U, x\not=y,
\end{equation}
where $\kappa(\delta)$ is a positive function (depending only on $\delta$)
with $\lim_{\delta\to0}\kappa(\delta)=0.$
\end{lem}
\begin{proof}
In the first paragraph of the proof of Theorem 6.1 in \cite{kms01} (see page 294), the authors constructed a $\kappa(\delta)$-almost isometric homeomorphism $F$
from an neighborhood $U$ of $C$ to some $C^\infty$-Riemannian manifold $N$ with distance function $d_N$. That is, the map $F:U\to N$ is a bi-Lipschitz homeomorphism satisfying
$$\bigg|\frac{d_N(F(x),F(y))}{|xy|}-1\bigg|<\kappa(\delta) \qquad {\rm for\ any}\ x,y\in U, x\not=y.$$
Now let us consider the distance function $d_\delta$ on $U$ defined by
 $$d_\delta(x,y):=d_N\big(F(x),F(y)\big).$$
The map $F: (U,d_\delta)\to (N,d_N)$ is an isometry, and hence the desired $C^\infty$-Riemannian metric $g_\delta$ can be defined by the pull-back of the Riemanian metric $g_N$.
\end{proof}

\subsection{Semi-concave functions and Perelman's concave functions}$\ $

Let $M$ be an Alexandrov space without boundary and $\Omega\subset M $ be an open set. A locally Lipschitz
function $f: \Omega\to\R$ is called to be  $\lambda$-\emph{concave} (\cite{pet-con}) if for all geodesics $\gamma(t)$ in $ \Omega$, the function
$$f\circ\gamma(t)-\lambda\cdot t^2/2$$
is concave.  A function $f: \Omega\to\R$ is called to be \emph{semi-concave} if for any $x\in \Omega$, there exists
 a neighborhood of $U_x\ni  x$ and a number $\lambda_x\in \R$ such that $f|_{U_x}$ is  $\lambda_x$-concave.
  (see Section 1 in \cite{pet-con} for the basic properties of semi-concave functions).

\begin{prop}
[Perelman's concave function, \cite{per-morse,k02}]\label{per-concave}\indent  Let $p\in M$. There exists a constant $r_1>0$
 and  a function $h: B_p(r_1)\to\R$ satisfying:\\
\indent (i)\indent $h$ is $(-1)$--concave;\\
\indent (ii)\indent$h$ is $2$-Lipschitz, that is, $h$ is Lipschitz continuous with a Lipschitz constant 2.
\end{prop}
We refer the reader to \cite{zz12} for the further properties for Perelman's concave functions.

\section{Analysis on Alexandrov spaces}
In this section, we will summarize some basic analytic results on Alexandrov spaces, including Sobolev spaces, Laplacian and harmonicity via Perron's method.
\subsection{Sobolev spaces on Alexandrov spaces} $\ $

Several different notions of Sobolev spaces on metric  spaces have been established,
 see\cite{c99,kms01,n00,ks93,ks-sob03,hkst01}\footnote{In \cite{c99,ks93,n00,ks-sob03,hkst01}, Sobolev spaces
  are defined on  metric measure spaces supporting a doubling property and a Poincar\'e inequality.}.
    They coincide with each other on Alexandrov spaces.

Let $M$ be an $n$-dimensional Alexandrov space with curvature $\gs k$ for some $k\in \mathbb R$. It is well-known
 (see \cite{kms01} or the survey \cite{zz10-2}) that the metric measure space $(M,| \cdot \cdot\ |,\rv)$
  is locally doubling and supports a local (weak) $L^2$-Poincar\'e inequality. Moreover, given a bounded domain $\Omega\subset M$, both the doubling constant $C_d$ and the
Poincar\'e constant $C_P$ on $\Omega$  depend only on $n,k$ and ${\rm diam}(\Omega)$.

Let $\Omega$ be an open domain in $M$.
Given $f\in C(\Omega)$ and point $x\in\Omega$,
  the \emph{pointwise Lipschitz constant} (\cite{c99})
    of $f$ at $x$ is defined by:
$${\rm Lip}f(x):=\limsup_{y\to x}\frac{|f(x)-f(y)|}{|xy|}.$$

We denote by $Lip_{\rm loc}(\Omega)$ the set of locally Lipschitz continuous functions on $\Omega$, and by $Lip_0(\Omega)$
 the set of Lipschitz continuous functions on $\Omega$ with compact support in $\Omega.$ For any $1\ls p\ls +\infty$ and
   $f\in Lip_{\rm loc}(\Omega)$, its $W^{1,p}(\Omega)$-norm is defined by
$$\|f\|_{W^{1,p}(\Omega)}:=\|f\|_{L^{p}(\Omega)}+\|{\rm Lip}f\|_{L^{p}(\Omega)}.$$
The Sobolev space  $W^{1,p}(\Omega)$ is defined by the closure of the set
$$\{f\in Lip_{\rm loc}(\Omega)|\ \|f\|_{W^{1,p}(\Omega)}<+\infty\},$$
under  $W^{1,p}(\Omega)$-norm.
The space $W_0^{1,p}(\Omega)$ is defined by the closure of $Lip_0(\Omega)$ under  $W^{1,p}(\Omega)$-norm.
(This coincides with the definition in \cite{c99}, see Theorem 4.24 in \cite{c99}.)
We say a function $f\in W^{1,p}_{\rm loc}(\Omega)$ if $f\in W^{1,p}(\Omega')$ for every open subset $\Omega'\subset\subset\Omega.$ Here and in the following,
 ``$\Omega'\subset\subset\Omega$"
means $\Omega'$ is compactly contained in $\Omega$.
  In Theorem 4.48 of \cite{c99}, Cheeger
   proved that $W^{1,p}(\Omega)$ is reflexible for any $1<p<\infty.$

\subsection{Laplacian and super-solutions} $\ $

Let us recall a concept of Laplacian  \cite{petu11,zz12}
 on Alexandrov spaces, as a functional acting on the space of Lipschitz functions with compact support.

Let $M$  be an $n$-dimensional  Alexandrov space and $\Omega$
 be a bounded domain in $M$.
 Given a function $f\in W_{\rm loc}^{1,2}(\Omega)$,  we define  a functional $\mathscr L_f$ on $Lip_0(\Omega)$, called the \emph{Laplacian functional} of $f$, by
$$\mathscr L_f(\phi):=-\int_\Omega\ip{\nabla f}{\nabla \phi}d{\rm vol},\qquad \forall \phi\in Lip_0(\Omega).$$

When a function $f$ is   $\lambda$-concave, Petrunin in \cite{petu11} proved that   $\mathscr L_f$ is a signed Radon measure.
 Furthermore, if we write its Lebesgue decomposition as
\begin{equation}
\label{eq3.1}
\mathscr L_f=\Delta f\cdot{\rm vol}+\Delta^s f,
\end{equation}
 then
 \begin{center}
 $\Delta^sf\ls 0\quad$ and $\quad\Delta f\cdot\rv\ls n\cdot\lambda\cdot\rv$.
 \end{center}

Let $h\in L^1_{\rm loc}(\Omega) $ and $f\in W_{\rm loc}^{1,2}(\Omega)$. The function $f$ is said to be a \emph{super-solution} (\emph{sub-solution}, resp.) of the Poisson equation
  $$\mathscr L_f=h\cdot{\rm vol},$$
if the functional $\mathscr L_f$  satisfies
$$\mathscr L_f(\phi)\ls\int_\Omega h\phi d{\rm vol}\qquad \Big({\rm or}\quad \mathscr L_f(\phi)\gs\int_\Omega h\phi d{\rm vol}\Big)$$
 for all nonnegative $\phi\in Lip_0(\Omega)$. In this case, according to the Theorem 2.1.7 of \cite{h89}, the functional $\mathscr L_f$ is a signed Radon  measure.

Equivalently,  $f\in W_{\rm loc}^{1,2}(\Omega)$ is sub-solution of $\mathscr L_f=h\cdot{\rm vol}$ if and only if
 it is a local minimizer of the energy
 $$\mathcal E(v)=\int_{\Omega'}\big(|\nabla v|^2+2hv\big)d{\rm vol}$$
  in the set of functions $v$ such that $f\gs v$ and $f-v$ is  in $W^{1,2}_0(\Omega')$ for every fixed $\Omega'\subset\subset\Omega.$
   It is known (see for example \cite{k08}) that every continuous super-solution of $\mathscr L_f=0$ on $\Omega$ satisfies Maximum Principle,
    which states that
    $$\min_{x\in \Omega'}f\gs \min_{x\in\partial \Omega'}f$$
     for any open set $\Omega'\subset\subset\Omega.$


 A function $f$ is a (\emph{weak}) \emph{solution} (in the sense of distribution) of Poisson equation $\mathscr L_f=h\cdot{\rm vol}$ on $\Omega$ if it is
 both a sub-solution and a super-solution of the equation. In particular, a (weak) solution of $\mathscr L_f=0$ is called a harmonic function.

Now remark that $f$ is a (weak) solution of Poisson equation $\mathscr L_f=h\cdot{\rm vol}$ if and only if $\mathscr L_f$ is a signed Radon measure
 and its Lebesgue's decomposition $\mathscr L_f=\Delta f\cdot \rv+\Delta^sf$ satisfies
$$\Delta f=h\qquad{\rm and}\qquad\Delta^sf=0.$$

 Given a function $h\in L^2(\Omega)$ and $g\in W^{1,2}(\Omega)$, we can solve the  Dirichlet problem of the equation
  \begin{equation*}
 \begin{cases}\mathscr L_f&=h\cdot{\rm vol}\\ f&=g|_{\partial\Omega}.\end{cases}
 \end{equation*}
  Indeed, by the Sobolev  embedding theorem (see \cite{hk00,kms01}) and a standard argument
   (see, for example, \cite{gt01}),  it is known that the solution of the Dirichlet problem exists uniquely  in
    $W^{1,2}(\Omega).$ (see, for example, Theorem 7.12 and Theorem 7.14 in \cite{c99}.)
    Furthermore, if we add the
     assumption $h\in L^s$ with $s>n/2$, then the solution $f$ is locally H\"older continuous  in $\Omega$ (see \cite{kshan01,kms01}).
\begin{lem}
\label{lem3.1}
Let $\Omega$ be a bounded domain of an Alexandrov space. Assume that
 $g\in L^\infty(\Omega)$. If $f\in W^{1,2}(\Omega)$ is a weak solution of the Poisson equation
 $$\mathscr L_f=g\cdot\rv.$$
 Then $f$ is locally Lipschitz continuous in $\Omega$.
\end{lem}
\begin{proof}
In \cite[Theorem 3.1]{jky14}, it has been shown that Yau's gradient estimate for harmonic functions implies that the local Lipschitz continuity  for solutions of $\mathscr L_f=g\cdot\rv.$ On the other hand, Yau's gradient estimate for harmonic functions has been established in \cite{zz12} (see also \cite{jiang}).
\end{proof}

The following mean value inequality is a slight extension of Corollary 4.5 in \cite{zz12}.
\begin{prop}\label{mean}
Let $M$  be an $n$-dimensional  Alexandrov space and $\Omega$
 be a bounded domain in $M$. Assume function $h\in L^1_{\rm loc}(\Omega)$ with $h(x)\ls C$ for some constant $C$.
 Suppose that $f\in  W^{1,2}_{\rm loc}(\Omega)\cap C(\Omega)$ is nonnegative and satisfies that
 $$\mathscr L_f\ls h\cdot \rv.$$
If $p\in \Omega$ is a Lebesgue point of $h$, then
 $$\frac{1}{H^{n-1}(\partial B_o(R)\subset T^k_p)}\int_{\partial B_p(R)}f(x)d\rv\ls f(p)+\frac{h(p)}{2n}\cdot R^{2}+o(R^{2}).$$
 \end{prop}
\begin{proof} The same assertion has been proved under the added assumption that $h\in L^\infty$ in Corollary 4.5 in \cite{zz12}.
Here, we will use an approximated argument.

 For each $j\in \mathbb N$, by setting $h_j:=\max\{-j,h\}$, we  conclude that $h_j\in L^\infty(\Omega)$, $h_j$ is monotonely converging to $h$,  and
$$\mathscr L_f\ls h\cdot\rv\ls h_j\cdot\rv,\qquad \forall\ j\in \mathbb N.$$

For any $p\in\Omega$, by using Proposition 4.4 in \cite{zz12}, we have, for all $R>0$ with $B_p(R)\subset\subset \Omega$ and for each $j\in \mathbb N$,
\begin{equation*}
 \frac{1}{H^{n-1}(\partial B_o(R)\subset T^k_p)}\int_{\partial B_p(R)}fd\rv-f(p)\ls (n-2)\cdot\frac{\omega_{n-1}}{\rv(\Sigma_p)}\cdot\varrho_j(R),
\end{equation*}
where
\begin{equation*}\begin{split}
\varrho_j(R)&=\int_{B^*_p(R)}Gh_jd{\rm vol}-\phi_k(R)\int_{B_p(R)}h_jd{\rm vol},
\end{split}\end{equation*}
 where  $B^*_p(R)=B_p(R)\backslash\{p\}$, the function $G(x)\!:=\phi_k(|px|)$  and $\phi_k(r)$ is the real value function such that $\phi\circ dist_o$ is the Green function on
 $\mathbb M^n_k$ with singular
point $o$. That is, if $n\gs3$,
 \begin{equation*}
\phi_k(r)=\frac{1}{(n-2)\cdot\omega_{n-1}}\int_r^{\infty}s^{1-n}_k(t)dt,
\end{equation*}  and
 \begin{equation*}s_k(t)=\begin{cases}\sin(\sqrt kt)/\sqrt k& \quad k>0\\t&\quad k=0\\ \sinh(\sqrt{- k}t)/\sqrt{- k}& \quad k<0.\end{cases}
\end{equation*}
Here, $\omega_{n-1}$ is the volume of $(n-1)$-sphere $\mathbb S^{n-1}$ with standard metric. If $n=2$, the function $\phi_k$ can be given similarly.

Letting $j\to\infty$ and applying the monotone convergence theorem, we get
\begin{equation}\label{eq3.2}
 \frac{1}{H^{n-1}(\partial B_o(R)\subset T^k_p)}\int_{\partial B_p(R)}fd\rv-f(p)\ls (n-2)\cdot\frac{\omega_{n-1}}{\rv(\Sigma_p)}\cdot\varrho(R),
\end{equation}
where
\begin{equation*}\begin{split}
\varrho(R)&=\int_{B^*_p(R)}Ghd{\rm vol}-\phi_k(R)\int_{B_p(R)}hd{\rm vol}.
\end{split}\end{equation*}
Letting $p$ be a Lebesgue point of $h$, it is calculated in \cite{zz12} that (see from line 6 to line 14 on page 470 of \cite{zz12},)
\begin{equation*}
\varrho(R)=\frac{\rv(\Sigma_p)}{2n(n-2)\omega_{n-1}}h(p)\cdot R^2+o(R^2).
\end{equation*}
Therefore, the desired result follows from this and equation \eqref{eq3.2}.
\end{proof}

\subsection{Harmonicity via Perron's method}$\  $

The Perron's method has been studied in \cite{bbk03,km02} in the setting of measure metric spaces.
We follow Kinnunen-Martio\footnote{Kinnunen-Martio works in the setting of metric measure spaces, which supported a doubling measure and a Poincar\'e inequality. These conditions are satisfied by  Alexandrov space with CBB, see \cite{kms01,zz10-2}.}, Section 7 of \cite{km02}, to defined the super-harmonicity.
\begin{defn}
Let $\Omega$ be an open subset of an Alexandrov space. A function $f:\Omega\to(-\infty,\infty]$ is called \emph{super-harmonic} on $\Omega$
 if it satisfies the following properties:\\
\indent (i) $f$ is lower semi-continuous in $\Omega$;\\
\indent (ii) $f$ is not identically $\infty$ in any component of $\Omega$;\\
\indent (iii) for every domain $\Omega'\subset\subset\Omega$ the following comparison principle holds:
 if $v\in  C(\overline{\Omega'})\cap W^{1,2}(\Omega')$ and $v\ls f$ on $\partial \Omega'$, then $h(v)\ls f$ in $\Omega'$. Here $h(v)$ is the (unique) solution of
 the equation $\mathscr L_{h(v)}=0$ in $\Omega$ with
$v-h(v)\in W^{1,2}_0(\Omega')$.

A function $f$ is \emph{sub-harmonic} on $\Omega$, if $-f$ is super-harmonic on $\Omega$.
\end{defn}

For our purpose in this paper, we will focus on the case where $\Omega$ is a bounded domain and
the function $f\in C(\Omega)\cap W^{1,2}_{\rm loc}(\Omega)$. Therefore, in this case, we can simply replace the definition of super-harmonicity as follows.\\[5pt]
\emph{Definition} $3.3'$:
Let $\Omega$ be a bounded domain of an Alexandrov space.
A function $f\in C(\Omega)\cap W^{1,2}_{\rm loc}(\Omega)$ is called \emph{super-harmonic} on $\Omega$ if the following comparison principle holds:\\
\indent (iii$'$)\indent for every domain $\Omega'\subset\subset\Omega$, we have $h(f)\ls f$ in $\Omega'$.\\[5pt]
Indeed, if $f\in C(\Omega)\cap W^{1,2}_{\rm loc}(\Omega)$, then  $f\in C(\overline{\Omega'})\cap W^{1,2}(\Omega')$ for any domain $\Omega'\subset\subset\Omega$.
Hence, the  the condition (iii) implies (iii$'$). The inverse follows from Maximum Principle. Indeed, given any domain $\Omega'\subset\subset\Omega$ and any $v\in  C(\overline{\Omega'})\cap W^{1,2}(\Omega')$ with $v\ls f$ on $\partial \Omega'$, Maximum Principle implies that $h(v)\ls h(f)$ in $\Omega'$.
 Consequently, the condition
(iii$'$) implies (iii).
\begin{lem}[Kinnunen-Martio \cite{km02}]\label{lem3.4}
 Let $\Omega$ be a bounded domain of an Alexandrov space. Assume that
 $f\in W_{\rm loc}^{1,2}(\Omega)\cap C(\Omega)$.
 Then the following properties are equivalent to each other:\\
\indent (i) $\ f$ is a super-solution of $\mathscr L_f=0$ on $\Omega$;\\
\indent (ii) $\ f$ is a super-harmonic function in the Definition $3.3'$.
\end{lem}
\begin{proof}
Let $f\in W_{\rm loc}^{1,2}(\Omega)$. The function $f$ is a super-solution of $\mathscr L_f=0$ on $\Omega$ if and only if
 it is a superminimizer in $\Omega$, defined by Kinnunen-Martio on Page 865 of \cite{km02}.

 Now the equivalence between (i) and (ii) follows from the Corollary 7.6 and Corollary 7.9 in \cite{km02}.
\end{proof}
It is easy to extend the Lemma \ref{lem3.4} to Poisson equations.
\begin{cor}\label{cor3.5}
Let $\Omega$ be a bounded domain of an Alexandrov space. Assume that
 $f\in W_{\rm loc}^{1,2}(\Omega)\cap C(\Omega)$ and $g\in L^\infty(\Omega)$.
 Then the following properties are equivalent to each other:\\
\indent (i) $\ f$ is a super-solution of $\mathscr L_f=g\cdot\rv$ on $\Omega$;\\
\indent (ii) $\ f$ satisfies the following comparison principle: for each domain $\Omega'\subset\subset\Omega$, we have
$v\ls f$ in $\Omega'$, where $v\in W^{1,2}(\Omega')$ is the (unique) solution of
$$\mathscr L_v=g\cdot\rv \quad {\rm with} \quad v-f\in W^{1,2}_0(\Omega').$$
\end{cor}
\begin{proof}
Let $w$ be a weak solution of $\mathscr L_w=g\cdot\rv$ on $\Omega$ (in the sense of distribution). Then, by Lemma \ref{lem3.1}, we have $w\in C(\Omega)\cap W^{1,2}_{\rm loc}(\Omega)$.
We denote
 $$\tilde f:=f-w\in C(\Omega)\cap W^{1,2}_{\rm loc}(\Omega).$$

Obviously, the property (i) is equivalent to that $\tilde f$ is a super-solution of $\mathscr L_{\tilde f}=0$ on $\Omega$.
On the other hand, taking any domain $\Omega'\subset\subset\Omega$ and letting $v\in W^{1,2}(\Omega')$ is the (unique) solution of
$\mathscr L_v=g\cdot\rv$ with $ v-f\in W^{1,2}_0(\Omega'),$ we have
 $$\mathscr L_{v-w}=0\quad {\rm with}\quad   (v-w)-\tilde f\in W^{1,2}_0(\Omega').$$
That is, $h(\tilde f)=v-w.$ Hence, the property (ii) is equivalent to that $\tilde f$ is a super-harmonic function in the Definition $3.3'$.
Now the Lemma is a consequence of Lemma \ref{lem3.4}.
\end{proof}

\section{Energy functional}

From now on, in this section, we always denote by $\Omega$ a bounded open domain of an $n$-dimensional Alexandrov space $(M,|\cdot,\cdot|)$ with curvature $\gs k$ for some $k\ls 0$, and denote by $(Y,d_Y)$ a complete metric space.

 Fix any $p\in[1,\infty)$.
A Borel measurable map $u:\ \Omega\to Y$ is said to be in the space $L^p(\Omega,Y)$ if it has separable range and, for some (hence, for all) $P\in Y$,
$$\int_\Omega d^p_Y\big(u(x),P\big)d\rv(x)<\infty.$$
We equip $L^p(\Omega,Y)$ with a distance given by
$$d^p_{L^p}(u,v):=\int_\Omega d^p_Y\big(u(x),v(x)\big)d\rv(x), \qquad \forall\ u,v\in L^p(\Omega,Y).$$

Denote by $C_0(\Omega)$ the set of continuous functions compactly supported on $\Omega$. Given $ p\in [1,\infty)$ and a map $u\in L^p(\Omega,Y),$
 for each $\epsilon>0$, the \emph{approximating energy} $E^u_{p,\epsilon}$ is defined as a functional on $C_0(\Omega)$:
$$E^u_{p,\epsilon}(\phi):=\int_\Omega\phi(x) e^u_{p,\epsilon}(x)d\rv(x)$$
where $\phi\in C_0(\Omega)$ and $e^u_{p,\epsilon}$ is \emph{approximating energy density} defined by
$$e^u_{p,\epsilon}(x):=\frac{n+p}{c_{n,p}\cdot\epsilon^n}\int_{B_x(\epsilon)\cap\Omega}\frac{d^p_Y\big(u(x),u(y)\big)}{\epsilon^p}d\rv(y),$$
where the constant $c_{n,p}=\int_{\mathbb S^{n-1}}|x^1|^p\sigma(dx),$ and $\sigma$ is the
canonical Riemannian volume on $\mathbb S^{n-1}$. In particular, $c_{n,2}=\omega_{n-1}/n$,
 where $\omega_{n-1}$ is the volume of $(n-1)$-sphere $\mathbb S^{n-1}$ with standard metric.

Let $ p\in [1,\infty)$ and a  $u\in L^p(\Omega,Y).$ Given any  $\phi\in C_0(\Omega)$, it is easy to check that, for any sufficiently small
$\epsilon>0 $ (for example, $10\epsilon<d(\partial \Omega,{\rm supp}\phi)$), the approximating energy  $E^u_{p,\epsilon}(\phi)$ coincides, up to a constant, with the
 one defined by K. Kuwae and T. Shioya in \cite{ks-sob03}\footnote{Indeed, K. Kuwae and T. Shioya in \cite{ks-sob03} defined it on more general  metric spaces satisfying a SMCPBG condition.
 And they proved  that Alexandrov spaces satisfy such a condition (see Theorem 2.1 of \cite{ks-sob03}).}, that is,
$$\widetilde{E}^u_{p,\epsilon}(\phi)\!:=\!\frac{n}{2\omega_{n-1}\epsilon^n}\!\!\int_\Omega\!\phi(x)\!\!\int_{B_x(\epsilon)\cap\Omega}\!\!\!
\frac{d^p_Y(u(x),u(y))}{\epsilon^p}\!\cdot\! I_{Q(\Omega)}(x,y)d\rv(y)d\rv(x),$$
where
\begin{center}$Q(\Omega):=\big\{(x,y)\in \Omega\times\Omega:\ |xy|<|\gamma_{xy},\partial\Omega|,\ \ \forall {\rm geodesic} \
 \gamma_{xy} \ {\rm\ from}\ x\ {\rm to}\ y\big\},$\end{center}
 and  $I_{Q(\Omega)}(x,y)$ is the indicator function of the set $Q(\Omega)$.
It is proved in \cite{ks-sob03} that, for each $\phi\in C_0(\Omega)$, the limit
$$E^u_{p}(\phi):=\lim_{\epsilon\to0^+}E^u_{p,\epsilon}(\phi)$$
exists. The limit functional $E^u_{p}$ is called the \emph{energy functional}.

 Now the $p^{th}$ \emph{order Sobolev space} from $\Omega$ into $Y$ is defined by
$$W^{1,p}(\Omega,Y):=\mathscr D(E^u_p):=\big\{u\in L^p(\Omega,Y)|\ \sup_{0\ls \phi\ls1,\ \phi\in C_0(\Omega)}E^u_p(\phi)<\infty\big\},$$
and $p^{th}$ \emph{order energy} of $u$ is
$$E^u_p:=\sup_{0\ls \phi\ls1,\ \phi\in C_0(\Omega)}E^u_p(\phi).$$

In the following proposition, we will collect some  results in \cite{ks-sob03}.
\begin{prop}[Kuwae--Shioya \cite{ks-sob03}]\label{prop4.1}
Let $1<p<\infty$ and $u\in W^{1,p}(\Omega, Y)$. Then the following assertions (1)--(5) hold.\\
$(1)$ {\rm(Contraction property, Lemma 3.3 in \cite{ks-sob03})}\indent Consider another complete metric spaces $(Z,d_Z)$ and a Lipschitz map $\psi: Y\to Z$, we have
$\psi\circ u\in W^{1,p}(\Omega,Z)$ and
$$E_p^{\psi\circ u}(\phi)\ls {\rm \bf{Lip}}^p(\psi) E_p^u(\phi)$$
for any $0\ls\phi\in C_0(\Omega)$, where
\begin{center}${\rm \bf{Lip}}(\psi):=\sup_{y,y'\in Y,\ y\not=y'}\frac{d_Z(\psi(y),\psi(y'))}{d_Y(y,y')}.$\end{center}
In particular, for any point $Q\in Y$, we have $d_Y\big(Q,u(\cdot)\big)\in W^{1,p}(\Omega, \mathbb R)$ and
 $$E_p^{d_Y(Q,u(\cdot))}(\phi)\ls E_p^u(\phi) $$
  for any $0\ls\phi\in C_0(\Omega)$.

\noindent $(2)$ {\rm(Lower semi-continuity, Theorem 3.2 in \cite{ks-sob03})}\indent For any sequence $u_j\to u$ in $L^p(\Omega, Y)$ as $j\to\infty$, we have
$$E^u_p(\phi)\ls \liminf_{j\to\infty} E^{u_j}_p(\phi)$$
for any $0\ls \phi\in C_0(\Omega).$ \\
$(3)$ {\rm(Energy measure, Theorem 4.1 and Proposition 4.1 in \cite{ks-sob03})}\indent There exists a finite Borel measure, denoted by $E^u_p$ again, on $\Omega$, is called \emph{energy measure} of $u$,
 such that for any $0\ls\phi\in C_0(\Omega)$
$$E^u_p(\phi)= \int_\Omega\phi(x)dE^u_p(x).$$
Furthermore, the measure is strongly local. That is, for any nonempty open subset $O\subset \Omega$, we have $u|_O\in W^{1,p}(O,Y)$,
 and moreover, if $u$ is a constant map almost everywhere on $O$, then $E^u_p(O)=0.$\\
\noindent $(4)$ {\rm(Weak Poincar\'e inequality, Theorem 4.2(ii) in \cite{ks-sob03})}\indent For any open set $O=B_q(R)$ with $B_q(6R)\subset\subset\Omega$,
 there exists postive constant $C=C(n,k,R)$ such that
the following holds: for any $z\in O$ and any $0<r<R/2$, we have
$$\int_{B_z(r)}\int_{B_z(r)}d_Y^p\big(u(x),u(y)\big)d\rv(x)d\rv(y)\ls Cr^{n+2}\cdot\int_{B_z(6r)}dE^u_p(x),$$
where  the constant $C$ given on Page 61 of \cite{ks-sob03}  depends only on the constants $R$, $\vartheta,$ and  $\Theta$ in the Definition 2.1 for WMCPBG condition in \cite{ks-sob03}. In particular, for the case of Alexandrov spaces as shown in the proof of Theorem 2.1 in \cite{ks-sob03}, one can choose $R>0$ arbitrarily, $\vartheta=1$ and $\Theta=\sup_{0<r<R}\frac{\rv(B_o(r)\subset \mathbb M^n_k)}{\rv(B_o(r)\subset \mathbb R^n)}=C(n,k,R)$.

\noindent $(5)$ {\rm(Equivalence for $Y=\mathbb R$, Theorem 6.2 in \cite{ks-sob03})}\indent If $Y=\mathbb R$,  the above Sobolev space $W^{1,p}(\Omega,\mathbb R)$
is equivalent to the Sobolev space $W^{1,p}(\Omega)$ given in previous Section 3. To be precise:
For any $u\in W^{1,p}(\Omega,\mathbb R)$, the energy measure of $u$ is absolutely continuous with respect to $\rv$ and
$$\frac{dE^u_p}{d\rv}(x)= |\nabla u(x)|^p.$$
\end{prop}

\begin{rem}It is not clear whether the energy measure of $u\in W^{1,p}(\Omega,Y)$ is absolutely continuous with respect to the Hausdorff measure $\rv$ on $\Omega$.
 If $\Omega$ is a domain in a Lipschitz Riemannian manifold, the absolute continuity has been proved by
  G. Gregori in \cite{gre98} (see also Korevaar-Schoen \cite{ks93} for the case where $\Omega$ is a domain in a $C^2$ Riemannian manifold).
\end{rem}

Let $p>1$ and let $u$ be a map with $u\in W^{1,p}(\Omega,Y)$ with energy measure $E^u_p$.
   Fix any sufficiently small positive number $\delta$ with $0<\delta<\delta_{n,k}$, with $\delta_{n,k}$ as in Fact \ref{fact} in Section 2.3. Then the set
$$\Omega^\delta:=\Omega\cap M^\delta:=\big\{x\in \Omega:\ \rv(\Sigma_x)>(1-\delta)\rv(\mathbb S^{n-1})\big\}$$
 is an open subset in $\Omega$ and forms a Lipschitz manifold. Since the singular set of $M$ has (Hausdorff) codimension at least two (\cite{bgp92}), we have   $\rv(\Omega\backslash\Omega^\delta)=0.$
 Hence, by the strongly local property of the measure $E^u_p$, we have $u\in W^{1,p}(\Omega^\delta, Y)$ and its energy measure is $E^u_p|_{\Omega^\delta}$.
Since $\Omega^\delta$ is a Lipschitz manifold, according to Gregori in \cite{gre98},
 we obtain that the energy measure  $E^u_p|_{\Omega^\delta}$
 is absolutely continuous with respect to $\rv$. Denote its density by $|\nabla u|_p$. (We write $|\nabla u|_p$ instead of $|\nabla u|^p$ because
  the quantity $p$ does not in general behave like power, see \cite{ks93}.) Considering the Lebesgue decomposition of $E^u_p$ with respect to $\rv$ on $\Omega$,
$$E^u_p=|\nabla u|_p\cdot\rv+(E^u_p)^s,$$
we have that the support of the singular part $(E^u_p)^s$ is contained in   $\Omega\backslash\Omega^\delta.$

Clearly, the energy density $|\nabla u|_p$ is the weak limit (limit as measures) of the approximating energy density $e^u_{p,\epsilon}$ as $\epsilon\to0$ on $\Omega^\delta$.
We now show that $e^u_{p,\epsilon}$ converges  \emph{almost} to $|\nabla u|_p$
  in $L^1_{\rm loc}(\Omega)$ in the following sense.

\begin{lem}\label{lem4.3}
Let $p>1$ and $u\in W^{1,p}(\Omega,Y)$. Fix any sufficiently small $\delta>0$ with $0<\delta<\delta_{n,k}$, with $\delta_{n,k}$ as in
 Fact \ref{fact} in Section 2.3. Then, for any open subset $B\subset\subset\Omega^\delta$,
 there exists a constant
$\bar\epsilon=\bar\epsilon(\delta,B)$ such that,
 for any $0<\epsilon<\bar\epsilon(\delta,B)$,
 we have
 $$\int_B \big|e^u_{p,\epsilon}(x)-|\nabla u|_p(x)\big|d\rv(x)\ls \bar\kappa(\delta),$$
 where $\bar\kappa(\delta)$ is a positive function (depending only on $\delta$) with $\lim_{\delta\to0}\bar\kappa(\delta)=0$.
\end{lem}
\begin{proof}
Fix any sufficiently small $\delta>0$ and any open set $B$ as in the assumption.
 By applying Lemma \ref{appo}, there exists
 some neighborhood $U_\delta\supset \overline B$ and a
 smooth Riemannian metric $g_\delta$ on $U_\delta$ such that the distance $d_\delta$ on $U_\delta$ induced from $g_\delta$ satisfies
$$\bigg|\frac{d_\delta(x,y)}{|xy|}-1\bigg|\ls \kappa_1(\delta) \quad {\rm for\ any}\ \ x,y\in U_\delta,\ x\not=y,$$
where $\kappa_1(\delta)$ is a positive function (depending only on $\delta$) with $\lim_{\delta\to0}\kappa_1(\delta)=0.$
This implies that
\begin{equation}\label{eq4.1}
B^\delta_x\big( r\cdot(1-\kappa_1(\delta))\big)\subset B_x(r)\subset B^\delta_x\big(r\cdot(1+\kappa_1(\delta)\big)
\end{equation}
for any $x\in U_\delta$ and $r>0$ with the ball $ B^\delta_x\big((1+\kappa_1(\delta)r\big)\subset U_\delta$ and
\begin{equation}\label{eq4.2}
1-\kappa^n_1(\delta)\ls \frac{d\rv_\delta(x)}{d\rv(x)}\ls 1+\kappa^n_1(\delta)\qquad \forall \ x\in U_\delta,
\end{equation}
where $B_x^\delta(r)$ is the geodesic balls with center $x$ and radius $r$ with respect to the metric $g_\delta$,
 and $\rv_{\delta}$ is the $n$-dimensional Riemannian volume on $U_\delta$ induced from metric $g_\delta$.\\

\noindent (i). \emph{Uniformly approximated by smooth metric $g_\delta$.}

For any $\epsilon>0$, we write the energy density and approximating energy density of $u$ by $|\nabla u|_{p,g_\delta}$ and $e^u_{p,\epsilon,g_\delta}$
on $(U_\delta,g_\delta)$ with respect to the smooth Riemannian metric $g_\delta.$

\begin{slem}
 We have, for any $x\in U_\delta$ and any $\epsilon>0$ with $B_x(10\epsilon)\subset U_\delta$,
\begin{equation}\label{eq4.3}
\begin{split}
\big|e^u_{p,\epsilon}(x)-e^u_{p,\epsilon,g_\delta}(x)\big|
\ls &
 \kappa_4(\delta) \cdot e^u_{p,2\epsilon}(x)+\big|e^u_{p,\epsilon(1+\kappa_1(\delta)),g_\delta}(x)-e^u_{p,\epsilon,g_\delta}(x)\big| \\
&+\big|e^u_{p,\epsilon,g_\delta}(x)-e^u_{p,\epsilon(1-\kappa_1(\delta)),g_\delta}(x)\big|,
\end{split}
\end{equation}
where $\kappa_4(\delta)$ is a positive function (depending only on $\delta$) with $\lim_{\delta\to0}\kappa_4(\delta)=0.$
\end{slem}
\begin{proof}
For each $x\in U_\delta$ and $\epsilon>0$ with $B_x(10\epsilon)\subset U_\delta$, by applying equations \eqref{eq4.1}--\eqref{eq4.2}  and setting
$$f(y):=2(n+p)\cdot c^{-1}_{n,p}\cdot d^p_Y\big(u(x),u(y)\big) ,$$
we have, from the definition of approximating energy density,
\begin{equation}\begin{split}\label{eq4.4}
e^u_{p,\epsilon}(x)&=\int_{B_x(\epsilon)\cap\Omega} \frac{f}{\epsilon^{n+p}}d\rv(y)\\
&\ls \big(1-\kappa^n_1(\delta)\big)^{-1}\cdot\int_{B^{\delta}_x\big(\epsilon\cdot(1+\kappa_1(\delta))\big)}\frac{f}{\epsilon^{n+p}}d\rv_{\delta}(y)\\
&= \big(1-\kappa^n_1(\delta)\big)^{-1}\cdot(1+\kappa_1(\delta))^{n+p}\cdot e^u_{p,\epsilon\cdot(1+\kappa_1(\delta)),g_\delta}(x)\\
&:=\big(1+\kappa_2(\delta)\big)\cdot e^u_{p,\epsilon\cdot(1+\kappa_1(\delta)),g_\delta}(x).
\end{split}
\end{equation}
Similarly, we have
\begin{equation}\label{eq4.5}
\begin{split}
e^u_{p,\epsilon}(x)&\gs \big(1+\kappa^n_1(\delta)\big)^{-1}\cdot(1-\kappa_1(\delta))^{n+p}\cdot e^u_{p,\epsilon\cdot(1-\kappa_1(\delta)),g_\delta}(x)\\
&:=\big(1-\kappa_3(\delta)\big)\cdot e^u_{p,\epsilon\cdot(1-\kappa_1(\delta)),g_\delta}(x).
\end{split}
 \end{equation}
Thus
\begin{equation}\label{eq4.6}
\begin{split}
\big|e^u_{p,\epsilon}&(x)-e^u_{p,\epsilon,g_\delta}(x)\big|\\
\ls& \max\left\{
\begin{array}{c}
 \kappa_2(\delta) \cdot e^u_{p,\epsilon(1+\kappa_1(\delta)),g_\delta}(x)+|e^u_{p,\epsilon(1+\kappa_1(\delta)),g_\delta}(x)-e^u_{p,\epsilon,g_\delta}(x)|, \\
 \kappa_3(\delta) \cdot e^u_{p,\epsilon(1-\kappa_1(\delta)),g_\delta}(x)+|e^u_{p,\epsilon,g_\delta}(x)-e^u_{p,\epsilon(1-\kappa_1(\delta)),g_\delta}(x)|
\end{array}\right\}.
\end{split}
\end{equation}
Without loss of the generality, we can assume that $\kappa_1(\delta)<1/3$ for any sufficiently small $\delta$.
Then, from \eqref{eq4.5} and the definition of the approximating energy density,
\begin{equation*}\begin{split}
e^u_{p,\epsilon(1+\kappa_1(\delta)),g_\delta}(x)&\ls \big(1-\kappa_3(\delta)\big)^{-1}\cdot e^u_{p,\epsilon(\frac{1+\kappa_1(\delta)}{1-\kappa_1(\delta)})}(x)\\
&\ls \big(1-\kappa_3(\delta)\big)^{-1}\cdot \Big[2\cdot\frac{1-\kappa_1(\delta)}{1+\kappa_1(\delta)}\Big]^{n+p}\cdot e^u_{p,2\epsilon}(x)\\
&\ls \big(1-\kappa_3(\delta)\big)^{-1}\cdot 2^{n+p}\cdot e^u_{p,2\epsilon}(x)\\
\end{split}
\end{equation*}
and
\begin{equation*}\begin{split}
e^u_{p,\epsilon(1-\kappa_1(\delta)),g_\delta}(x)&\ls \big(1-\kappa_3(\delta)\big)^{-1}\cdot e^u_{p,\epsilon}(x)\\
&\ls \big(1-\kappa_3(\delta)\big)^{-1}\cdot 2^{n+p}\cdot e^u_{p,2\epsilon}(x).\ \ \qquad\qquad
\end{split}
\end{equation*}
By substituting the above two inequalities in equation \eqref{eq4.6}, we obtain
\begin{equation*}
\begin{split}
\big|e^u_{p,\epsilon}(x)-e^u_{p,\epsilon,g_\delta}(x)\big|
\ls &
 \kappa_4(\delta) \cdot e^u_{p,2\epsilon}(x)+|e^u_{p,\epsilon(1+\kappa_1(\delta)),g_\delta}(x)-e^u_{p,\epsilon,g_\delta}(x)| \\
&+|e^u_{p,\epsilon,g_\delta}(x)-e^u_{p,\epsilon(1-\kappa_1(\delta)),g_\delta}(x)|,
\end{split}
\end{equation*}
where the function $\kappa_4(\delta):=\big(1-\kappa_3(\delta)\big)^{-1}\cdot 2^{n+p}\cdot\max\{\kappa_2(\delta),\kappa_3(\delta)\}.$
The proof of the Sublemma is finished.
\end{proof}

\noindent (ii). \emph{Uniformly estimate for integral}
\begin{center}
$\int_B\big|e^u_{p,\epsilon}(x)-e^u_{p,\epsilon,g_\delta}(x)\big|d\rv(x)$.
\end{center}

To deal with this integral, we need to estimate integrals of the right hand side in equation \eqref{eq4.3}.

 Noting that the metric $g_\delta$ is smooth on $U_\delta$, The following assertion is summarized in \cite{gre98}, and essentially proved by \cite{ser95}.
Please see the paragraph
between Lemma 1 and Lemma 2 on Page 3 of \cite{gre98}.

\begin{fact}\label{fact4.5}
The approximating energy densities
\begin{equation*}
\lim_{\epsilon\to0}e^u_{p,\epsilon,g_\delta}=|\nabla u|_{p,g_\delta}\qquad {\rm in}\quad L^1_{\rm loc}(U_\delta,g_\delta).
\end{equation*}
\end{fact}

Now let us continue the proof of this Lemma.

Since the set $B\subset\subset U_\delta $, from the above Fact \ref{fact4.5}, there exists a constant $\epsilon_1=\epsilon_1(\delta,B)$
 such that for any $0<\epsilon<\epsilon_1$, we have
\begin{equation*}
\int_B\big||\nabla u|_{p,g_\delta}(x)-e^u_{p,\epsilon,g_\delta}(x)\big|d\rv_{\delta}\ls \delta.
 \end{equation*}
Hence, by using equation \eqref{eq4.2},
\begin{equation}\label{eq4.7}
\int_B\big||\nabla u|_{p,g_\delta}(x)-e^u_{p,\epsilon,g_\delta}(x)\big|d\rv\ls \delta\cdot \big(1+\kappa_1^n(\delta)\big):=\kappa_5(\delta).
 \end{equation}
Triangle inequality concludes that, for any number $\epsilon$ with $0<\epsilon<\frac{\epsilon_1}{1+\kappa_1(\delta)},$
\begin{equation}\label{eq4.8}
\int_B\big|e^u_{p,\epsilon(1+\kappa_1(\delta)),g_\delta}(x)-e^u_{p,\epsilon,g_\delta}(x)\big|d\rv(x) \ls 2\kappa_5(\delta)
\end{equation}
and
\begin{equation}\label{eq4.9}
\int_B\big|e^u_{p,\epsilon,g_\delta}(x)-e^u_{p,\epsilon(1-\kappa_1(\delta)),g_\delta}(x)\big|d\rv(x) \ls 2\kappa_5(\delta).
\end{equation}

By using Lemma 3 in \cite{gre98} (more precisely, the equation (35) in \cite{gre98}),
 for any $\phi\in C_0(U_\delta)$ and any $ \gamma>0$, there exists a constant  $\epsilon_2=\epsilon_2(\gamma,\phi)$ such that the following estimate holds for any
$0<\epsilon<\epsilon_2$:
$$ E^u_{p,\epsilon}(\phi)\ls E^u_p(\phi)+C\gamma,$$
where $C$ is a constant independent of $\gamma$ and $\epsilon$.
Now, since $B\subset\subset U_\delta$, there exists $\varphi\in C_0(U_\delta)\ (\subset C_0(\Omega))$ with $\varphi|_B=1$ and $0\ls\varphi\ls 1$ on $U_\delta$.
Fix such a function $\varphi$ and a constant $\gamma_1>0$ with $C\gamma_1\ls1$.
Then for any $0<\epsilon<\epsilon_3:=\min\{\epsilon_2(\gamma_1,\varphi)/2,{\rm dist}({\rm supp \varphi},\partial U_\delta)/10\}$, we have
\begin{equation}\label{eq4.10}
\begin{split}
\int_B e^u_{p,2\epsilon}(x)d\rv&\ls\int_{U_\delta}\varphi(x)e^u_{p,2\epsilon}(x)d\rv\ls E^u_{p,2\epsilon}(\varphi)
\ls E^u_{p}(\varphi)+1\\ &\ls E^u_p(\Omega)+1.
\end{split}
\end{equation}

By integrating equation \eqref{eq4.3} on $B$ with respect to $\rv$ and combining with equation \eqref{eq4.8}--\eqref{eq4.10}, we obtain that, for any $0<\epsilon<\min\{\epsilon_3,\epsilon_1/\big(1+\kappa_1(\delta)\big)\}$,
\begin{equation}\label{eq4.11}
\int_B\big|e^u_{p,\epsilon}(x)-e^u_{p,\epsilon,g_\delta}(x)\big|d\rv(x)\ls \kappa_6(\delta),
\end{equation}
where the positive function $\kappa_6(\delta)=\kappa_4(\delta)\cdot  \big(E^u_p(\Omega)+1\big)+4\kappa_5(\delta).$\\

\noindent(iii). \emph{Uniformly estimate for the desired integral}
\begin{center}
$\int_B \big|e^u_{p,\epsilon}(x)-|\nabla u|_p(x)\big|d\rv(x)$.
\end{center}

According to equation \eqref{eq4.7} and \eqref{eq4.11}, we have, for any sufficiently small $\epsilon>0$,
\begin{equation}\label{eq4.12}
\begin{split}
\int_B\big|&e^u_{p,\epsilon}(x)-|\nabla u|_p(x)\big|d\rv(x)\\
\ls&\! \int_B\big|e^u_{p,\epsilon}(x)\!-\!e^u_{p,\epsilon,g_\delta}(x)\big|d\rv(x)+\!\!\int_B\big|e^u_{p,\epsilon,g_\delta}(x)\!-\!|\nabla u|_{p,g_\delta}(x)\big|d\rv(x)\\
&\quad+\int_B\big||\nabla u|_{p,g_\delta}(x)-|\nabla u|_p(x)\big|d\rv(x)\\
\ls& \ \kappa_6(\delta)+\kappa_5(\delta)+\int_B\big||\nabla u|_{p,g_\delta}(x)-|\nabla u|_p(x)\big|d\rv(x).
\end{split}\end{equation}
To estimate the desired integral, we need only to control the last term in above equation.
It is implicated by the combination of the uniformly estimate \eqref{eq4.11} and Fact \ref{fact4.5}. We give the argument in detail as follows.

   By equation \eqref{eq4.2},  for any $\phi\in C_0(U_\delta)$ we have
\begin{equation*}
\begin{split}
\Big|\int_{U_\delta}\phi (x)&\cdot\big( e^u_{p,\epsilon,g_\delta}-|\nabla u|_{p,g_\delta}\big)d\rv(x)\Big|\\
&\ls \max|\phi|\cdot\int_{W}\big| e^u_{p,\epsilon,g_\delta}-|\nabla u|_{p,g_\delta}\big|d\rv(x)\\
&\ls \max|\phi|\cdot\int_{W}\big| e^u_{p,\epsilon,g_\delta}-|\nabla u|_{p,g_\delta}\big|d\rv_\delta(x)\cdot(1+\kappa_1^n(\delta)),
\end{split}
\end{equation*}
where $W$ is the support set of $\phi.$   By taking limit as $\epsilon\to0$, and using Fact \ref{fact4.5}, we have,
 weakly converging as measure
 $$ e^u_{p,\epsilon,g_\delta}\cdot\rv\overset{w}{\rightharpoonup}|\nabla u|_{p,g_\delta}\cdot\rv.$$
Combining with the fact $ e^u_{p,\epsilon}\cdot\rv\overset{w}{\rightharpoonup}|\nabla u|_{p}\cdot\rv,$ we have
$$\big(e^u_{p,\epsilon}-e^u_{p,\epsilon,g_\delta}\big)\cdot\rv\overset{w}{\rightharpoonup}\big(|\nabla u|_{p}-|\nabla u|_{p,g_\delta}\big)\cdot\rv.$$
By applying estimate of \eqref{eq4.11} and according the lower semi-continuity of $L^1$-norm  with respect to weakly converging of measure,
 we have
$$\int_B\big||\nabla u|_{p}-|\nabla u|_{p,g_\delta}\big|d\rv
\ls\liminf_{\epsilon\to0}\int_B\big|e^u_{p,\epsilon}-e^u_{p,\epsilon,g_\delta}\big|d\rv\ls\kappa_6(\delta).$$
By substituting the estimate into equation \eqref{eq4.12}, we get
\begin{equation*}
\int_B\big|e^u_{p,\epsilon}(x)-|\nabla u|_p(x)\big|d\rv(x)\ls  \kappa_5(\delta)+2\kappa_6(\delta):=\bar\kappa(\delta).
\end{equation*}
This completes the proof of the lemma.
\end{proof}
\begin{cor}\label{coro4.6}
Let $p>1$ and $u\in W^{1,p}(\Omega,Y)$. Then, for any sequence of number $\{\epsilon_j\}_{j=1}^\infty$ converging to $0$, there
 exists a subsequence $\{\varepsilon_j\}_j\subset\{\epsilon_j\}_j$ such that, for almost everywhere $x\in \Omega$,
$$\lim_{\varepsilon_j\to0} e^u_{p,\varepsilon_j}(x)=|\nabla u|_p(x).$$
\end{cor}
\begin{proof}
Take any sequence $\{\delta_j\}_j$ going to $0$, and let $\{B_j\}_j$ be a sequence of open sets such that, for each $j\in \mathbb N$,
  $$B_j\subset\subset\Omega^{\delta_j} \quad{\rm and}\quad\rv(\Omega^{\delta_j}\backslash B_j)\ls \delta_j.$$
Since the sequence $\{\epsilon_j\}_j$ tends to $0$, we can choose a subsequence $\{\varepsilon_j\}_j$ of $\{\epsilon_j\}_j$ such that, for each $j\in\mathbb N$,
 $\varepsilon_j<\bar\epsilon(\delta_j,B_j)$, which is the constant given in Lemma \ref{lem4.3}. Hence, we have
$$\int_{B_j}\big|e^u_{p,\varepsilon_j}-|\nabla u|_p\big|d\rv\ls \bar\kappa(\delta_j),\qquad \forall\ j\in\mathbb N.$$
For each $j\in \mathbb N$, $\rv(\Omega\backslash\Omega^{\delta_j})=0$. So, the functions $e^u_{p,\varepsilon_j}$ is measurable on $\Omega$ for any $j\in \mathbb N$.
In the following, we will prove that the sequence
 $$  \{f_j:=e^u_{p,\varepsilon_j}\}_j$$
 converges to $f:=|\nabla u|_p$ in measure on
 $\Omega$. Namely, given any number $\lambda>0$, we will prove
 $$\lim_{j\to\infty}\rv\big\{x\in\Omega:\ |f_j(x)-f(x)|\gs \lambda\big\}=0.$$

Fix any $\lambda>0$, we consider the sets
$$A_j(\lambda):=\big\{x\in \Omega\backslash S_M:\ |f_j(x)-f(x)|\gs \lambda\big\}.$$
Noting that $ S_M$ has zero measure (indeed, it has Hausdorff codimension at least two \cite{bgp92}), we need only to show
  $$\lim_{j\to\infty}\rv\big(A_j(\lambda)\big)=0.$$
By Chebyshev inequality, we get
$$\lambda\cdot\rv\big(A_j(\lambda)\cap B_j\big)\ls\int_{A_j(\lambda)\cap B_j}|f_j-f|d\rv\ls \int_{B_j}|f_j-f|d\rv\ls \bar\kappa(\delta_j) $$
for any $j\in \mathbb N$. Thus, noting that $A_j(\lambda)\subset \Omega\backslash S_M \subset \Omega^{\delta_j}$ for each $j\in\mathbb N$, we have
\begin{equation*}
\begin{split}
\rv\big(A_j(\lambda)\big)&\ls\rv\big(A_j(\lambda)\cap B_j\big)+ \rv\big(A_j(\lambda)\backslash B_j\big)\ls\frac{\bar\kappa(\delta_j)}{\lambda}+\rv\big(\Omega^{\delta_j}\backslash B_j\big)\\ &\ls\frac{\bar\kappa(\delta_j)}{\lambda}+\delta_j
\end{split}
\end{equation*}
for any $j\in \mathbb N$. This implies that $\lim_{j\to\infty}\rv\big(A_j(\lambda)\big)=0$, and hence, that $\{f_j\}_j$ converges to $f$ in measure.

Lastly, by F. Riesz theorem, there exists a subsequence of $\{\varepsilon_j\}_j$, denoted by $\{\varepsilon_j\}_j$ again, such that
the sequence $\{e^u_{p,\varepsilon_j}\}_j$ converges to $|\nabla u|_p$ almost everywhere in $\Omega.$
\end{proof}
The above pointwise converging provides the following mean value property, which will be used later.
\begin{cor}\label{coro4.7}
Let $p>1$ and $u\in W^{1,p}(\Omega,Y)$. Then, for any sequence of number $\{\epsilon_j\}_{j=1}^\infty$ converging to $0$, there exists
 a subsequence $\{\varepsilon_j\}_j\subset\{\epsilon_j\}_j$ such that for almost everywhere $x_0\in \Omega$, we have the following mean value property:
\begin{equation}
\int_{B_{x_0}(\varepsilon_j)}d^p_Y\big(u(x_0),u(x)\big)d\rv(x)=\frac{c_{n,p}}{n+p}|\nabla u|_p(x_0)\cdot \varepsilon_j^{n+p}+o( \varepsilon_j^{n+p}).
\end{equation}
\end{cor}
\begin{proof}
According to the previous Corollary \ref{coro4.6},  there exists
 a subsequence $\{\varepsilon_j\}_j\subset\{\epsilon_j\}_j$ such that
 $$\lim_{\varepsilon_j\to0} e^u_{p,\varepsilon_j}(x_0)=|\nabla u|_p(x_0)\quad{\rm for\ almost\ all}\ x_0\in \Omega.$$
Fix such a point $x_0$. By the definition of approximating energy density, we get
$$\frac{n+p}{c_{n,p}\cdot\varepsilon_j^n}\int_{B_{x_0}(\varepsilon_j)}d^p_Y\big(u(x_0),u(x)\big)d\rv(x)=|\nabla u|_p(x_0)\cdot \varepsilon_j^p+o( \varepsilon_j^p).$$
The proof is finished. \end{proof}

\section{Pointwise Lipschitz constants}
Let $\Omega$ be a bounded domain of an Alexandrov space with curvature $\gs k$ for some $k\ls0$. In this section, we will established
 an estimate for pointwise Lipschitz constants of harmonic maps from $\Omega$ into a complete, non-positively curved metric space $(Y,d_Y)$.

 Let us first review the concept of metric spaces with (global) non-positive curvature in the sense of Alexandrov.

\subsection{NPC spaces}
\begin{defn}[see, for example, \cite{bh91}]
A geodesic space $(Y,d_Y)$ is said to have global \emph{non-positive curvature} in the sense of Alexandrov, denoted by $NPC$, if the
 following comparison property is to hold: Given any triangle $\triangle PQR\subset Y$ and point $S\in QR$ with
 $$d_Y(Q,S)=d_Y(R,S)= \frac{1}{2} d_Y(Q,R),$$
there exists a comparison triangle $\triangle \bar P\bar Q\bar R$ in Euclidean plane $\mathbb R^2$ and point $\bar S\in \bar Q\bar R$ with
 $$|\bar Q\bar S|=|\bar R\bar S|=\frac{1}{2}|\bar Q\bar R|$$
such that
$$d_Y(P,S)\ls |\bar P\bar S|.$$
 It is also called a $CAT(0)$ space.
\end{defn}
The following lemma is a special case of Corollary 2.1.3 in \cite{ks93}.
\begin{lem}\label{lem5.2}
Let $(Y,d_Y)$ be an NPC space. Take any ordered sequence $\{P,Q,R,S\}\subset Y$, and let point $Q_{m}$ be the mid-point of $QR$.
we denote the distance $d_Y(A,B)$ abbreviatedly by $d_{AB}.$
Then we have
\begin{equation}\label{eq5.1}
( d_{PS}-d_{QR})\cdot d_{QR}\gs (d^2_{PQ_{m}}-d^2_{PQ}-d^2_{Q_{m}Q})+(d^2_{SQ_{m}}-d^2_{SR}-d^2_{Q_{m}R}).
 \end{equation}
\end{lem}
\begin{proof}
Taking $t=1/2$ and $\alpha=1$ in Equation (2.1v) in Corollary 2.1.3 of \cite{ks93}, we get
$$ d^2_{PQ_{m}}+d^2_{SQ_{m}}\ls d^2_{PQ}+d^2_{RS}-\frac{1}{2}d^2_{QR}+d_{PS}\cdot d_{QR}.$$
Since
\begin{center} $d_{QR}=2d_{Q_{m}Q}=2d_{Q_{m}R}$, \end{center}
we have
$$ d_{PS}\cdot d_{QR}-d^2_{QR}\gs (d^2_{PQ_{m}}-d^2_{PQ}-d^2_{Q_{m}Q})+(d^2_{SQ_{m}}-d^2_{SR}-d^2_{Q_{m}R}).$$
This is equation (5.1).
\end{proof}

\subsection{Harmonic maps} $\  $

Let $\Omega$ be a bounded domain in an Alexandrov space $(M,|\cdot,\cdot|)$ and let $Y$ be an NPC space. Given any $\phi\in W^{1,2}(\Omega,Y)$, we set
$$W^{1,2}_\phi(\Omega,Y):=\big\{u\in W^{1,2}(\Omega,Y):\ d_Y\big(u(x),\phi(x)\big)\in W^{1,2}_0(\Omega,\mathbb R)\big\}.$$
Using the variation method in \cite{jost97,lin97}, (by the lower semi-continuity of energy,) there exists a unique $u\in W^{1,2}_\phi(\Omega,Y)$ which is minimizer of energy
$E_2^u$. That is, the energy $E_2^u:=E^u_2(\Omega)$ of $u$ satisfies
$$E_2^u=\inf_w\big\{E_2^w:\ w\in W^{1,2}_\phi(\Omega,Y)\big\}.$$
Such an energy minimizing map is called a \emph{harmonic map}.

\begin{lem}[Jost \cite{jost97}, Lin \cite{lin97}]\label{lem5.3}
 Let $\Omega$ be a bounded domain in an Alexandrov space $(M,|\cdot,\cdot|)$ and let $Y$ be an NPC space.
Suppose that $u$ is a harmonic map from $\Omega$ to $Y$. Then the following two properties are satisfied:\\
\indent (i) The map $u$ is locally H\"older continuous on $\Omega$;\\
\indent (ii) $($Lemma 5 in \cite{jost97}, see also Lemma 10.2 of \cite{ef01} for harmonic maps between Riemannian polyhedra$)$\ \ For any $P\in Y$, the function
$$f_P(x):=d_Y\big(u(x),P\big)\ \ \ \big(\in W^{1,2}(\Omega)\big)$$
 satisfies $f^2_P\in W^{1,2}_{\rm loc}(\Omega)$ and
$$\mathscr L_{f^2_P}\gs 2E^u_2\gs2 |\nabla u|_2\cdot\rv.\footnote{The assertion was proved essentially in Lemma 5 of \cite{jost97}, where J. Jost consider a different energy form $E$. Jost's argument was adapted in \cite{ef01} to prove the same assertion for energy minimizing maps from Riemannian polyhedra associated to the energy $E^u_2$ (given in the above Section 4). By checking the proof in Lemma 10.2 of \cite{ef01} word by word, the same proof also applies to our setting without changes.}$$
\end{lem}


According to this Lemma, we always assume that a harmonic map form $\Omega$ into an NPC space is continuous in $\Omega$.

\subsection{Estimates for pointwise Lipschitz constants} $\ $

Let $u$ be a harmonic map from  a bounded domain  $\Omega$ of an Alexandrov space $(M,|\cdot,\cdot|)$ to an NPC space $(Y,d_Y)$.
In this subsection, we will estimate the \emph{pointwise Lipchitz constant} of $u$, that is,
$${\rm Lip}u(x):=\limsup_{y\to x}\frac{d_Y\big(u(x),u(y)\big)}{|xy|}=\limsup_{r\to0}\sup_{|xy|\ls r}\frac{d_Y\big(u(x),u(y)\big)}{r}.$$
It is convenient to
consider the function $f:\Omega\times\Omega\to \mathbb R$ defined by
\begin{equation}\label{eq5.2}
f(x,y):=d_Y\big(u(x),u(y)\big),
\end{equation}
where $\Omega\times\Omega\subset M\times M$, which is equipped the product metric defined as
\begin{center}
 $ |(x,y),(z,w)|_{M\times M}^2:=|xz|^2+|yw|^2\qquad {\rm for\ any}\quad x,y,z,w\in M.$
\end{center}
Recall that $(M\times M,|\cdot, \cdot|_{M\times M})$ is also an Alexandrov space.
 The geodesic balls in $M\times M$ are denoted by
 \begin{center}$B^{M\times M}_{(x,y)}(r):=\{(z,w):\ |(z,w),(x,y)|_{M\times M}<r\}$.\end{center}
\begin{prop}\label{prop5.4}
Let $\Omega,Y$ and $u,f$ be as the above. Then the function $f$ is sub-solution of $\mathscr L^{(2)}_f=0$ on $\Omega\times\Omega$,
where $\mathscr L^{(2)}$ is the Laplacian on $\Omega\times\Omega$.\\
(Because $M\times M$ is also an Alexandrov space, the notion  $\mathscr L^{(2)}$ makes sense.)
\end{prop}
\begin{proof}We divide the proof into three steps.

(i)  For any $P\in Y$, we firstly prove that the functions $f_P(x):=d_Y\big(u(x),P\big)$ satisfy $\mathscr L_{f_P}\gs0$ on $\Omega$.

 Take any $\epsilon>0$ and set
 $$f_\epsilon(x):=\sqrt{f^2_P(x)+\epsilon^2}.$$
We have
$$ |\nabla f_\epsilon|=\frac{f_P}{f_\epsilon}\cdot |\nabla f_P|\ls |\nabla f_P|.$$
Thus, we have $f_\epsilon\in W^{1,2}(\Omega)$, since $f_P\in W^{1,2}(\Omega)$.
 We will prove that, for any $\epsilon>0$,  $\mathscr L_{f_\epsilon}$ forms a nonnegative Radon measure.

From Proposition \ref{prop4.1} (1) and (5), we get that $f_P\in W^{1,2}(\Omega)$ and
$$E^u_2\gs E^{f_P}_2=|\nabla f_P|^2\cdot\rv.$$
By combining with Lemma \ref{lem5.3} (ii),
\begin{equation}\label{eq5.3}
 \mathscr L_{f^2_\epsilon}=\mathscr L_{f^2_P}\gs 2E^u_2\gs 2|\nabla f_P|^2\cdot\rv\gs2|\nabla f_\epsilon|^2\cdot\rv.
\end{equation}
Take any   test function $\phi\in Lip_0(\Omega)$ with $\phi\gs0$. By using
\begin{equation*}
 -\mathscr L_{f^2_\epsilon}(\phi)\!=\!\!\!\int_\Omega\!\!\ip{\nabla f^2_\epsilon}{\nabla \phi}\!d\rv\!=\!2\!\!\int_\Omega
 \!\ip{\nabla f_\epsilon}{\nabla (f_\epsilon\cdot\phi)}\!d\rv\!-\!2\!\!\int_\Omega\!\phi\cdot|\nabla f_\epsilon|^2\!d\rv,
\end{equation*}
and combining with equation \eqref{eq5.3}, we obtain that the functional
$$I_\epsilon(\phi):= -\int_\Omega\ip{\nabla f_\epsilon}{\nabla (f_\epsilon\cdot\phi)}d\rv=\mathscr L_{f_\epsilon}(f_\epsilon\cdot\phi)$$
 on $Lip_0(\Omega)$ is nonnegative. According to the Theorem 2.1.7 of \cite{h89}, there exists a (nonnegative) Radon measure, denoted by $\nu_\epsilon$,
  such that
$$\nu_\epsilon(\phi)=I_\epsilon(\phi)=\mathscr L_{f_\epsilon}(f_\epsilon\cdot\phi).$$
This implies that, for any $\psi\in Lip_0(\Omega)$ with $\psi\gs0$,
$$\mathscr L_{f_\epsilon}(\psi)=\nu_\epsilon(\frac{\psi}{f_\epsilon})\gs0.$$
Thus, we get that $\mathscr L_{f_\epsilon}$ is a nonnegative functional on $Lip_0(\Omega)$, and hence, by using  the Theorem 2.1.7 of \cite{h89} again, it forms a nonnegative Radon measure.

Now let us prove the sub-harmonicity of $f_P$. Noting that, for any $\epsilon>0$,
 $$|\nabla f_\epsilon|\ls |\nabla f_P|\qquad{\rm and} \quad 0<  f_\epsilon\ls f_P+\epsilon,$$
 we get that the set $\{f_\epsilon\}_{\epsilon>0}$ is bounded uniformly in $W^{1,2}(\Omega).$
  Hence, it is weakly compact. Then there exists a sequence of numbers $\epsilon_j\to0$ such that
 $$f_{\epsilon_j}\overset{w}{\rightharpoonup}f_P\quad {\rm  in\ }\ W^{1,2}(\Omega).$$
Therefore, the sub-harmonicity of $f_{\epsilon_j}$ for any $j\in\mathbb N$ implies that  $f_P$ is sub-harmonic. This completes the proof of (i).

(ii) We next prove that $f$ is in $W^{1,2}(\Omega\times\Omega)$.

Let us consider the approximating energy density of $f$ at point $(x,y)\in\Omega\times\Omega$. Fix any positive number
 $\epsilon$ with $B_x(2\epsilon)\subset \Omega$ and $B_y(2\epsilon)\subset \Omega$.
By the definition of approximating energy density, the triangle inequality, and by noting that the ball in $\Omega\times\Omega$ satisfying
\begin{center}$B^{M\times M}_{(x,y)}(\epsilon)\subset B_x(\epsilon)\times B_y(\epsilon)\subset \Omega\times\Omega,$\end{center}
we have
\begin{equation*}
\begin{split}
\frac{c_{2n,2}}{2n+2}&\cdot e^f_{2,\epsilon}(x,y)\\
&=\int_{B^{M\times M}_{(x,y)}(\epsilon)}\frac{|f(x,y)-f(z,w)|^2}{\epsilon^{2n+2}}d\rv(z)\rv(w)\\
&\ls \int_{B_x(\epsilon)\times B_y(\epsilon)}\frac{\big[d_Y\big(u(x),u(z)\big)+d_Y\big(u(y),u(w)\big)\big]^2}{\epsilon^{2n+2}}d\rv(z)d\rv(w)\\
&\ls 2\cdot\rv\big(B_y(\epsilon)\big)\cdot\int_{B_x(\epsilon)}\frac{d_Y\big(u(x),u(z)\big)^2}{\epsilon^{2n+2}}d\rv(z)\\
&\qquad  +2\cdot\rv\big(B_x(\epsilon)\big)\cdot\int_{B_y(\epsilon)}\frac{d_Y\big(u(y),u(w)\big)^2}{\epsilon^{2n+2}}d\rv(w)\\
&\ls 2\frac{\rv\big(B_y(\epsilon)\big)}{\epsilon^n}\cdot \frac{c_{n,2}}{n+2}e^u_{2,\epsilon}(x)+2\frac{\rv\big(B_x(\epsilon)\big)}{\epsilon^n}\cdot \frac{c_{n,2}}{n+2}e^u_{2,\epsilon}(y)\\
&\ls c_{n,k,{\rm diameter}(\Omega)} \cdot\big(e^u_{2,\epsilon}(x)+ e^u_{2,\epsilon}(y)\big).
\end{split}
\end{equation*}
Then, by  the definition of energy functional, it is easy to see that $f$ has finite energy. Hence  $f$ is in $W^{1,2}(\Omega\times\Omega)$.

(iii) We want to prove that $f$  is sub-harmonic on $\Omega\times\Omega$.

 For any $g\in W^{1,2}(\Omega\times\Omega)$, by Fubini's Theorem, we   conclude that, for almost all $x\in \Omega$, the functions $g_x(\cdot):=g(x,\cdot)$ are in $W^{1,2}(\Omega)$,
 and that the same assertions hold for the functions $g_y(\cdot):=g(\cdot,y)$. We denote by $\nabla^{M\times M}g$ the weak gradient of $g$. Note that the metric on $M \times M$ is the product metric,
we have
\begin{center}
$\ip{\nabla_{M\times M} g}{\nabla_{M\times M} h}(x,y)=\ip{\nabla_1 g}{\nabla_1 h}+\ip{\nabla_2 g}{\nabla_2 h},$
\end{center}
for any $g,h\in W^{1,2}(\Omega\times\Omega)$, where $\nabla_1g$ is the weak gradient of the function $g_y(\cdot):=g(\cdot,y):\Omega\to \mathbb R,$ and $\nabla_2g$ is similar.

Now, we are in the position to prove   sub-harmonicity of $f$. Take any test function $\varphi(x,y)\in Lip_0(\Omega\times\Omega)$ with $\varphi(x,y)\gs0$.
\begin{equation}\label{eq5.4}
\begin{split}
 \int_{\Omega\times\Omega}&\ip{\nabla_{M\times M} f}{\nabla_{M\times M} \varphi}_{(x,y)}d\rv(x)d\rv(y)\\
 &\qquad=\int_{\Omega}\int_{\Omega}\ip{\nabla_1 f}{\nabla_1 \varphi}d\rv(x)d\rv(y)\\
 &\qquad\quad+\int_{\Omega}\int_{\Omega}\ip{\nabla_2 f}{\nabla_2 \varphi}d\rv(y)d\rv(x).
 \end{split}
 \end{equation}

Fix $y\in\Omega$ and note that the function $\varphi_y(\cdot):=\varphi(\cdot,y)\in Lip_0(\Omega)$. According to (i), the function $f_{u(y)}:=d_Y\big(u(\cdot),u(y)\big)$ is sub-harmonic on $\Omega$. Hence, we have
$$\int_{\Omega}\ip{\nabla_1 f}{\nabla_1 \varphi}d\rv(x)=-\mathscr L_{f_{u(y)}} (\varphi_y(\cdot))\ls0.$$
By the same argument, we get for any fixed $x\in\Omega$,
$$\int_{\Omega}\ip{\nabla_2 f}{\nabla_2 \varphi}d\rv(y)\ls0.$$
By substituting these above two inequalities into equation \eqref{eq5.4}, we have
$$  \int_{\Omega\times\Omega}\ip{\nabla_{M\times M} f}{\nabla_{M\times M} \varphi}_{(x,y)}d\rv(x)d\rv(y)\ls0,$$
for any function $\varphi\in Lip_0(\Omega\times\Omega)$.
This implies that $f$ is sub-harmonic on $\Omega\times\Omega$. The proof of the proposition is completed.
\end{proof}

Now we can establish the following estimates for pointwise Lipschitz constants of harmonic maps.
\begin{thm} \label{thm5.5}
Let $\Omega$ be a bounded domain in an $n$-dimensional Alexandrov space $(M,|\cdot,\cdot|)$ with curvature $\gs k$ for some $k\ls 0$, and let $Y$ be an NPC space.
Suppose that $u$ is a harmonic map from $\Omega$ to $Y$.
Then, for any ball $B_q(R)\subset\subset \Omega$, there exists a constant $C(n,k,R)$, depending only on $n,k$ and $R$, such that the following estimate holds:
\begin{equation}\label{eq5.5}
{\rm Lip}^2u(x)\ls C(n,k,R)\cdot |\nabla u|_{2}(x)<+\infty
\end{equation}
for almost everywhere $x\in B_q(R/6)$, where $|\nabla u|_2$ is the density of the absolutely continuous part of energy measure $E^u_2$ with respect to $\rv$.
\end{thm}
\begin{proof} Fix any ball $B_q(R)\subset\subset \Omega$. Throughout this proof, all of constants $C_1,C_2,\cdots$ depend only on $n,k$ and $R$.

Note that $M\times  M$ has curvature lower bound $\min\{k,0\}=k$, and that
  ${\rm diam}(B_q(R)\times B_q(R))= \sqrt2 R$.
Clearly, on $B_q(R)\times B_q(R)$, both the measure doubling property and the (weak) Poincar\'e inequality hold,
with the corresponding doubling and Poincar\'e constants
  depending only on $n,k$ and $R$.
On the other hand, from  Proposition \ref{prop5.4}, the function
\begin{center}$f(x,y):=d_Y\big(u(x),u(y)\big)$\end{center}
is sub-harmonic on $B_q(R)\times B_q(R)$.
 By Theorem 8.2 of \cite{bm06}, (or a Nash-Moser iteration argument), there exists a constant $C_1$ such that
$$\sup_{B^{M\times M}_{(x,y)}(r) }f\ls C_1\cdot\bigg(\fint_{B^{M\times M}_{(x,y)}(2r)}f^2d\rv_{M\times M}\bigg)^{\frac{1}{2}}$$
for any $(x,y)\in B_q(R/2)\times B_q(R/2)$ and any $r>0$ with $B^{M\times M}_{(x,y)}(2r)\subset\subset B_q(R)\times B_q(R),$ where,
for any function $h\in L^1(E)$ on a measurable set $E$,
\begin{center}
$\fint_Ehd\rv:=\frac{1}{\rv(E)}\int_Ehd\rv.$
\end{center}
In particular, for any fixed $z\in B_q(R/2)$ and any $r>0$ with $B_z(2r)\subset B_q(R)$, by noting that
\begin{center}
$B_z(r/2)\times B_z(r/2)\subset B^{M\times M}_{(z,z)}(r)\quad $ and $\quad B^{M\times M}_{(z,z)}(2r)\subset B_z(2r)\times B_z(2r)$,
\end{center}
 we have
\begin{equation}\label{eq5.6}
\sup_{y\in B_z(r/2)}\!f^2(y,z)\ls\sup_{B_z(r/2)\times B_z(r/2)}\!f^2\ls\frac{ C^2_1}{\rv\big(B^{M\times M}_{(z,z)}(2r)\big)}\! \int_{B_z(2r)\times B_z(2r)}
\!f^2d\rv_{M\times M}.
\end{equation}

From  Proposition \ref{prop4.1} (4), there exists constant $C_2$ such that
the following holds: for any $z\in B_q(R/6)$ and any $0<r<R/4$, we have
$$\int_{B_z(2r)}\int_{B_z(2r)}f^2 (x,y)d\rv(x)d\rv(y)\ls C_2r^{n+2}\cdot\int_{B_z(12 r)}dE^u_2. $$
By combining with equation \eqref{eq5.6}, we get for any $z\in B_q(R/6)$
\begin{equation}\label{eq5.7}
\sup_{y\in B_z(r/2)}\frac{f^2(y,z)}{r^2}\ls C^2_1\cdot C_2\cdot \frac{r^n\cdot \rv\big(B_z(12 r)\big)}{\rv\big(B^{M\times M}_{(z,z)}(2r)\big)} \fint_{B_z(12 r)}dE^u_2
\end{equation}
for any $0<r<R/4$.
Noticing that $B_z(r)\times B_z(r)\subset B^{M\times M}_{(z,z)}(2r)$ again, according to the Bishop--Gromov volume comparison \cite{bgp92}, we have
$$\frac{r^n\cdot\rv\big(B_z(12 r)\big)}{\rv\big(B^{M\times M}_{(z,z)}(2r)\big)}\ls \frac{r^n}{\rv\big(B_{z}(r)\big) }\cdot\frac{\rv\big(B_z(12 r)\big)}{\rv\big(B_{z}(r)\big)}\ls C_3\cdot \frac{r^n}{\rv\big(B_{z}(r)\big) }$$
for any  $0<r<R/4$.
Hence, by using this and the equation \eqref{eq5.7}, we obtain that, for any $z\in B_q(R/6)$,
\begin{equation*}
\sup_{y\in B_z(r/2)}\frac{f^2(y,z)}{r^2}\ls C_4\cdot \frac{r^n}{\rv\big(B_{z}(r)\big) } \cdot\fint_{B_z(12 r)}dE^u_2
\end{equation*}
 for any $0<r<R/4$, where $C_4:=C^2_1\cdot C_2\cdot C_3$.
Therefore, we  conclude that
\begin{equation}\label{eq5.8}
\begin{split}
{\rm Lip}^2u(z)&=\limsup_{r\to0}\sup_{|yz|\ls r/4}\frac{f^2(y,z)}{(r/4)^2}\ls 16\cdot\limsup_{r\to0}\sup_{|yz|< r/2}\frac{f^2(y,z)}{r^2}\\
&\ls 16 C_4\cdot \limsup_{r\to0} \frac{r^n}{\rv\big(B_{z}(r)\big) }\cdot \limsup_{r\to0} \fint_{B_z(12 r)}dE^u_2
\end{split}
\end{equation}
for any $z\in B_q(R/6)$.
According to the Lebesgue decomposition theorem (see, for example, Section 1.6 in \cite{evans}), we know that, for almost everywhere $x\in B_q(R/6)$, the limit $\lim_{r\to0} \fint_{B_x(r)}dE^u_2$ exists and
\begin{equation}\label{eq5.9}
\lim_{r\to0} \fint_{B_x(r)}dE^u_2=|\nabla u|_2(x).
\end{equation}
On the other hand, from \cite{bgp92}, we know that
\begin{equation}\label{eq5.10}
\lim_{r\to0} \frac{r^n}{\rv(B_{x}(r)) }=n/\omega_{n-1}
 \end{equation}
 for any regular point $x\in B_q(R/6)$ and that the set of regular points in an Alexandrov space has full measure. Thus, \eqref{eq5.10} holds for almost all $x\in B_q(R/6)$.  By using this and \eqref{eq5.8}-\eqref{eq5.10}, we   get the estimate \eqref{eq5.5}.
\end{proof}

Consequently, we have the following mean value inequality.
\begin{cor}\label{coro5.6}
Let $\Omega$ be a bounded domain in an $n$-dimensional Alexandrov space $(M,|\cdot,\cdot|)$ and let $Y$ be an NPC space.
Suppose that $u$ is a harmonic map from $\Omega$ to $Y$.
 Then, for almost everywhere $x_0\in\Omega$, we have
 the following holds:
$$ \int_{B_{x_0}(R)}\Big[ d^2_Y\big(P,u(x_0)\big)-d^2_Y\big(P,u(x)\big)\Big] d\rv(x)
 \ls  -\frac{|\nabla u|_2(x_0)\cdot\omega_{n-1}}{n(n+2)}\cdot R^{n+2}+o(R^{n+2}).$$
for every $P\in Y$.
\end{cor}
\begin{proof}
We define a subset of $\Omega$ as
\begin{center}
$A:=\big\{x\in\Omega|\ x\ \ {\rm is\ smooth},\ {\rm Lip}u(x)<+\infty, \qquad\qquad\ $
\end{center}
\begin{center}  $\quad\qquad\qquad {\rm and}\  x\ {\rm is\ a\ Lebesgue\ point\ of\ } |\nabla u|_2\ \big\}.$
\end{center}
According to the above Theorem \ref{thm5.5} and \cite{per-dc}, we have $\rv(\Omega\backslash A)=0$.

Fix any point $x_0\in A.$ For any $P\in Y$, we consider the function on $\Omega$
$$g_{x_0,P}(x):=d^2_Y\big(P,u(x_0)\big)-d^2_Y\big(P,u(x)\big).$$
Then, from Lemma \ref{lem5.3} (ii), we have
$$\mathscr L_{g_{x_0,P}}\ls-2E_2^u\ls -2|\nabla u|_2\cdot\rv.$$
Since $x_0$ is a Lebesgue point of the function $-2|\nabla u|_2$, by applying
 Proposition \ref{mean} to nonnegative function (note that $-2|\nabla u|_2\ls0$,)
 $$g_{x_0,P}(x)-\inf_{B_{x_0}(R)} g_{x_0,P}(x),$$
 we   obtain
 \begin{equation*}
 \begin{split}
 &\frac{1}{H^{n-1}(\partial B_o(R)\subset T^k_{x_0})}\int_{\partial B_{x_0}(R)}\Big[g_{x_0,P}(x)-\inf_{B_{x_0}(R)} g_{x_0,P}(x)\Big]d\rv\\
 &\qquad\quad\ls \big[g_{x_0,P}(x_0)-\inf_{B_{x_0}(R)} g_{x_0,P}(x)\big]-\frac{2|\nabla u|_2(x_0)}{2n}\cdot R^{2}+o(R^{2}).
 \end{split}
 \end{equation*}
Denote by
\begin{center}
$A(R):=\rv\big(\partial B_{x_0}(R)\subset M\big)\ $ and $\quad\overline{A}(R):=H^{n-1}\big(\partial B_o(R)\subset T^k_{x_0}\big)$.
\end{center}
Noting that $g_{x_0,P}(x_0)=0$,  we have
 \begin{equation}\label{eq5.11}
 \begin{split}
\int_{\partial B_{x_0}(R)}g_{x_0,P}(x)d\rv \ls &-\inf_{B_{x_0}(R)}g_{x_0,P}(x)\cdot\Big(\overline A(R)-A(R)\Big)\\
 &-\Big(\frac{|\nabla u|_2(x_0)}{n}\cdot R^{2}+o(R^{2})\Big)\cdot \overline A(R).
 \end{split}
 \end{equation}
By applying co-area formula, integrating  two sides of equation \eqref{eq5.11} on $(0,R)$, we have
 \begin{equation}\label{eq5.12}
 \begin{split}
\int_{ B_{x_0}(R)} g_{x_0,P}(x)d\rv&=\int_0^R \int_{\partial B_{x_0}(r)} g_{x_0,P}(x)d\rv\\
&\ls-\int_0^R\inf_{B_{x_0}(r)}g_{x_0,P}(x)\cdot\Big(\overline A(r)-A(r)\Big)dr\\
 &\quad \ -\int_0^R\Big(\frac{|\nabla u|_2(x_0)}{n}\cdot r^{2}+o(r^{2})\Big)\cdot \overline A(r)dr\\
 &:=I(R)+II(R).
 \end{split}
 \end{equation}
Since $M$ has curvature $\gs k$, the Bishop-Gromov inequality states that $A(r)\ls \overline A(r)$ for any $r>0$. Hence we have
$$\inf_{B_{x_0}(r)}g_{x_0,P}(x)\cdot\Big(\overline A(r)-A(r)\Big)\gs \inf_{B_{x_0}(R)}g_{x_0,P}(x)\cdot\Big(\overline A(r)-A(r)\Big)$$
for any $0\ls r\ls R$. So we obtain
\begin{equation}\label{eq5.13}
\begin{split}
I(R)&\ls -\inf_{B_{x_0}(R)}g_{x_0,P}(x)\cdot\int_0^R\Big(\overline A(r)-A(r)\Big)dr\\
&=-\inf_{B_{x_0}(R)}g_{x_0,P}(x)\cdot\Big(H^{n}\big( B_o(R)\subset T^k_{x_0}\big)-\rv\big( B_{x_0}(R)\big)\Big).
\end{split}
\end{equation}
By ${\rm Lip}u(x_0)<+\infty$ and the triangle inequality, we have
\begin{equation}\label{eq5.14}
\begin{split}
|g_{x_0,P}(x)|\!
=&\Big(d_Y\big(P,u(x_0)\big)+d_Y\big(P,u(x)\big)\Big)\cdot\big|d_Y\big(P,u(x_0)\big)-d_Y\big(P,u(x)\big)\big|\\
\ls& \Big(2d_Y\big(P,u(x_0)\big)\!+\!d_Y\big(u(x_0),u(x)\big)\Big)\cdot d_Y\big(u(x),u(x_0)\big) \\
\ls& \Big(2d_Y\big(P,u(x_0)\big)\!+\!{\rm Lip}u(x_0)\cdot R\!+\!o(R)\Big)\cdot \big({\rm Lip}u(x_0)\!\cdot\! R\!+\!o(R)\big)\\
=&O(R).
\end{split}
\end{equation}
Since $x_0$ ia  a smooth point, from Lemma \ref{smooth}, we have
$$\big|H^{n}\big( B_o(R) \subset T_{x_0}\big)-\rv\big( B_{x_0}(R)\big)\big|\ls o(R)\cdot H^{n}\big( B_o(R) \subset T_{x_0}\big)=o(R^{n+1}).$$
By using the fact that $x_0$ is smooth again, and hence  $T^k_{x_0}$ is isometric $\mathbb M^n_k$, we have
\begin{equation*}
\begin{split}
\big|H^{n}\big( B_o(R)&\! \subset\! T^k_{x_0}\big)-H^{n}\big(B_o(R)\! \subset\! T_{x_0}\big)\big|
=\big|H^{n}\big( B_o(R)\! \subset\! \mathbb M^n_k\big)-H^{n}\big( B_o(R)\! \subset\! \mathbb R^n\big)\big|\\
&=O(R^{n+2}).
 \end{split}
 \end{equation*}
By substituting the above two estimates and \eqref{eq5.14} into   \eqref{eq5.13}, we  obtain
\begin{equation}\label{eq5.15}
I(R)\ls o(R^{n+2}).
\end{equation}

Now let us estimate $II(R).$ Note that $x_0$ is a smooth point. In particular, it is a regular point. Hence
$$\overline A(r)=\rv(\Sigma_{x_0})\cdot s_k^{n-1}(r)=\omega_{n-1}r^{n-1}+o(r^{n-1}).$$
We have
\begin{equation}\label{eq5.16}
\begin{split}
II(R)&=-\int_0^R\Big(\frac{|\nabla u|_2(x_0)}{n}\cdot r^{2}+o(r^{2})\Big)\cdot \overline A(r)dr\\
&=-\frac{|\nabla u|_2(x_0)\cdot\omega_{n-1}}{n}\int_0^R\big( r^{n+1}+o(r^{n+1})\big) dr\\
&=-\frac{|\nabla u|_2(x_0)\cdot\omega_{n-1}}{n(n+2)}\cdot R^{n+2}+o(R^{n+2}).
\end{split}
\end{equation}
The combination of equations \eqref{eq5.12} and \eqref{eq5.15}--\eqref{eq5.16}, we have
 \begin{equation*}
\int_{ B_{x_0}(R)} g_{x_0,P}(x)d\rv\ls-\frac{|\nabla u|_2(x_0)\cdot\omega_{n-1}}{n(n+2)}\cdot R^{n+2}+o(R^{n+2}).
 \end{equation*}
This is desired estimate. Hence we complete the proof.\end{proof}

\section{Lipschtz regularity}

We will prove the main Theorem \ref{main-thm} in this section. The proof is split into
two steps, which are contained in the following two subsections.
In the first subsection, we will construct a family of auxiliary functions $f_t(x,\lambda)$ and prove that they are super-solutions of the heat equation (see Proposition \ref{prop6.12}).
In the second subsection, we will complete the proof.

Let $\Omega$ be a bounded domain in an $n$-dimensional Alexandrov space $(M,|\cdot,\cdot|)$ with curvature $\gs k$ for some number $k\ls0$,
and let $(Y,d_Y)$ be a complete \emph{NPC} metric space.
\emph{In this section,
we always assume that $u:\Omega\to Y$ is an (energy minimizing) harmonic map.
From Lemma 5.3, we can assume
that $u$ is continuous on $\Omega$.}

\subsection{A family of auxiliary functions with two parameters.} $\ $

Fix any domain $\Omega'\subset\subset\Omega$.
 For any $t>0$ and any $0\ls \lambda\ls 1$, we define the following auxiliary function $f_t(x,\lambda)$ on $\Omega'$ by:
\begin{equation}\label{eq6.1}
f_t(x,\lambda):=\inf_{y\in\Omega'}\Big\{e^{-2nk\lambda}\cdot\frac{|xy|^2}{2t}-d_Y\big(u(x),u(y)\big)\Big\},\qquad x\in \Omega'.
\end{equation}
We denote by $S_t(x,\lambda)$ the set of all points where are the ``inf" of \eqref{eq6.1} achieved, i.e.,
$$S_t(x,\lambda):=\Big\{y\in \Omega'\ |\ f_t(x,\lambda)=e^{-2nk\lambda}\cdot\frac{|xy|^2}{2t}-d_Y\big(u(x),u(y)\big) \Big\}.$$

It is clear that (by setting $y=x$)
\begin{equation}\label{eq6.2}
0\gs f_t(x,\lambda)\gs-{\rm osc}_{\overline{\Omega'}}u :=-\max_{x,y\in\overline{\Omega'}}d_Y\big(u(x),u(y)\big).
\end{equation}

 Given a function $g(x,\lambda)$ defined on $\Omega\times\mathbb R$, we always denote by $g(\cdot,\lambda)$ the function $x\mapsto g(x,\lambda)$ on $\Omega$.
 The notations $g(x,\cdot)$ and $g(\cdot,\cdot)$ are analogous.

\begin{lem}\label{lem6.1}
Fix any domain $\Omega''\subset\subset\Omega'$ and denote by
$$C_*:=2{\rm osc}_{\overline{\Omega'}}u+2\quad{\rm and}\quad t_0:=\frac{{\rm dist}^2(\Omega'',\partial \Omega')}{4C_*}.$$
For each $t\in(0,t_0)$,  we have \\
 \indent (i) for each $\lambda\in[0,1]$ and $x\in \Omega''$, the set $S_t(x,\lambda)\not=\varnothing $ and it is closed, and
 $$ f_t(x,\lambda)=\min_{y\in \overline{B_x(\sqrt{C_*t}})}\Big\{e^{-2nk\lambda}\cdot\frac{|xy|^2}{2t}-d_Y\big(u(x),u(y)\big)\Big\};$$

 \indent (ii) for each $\lambda\in[0,1]$, the function $f_t(\cdot,\lambda)$ is in $C(\Omega'')\cap W^{1,2}(\Omega'')$, and
 \begin{equation}\label{eq6.3}
 \int_{\Omega''}|\nabla f_t(x,\lambda)|^2d\rv(x)\ls 2\cdot e^{-4nk}\cdot\frac{{\rm diam}^2(\Omega')}{t^2}\cdot\rv(\Omega'')+2E^u_{2}(\Omega'');
 \end{equation}

 \indent (iii) for each $x\in\Omega''$, the function $f_t(x,\cdot)$ is Lipschitz continuous on $[0,1]$, and
 \begin{equation}\label{eq6.4}
 |f_t(x,\lambda)-f_t(x,\lambda')|\ls e^{-2nk}\cdot C_*\cdot|\lambda-\lambda'|,\qquad \forall \lambda,\lambda'\in[0,1].
 \end{equation}

  \indent (iv) the function $(x,\lambda)\mapsto f_t(x,\lambda)$ is in $C\big(\Omega''\times[0,1]\big)\cap W^{1,2}(\Omega''\times(0,1))$
   with respect to the product measure $\underline{\nu}\!:=\rv\times \mathcal L^1$, where $\mathcal L^1$ is the Lebesgue measure on $[0,1]$.
\end{lem}
\begin{proof}(i) Let $x\in\Omega''$.
The definition of $C_*$ and $t_0$ implies that $B_x(\sqrt{C_*t})\subset\subset \Omega'$. Let $t\in(0,t_0)$ and $\lambda\in[0,1]$.
  Take any a minimizing sequence $\{y_j\}_j$ of \eqref{eq6.1}.
We claim that
\begin{equation}\label{eq6.5}
|xy_j|^2\ls C_*t \qquad
\end{equation}
for all  sufficiently large $j\in\mathbb N$.
Indeed, from $f_t(x,\lambda)\ls0$, we get that
$$e^{-2nk\lambda}\cdot\frac{|xy_j|^2}{2t}-d_Y\big(u(x),u(y_j)\big)\ls 1$$
for all sufficiently large $j\in\mathbb N.$ Thus,
$$|xy_j|^2\ls 2t\big(1+d_Y\big(u(x),u(y_j)\big)\big)\ls 2t( 1+{\rm osc}_{\overline{\Omega'}}u)\ls C_*t$$
for all $j\in\mathbb N$ large enough,  where we have used that $k\ls0$ and the definition of $C_*.$
This proves \eqref{eq6.5}. The assertion (i) is implied by the combination of \eqref{eq6.5} and that $u$ is continuous.

(ii)
Let $t\in(0,t_0)$ and $\lambda\in[0,1]$ be fixed. Take any $x,y\in\Omega''$ and let point $z\in\Omega'$ achieve the minimum in the definition of $f_t(y,\lambda)$. We have, by the triangle inequality,
\begin{equation*}
\begin{split}
f_t(x,\lambda)-f_t(y,\lambda)&\ls e^{-2nk\lambda}\cdot \frac{|xz|^2}{2t}-d_Y\big(u(x),u(z)\big)-e^{-2nk\lambda}\cdot\frac{|yz|^2}{2t}+d_Y\big(u(y),u(z)\big)\\
&\ls e^{-2nk\lambda}\cdot \frac{(|xz|-|yz|)\cdot (|xz|+|yz|)}{2t}+d_Y\big(u(x),u(y)\big) \\
&\ls e^{-2nk\lambda}\cdot\frac{{\rm diam}(\Omega')}{t}\cdot|xy|+d_Y\big(u(x),u(y)\big).
\end{split}
\end{equation*}
By the symmetry of $x$ and $y$, we have
$$|f_t(x,\lambda)-f_t(y,\lambda)|\ls e^{-2nk\lambda}\cdot \frac{{\rm diam}(\Omega')}{t}\cdot|xy|+d_Y\big(u(x),u(y)\big).$$
 This inequality implies the following assertions:\\
 $\bullet$  $f(\cdot,\lambda)$ is continuous on $\Omega''$, since $u$ is continuous; \\
  $\bullet$  for any $\epsilon>0$, the approximating energy density of
$f(\cdot,\lambda)$ satisfies (since $e^{-2nk\lambda}\ls e^{-2nk}$)
$$ e^{f_t(\cdot,\lambda)}_{2,\epsilon}(x)\ls 2 e^{-4nk}\cdot{\rm diam}^2(\Omega')/t^2+2e^u_{2,\epsilon}(x),\qquad x\in \Omega''.$$
This implies \eqref{eq6.3}, and hence (ii).

(iii) Let any $x\in\Omega''$ be fixed. Take any   $\lambda,\mu\in[0,1]$. Let a point $z\in S_t(x,\mu)$. That is, point $z$ achieves the minimum in the definition of $f_t(x,\mu)$. By the triangle inequality, we get
\begin{equation*}
\begin{split}
f_t(x,\lambda)-f_t(x,\mu)&\ls e^{-2nk\lambda}\cdot\frac{|xz|^2}{2t}-d_Y\big(u(x),u(z)\big)-e^{-2nk\mu}\cdot\frac{|xz|^2}{2t}+d_Y\big(u(x),u(z)\big)\\
&\ls \frac{ (e^{-2nk\lambda}-e^{-2nk\mu})\cdot|xz|^2}{2t}  \\
&\ls |\lambda-\mu|\cdot e^{-2nk}\cdot\frac{C_*t}{2t}\ls e^{-2nk}\cdot C_*\cdot |\lambda-\mu|,
\end{split}
\end{equation*}
  where we have used  $\lambda,\mu\ls1$ and $|xz|\ls \sqrt{C_*t}\ $ (since (i)).
By the symmetry of $\lambda$ and $\mu$, we have
\begin{equation*}
|f_t(x,\lambda)-f_t(x,\mu)|\ls e^{-2nk}\cdot C_*\cdot |\lambda-\mu|.
\end{equation*}
This completes (iii).

(iv) is a consequence of the combination of equation \eqref{eq6.3} and \eqref{eq6.4}, and that $f_t$ is bounded on $\Omega''\times[0,1]$.
\end{proof}

Fix any domain $\Omega''\subset\subset\Omega'$ and let $t_0$ be given  in Lemma \ref{lem6.1}. For each $t\in(0,t_0)$ and each $\lambda\in[0,1]$, the set $S_t(x,\lambda)$  is closed for all $x\in \Omega''$, by Lemma \ref{lem6.1}(i).  We define a function $L_{t,\lambda}(x)$ on $\Omega''$ by
\begin{equation}\label{eq6.6}
L_{t,\lambda}(x):={\rm dist}\big(x, S_t(x,\lambda)\big)=\min_{y\in S_t(x,\lambda)}|xy|,\quad \ x\in\Omega''.
\end{equation}
\begin{lem}\label{lem6.2}
 Fix any domain $\Omega''\subset\subset\Omega'$. For each $t\in(0,t_0)$, we have:\\
 \indent (i) the function $(x,\lambda)\mapsto L_{t,\lambda}(x)$ is lower semi-continuous in $\Omega''\times[0,1]$;\\
 \indent (ii) for each $\lambda\in[0,1]$,
\begin{equation}\label{eq6.7}
\|L_{t,\lambda}\|_{L^\infty(\Omega'')}\ls \sqrt{C_*t},
\end{equation}
  where the constant $C_*$ is given in Lemma \ref{lem6.1}.
\end{lem}
\begin{proof}
Let $x\in \Omega''$ and $\lambda\in[0,1]$. We take sequences $\{(x_j,\lambda_j)\}_j\subset \Omega''\times[0,1]$ with $(x_j,\lambda_j)\to (x,\lambda)$, as $j\to\infty$, such that
$$\lim_{j\to\infty} L_{t,\lambda_j}(x_j)=\liminf_{z\to x,\ \mu\to\lambda}L_{t,\mu}(z).$$
 For each $j$, let $y_j\in S_t(x_j,\lambda_j)$ such that $L_{t,\lambda_j}(x_j)=|x_jy_j|.$
Since  ${\rm dist}(y_j,\Omega'')\ls \sqrt{C_*t_0}= {\rm dist}(\Omega'',\partial\Omega')/2$ for all $j\in\mathbb N$ (by Lemma \ref{lem6.1}(i)), there exists a subsequence, say $\{y_{j_l}\}_l$, converging to some $y\in \Omega'$. By the continuity of $u$ and $f_t(\cdot,\lambda)$ (see Lemma \ref{lem6.1}(iv)), we get
$$f_t(x,\lambda)=e^{-2nk\lambda}\cdot \frac{|xy|^2}{2t}-d_Y\big(u(x),u(y)\big).$$
This implies $y\in S_t(x,\lambda)$. From the definition of $L_{t,\lambda}(x)$, we have
$$L_{t,\lambda}(x)\ls |xy|=\lim_{l\to\infty}|x_{j_l}y_{j_l}|=\lim_{l\to\infty} L_{t,\lambda_j}(x_{j_l})=\liminf_{z\to x,\ \mu\to \lambda}L_{t,\lambda}(z).$$
Therefore, $L_{t,\lambda}$ is lower semi-continuous on $\Omega''\times[0,1]$. The proof of (i) is complete.

For each $t\in(0,t_0)$ and each $\lambda\in[0,1]$, the function $ L_{t,\lambda}(\cdot)$ is lower semi-continuous, and hence it is measurable, on $\Omega''$. By Lemma \ref{lem6.1}(i) and the definition of $ L_{t,\lambda}$, we have $0\ls L_{t,\lambda}(x)\ls \sqrt{C_*t}$ for all $x\in \Omega''$. Hence,
the estimate \eqref{eq6.7} holds. This completes the proof of the lemma.
\end{proof}
\begin{lem}\label{lem6.3}
Fix any domain $\Omega''\subset\subset\Omega'$. For each $t\in(0,t_0)$, we have
\begin{equation*}
\liminf_{\mu\to0^+}\frac{f_t(x,\lambda+\mu)-f_t(x,\lambda)}{\mu}\gs -e^{-2nk\lambda}\cdot\frac{nk}{t}\cdot L^2_{t,\lambda}(x)
\end{equation*}
for any $\lambda\in[0,1)$ and $x\in \Omega''$.

Consequently, we have, for each $x\in\Omega''$, (by Lemma \ref{lem6.1}(iii))
\begin{equation}\label{eq6.8}
\frac{\partial f_t(x,\lambda)}{\partial \lambda}\gs -e^{-2nk\lambda}\cdot\frac{nk}{t}\cdot L^2_{t,\lambda}(x)\qquad \mathcal L^1{\rm-}a.e.\ \ \ \lambda\in(0,1).
\end{equation}

\end{lem}
\begin{proof}
Let $t\in(0,t_0)$, $\lambda\in[0,1)$ and $x\in\Omega''.$ For each $0<\mu<1-\lambda$, we take a point $y_{\lambda+\mu}\in S_t(x,\lambda+\mu)$.
By the definition of $f_t(x,\lambda)$ and $S_t(x,\lambda)$, we have
\begin{equation*}
\begin{split}
f_t(x,&\lambda+\mu)-f_t(x,\lambda)\\
&=e^{-2nk(\lambda+\mu)} \frac{|xy_{\lambda+\mu}|^2}{2t}\!-\!d_Y\big(u(x),u(y_{\lambda+\mu})\big)-\inf_{z}\!\Big\{e^{-2nk\lambda} \frac{|xz|^2}{2t}\!-\!d_Y\big(u(x),u(z)\big)\Big\}\\
&\gs (e^{-2nk(\lambda+\mu)}-e^{-2nk\lambda})\cdot\frac{|xy_{\lambda+\mu}|^2}{2t}\\
&\gs (e^{-2nk(\lambda+\mu)}-e^{-2nk\lambda})\cdot\frac{L^2_{t,\lambda+\mu}(x)}{2t},
\end{split}
\end{equation*}
where we have used  $k\ls0.$ By the lower semi-continuity of $L_{t,\lambda}$, we have
$$\liminf_{\mu\to0^+}\frac{f_t(x,\lambda+\mu)-f_t(x,\lambda)}{\mu}\gs e^{-2nk\lambda}\cdot (-nk)\cdot \frac{L^2_{t,\lambda}(x)}{t}.$$
This proves the lemma.\end{proof}

 We need a mean value inequality.
\begin{lem}\label{lemma6.3}
Given any $z\in \Omega$ and $P\in Y$, we define a function $w_{z,P}$ by
$$w_{z,P}(\cdot):= d^2_Y\big(u(\cdot),u(z)\big) -d^2_Y\big(u(\cdot),P\big)+ d^2_Y\big(P,u(z)\big).$$
Then,  there exists
 a sequence $\{\varepsilon_j\}_j$ converging to $0$ and a set $\mathscr N$ with $\rv(\mathscr N)=0$
 such that the following property holds:  given any $x_0\in \Omega\backslash\mathscr N$ and any $P\in Y$, the following mean value inequalities
\begin{equation}\label{equation6.3}
\int_{B_o(\varepsilon_j)\cap\mathscr W}w_{x_0,P}\big(\exp_{x_0}(\eta)\big)d\eta\ls o(\varepsilon_j^{n+2})
\end{equation}
hold for any set $\mathscr W\subset \mathscr W_{x_0}$ satisfying
\begin{equation}\label{equation6.4}
\frac{H^n\big(\mathscr W\cap B_o(\varepsilon_j)\big)}{H^n\big(B_o(\varepsilon_j)\subset T_{x_o}\big)} \gs 1-o(\varepsilon_j).
\end{equation}
\end{lem}
\begin{proof} We firstly show that there exists
 a sequence $\{\varepsilon_j\}_j$ converging to $0$ and a set $\mathscr N$ with $\rv(\mathscr N)=0$
 such that the following property holds: for any $x_0\in \Omega\backslash\mathscr N$ and any $P\in Y$, we have
 \begin{equation}\label{equation6.5}
\int_{B_{x_0}(\varepsilon_j)}w_{x_0,P}(x)d\rv(x)\ls o( \varepsilon_j^{n+2}).
\end{equation}

This comes from the combination of Corollary \ref{coro4.7} and Corollary \ref{coro5.6}. Indeed, on the one hand, by applying Corollary \ref{coro4.7} with $p=2$ to
  the sequence $\{\epsilon_j=j^{-1}\}_{j=1}^\infty$, we   conclude that there exists
 a subsequence $\{\varepsilon_j\}_j\subset\{\epsilon_j\}_j$ and a set $N_1$ with $\rv(N_1)=0$
 such that for any point $x_0\in \Omega\backslash N_1$, we have
 \begin{equation}\label{equation6.6}
 \begin{split}
\int_{B_{x_0}(\varepsilon_j)}d^2_Y&\big(u(x_0),u(x)\big)d\rv(x)\\
&=\frac{\omega_{n-1}}{n(n+2)}|\nabla u|_2(x_0)\cdot \varepsilon_j^{n+2}+o( \varepsilon_j^{n+2}),
\end{split}
\end{equation}
where we have used $c_{n,2}=\omega_{n-1}/n$. On the other hand, from Corollary \ref{coro5.6}, there exists a set $N_2$ with $\rv(N_2)=0$ such that,
for all $x_0\in\Omega\backslash N_2$, we have
 \begin{equation}\label{equation6.7}
 \begin{split}
 \int_{B_{x_0}(\varepsilon_j)}\Big[ d^2_Y&\big(P,u(x_0)\big)-d^2_Y\big(P,u(x)\big)\Big] d\rv(x)\\
 &\ls - \frac{|\nabla u|_2(x_0)\cdot\omega_{n-1}}{n(n+2)}\cdot \varepsilon_j^{n+2}+o(\varepsilon_j^{n+2})
 \end{split}
\end{equation}
for every $P\in Y$. Now, denote by $\mathscr N=N_1\cup N_2$. The  equation \eqref{equation6.5} follows from the combination of the definition of function $w_{x_0,P}$ and
\eqref{equation6.6}--\eqref{equation6.7}.

According to \cite{per-dc}, the set of smooth points has full measure in $M$. Then, without loss the generality, we can assume that $x_0$ is smooth.
 By Theorem \ref{thm5.5}, we can also assume that
${\rm Lip}u(x_0)<+\infty$.

Since the point $x_0$ is smooth, by using Lemma \ref{smooth}, we have
\begin{equation}\label{equation6.8}
\begin{split}
\int_{B_{o}(\varepsilon_j)\cap \mathscr W_{x_0}}&w_{x_0,P}\big(\exp_{x_0}(\eta)\big)dH^n(\eta)\\
&=\int_{B_{x_0}(\varepsilon_j)\cap W_{x_0}}w_{x_0,P}(x)\cdot\big(1+o(\varepsilon_j)\big)d\rv(x)\\
&\ls\int_{B_{x_0}(\varepsilon_j)}w_{x_0,P}(x)d\rv(x)+\int_{B_{x_0}(\varepsilon_j)}|w_{x_0,P}(x)|\cdot o(\varepsilon_j)d\rv(x).
\end{split}
\end{equation}
Here we have used that  $W_{x_0}$ has full measure in $M$ \cite{os94}.
Since ${\rm Lip}u(x_0)<+\infty$, we have, for $x\in B_{x_0}(\varepsilon_j)$,
$$d^2_Y\big(u(x),u(x_0)\big)\ls {\rm Lip}^2u(x_0)\cdot\varepsilon_j^2+o(\varepsilon_j^2).$$
By combining with the definition of function $w_{x_0,P}$ and \eqref{eq5.14}, we get
\begin{equation}\label{equation6.9}
|w_{x_0,P}(x)|\ls O(\varepsilon_j), \qquad \forall\ x\in B_{x_0}(\varepsilon_j).
\end{equation}
The combination of  \eqref{equation6.5},\eqref{equation6.8} and \eqref{equation6.9} implies that
\begin{equation}\label{equation6.8++}
\begin{split}
\int_{B_{o}(\varepsilon_j)\cap \mathscr W_{x_0}}\!\!\!w_{x_0,P}\big(\exp_{x_0}(\eta)\big)dH^n(\eta)&
\ls o(\varepsilon_j^{n+2})\!+\!O(\varepsilon_j)\!\cdot\! o(\varepsilon_j)\!\cdot\! \rv(B_{x_0}(\varepsilon_j))\\
 &= o( \varepsilon_j^{n+2}).
\end{split}
\end{equation}

Given any set $\mathscr W\subset \mathscr W_{x_0}$ satisfying equation \eqref{equation6.4}, we obtain
\begin{equation}\label{equation6.10}
\begin{split}
\Big|\int_{B_{o}(\varepsilon_j)\cap (\mathscr W_{x_0}\backslash\mathscr W)}&w_{x_0,P}\big(\exp_{x_0}(\eta)\big)dH^n(\eta)\Big|\\
&\overset{\eqref{equation6.9}}{\ls} O(\varepsilon_j)\cdot H^n\big( B_{o}(\varepsilon_j)\cap (\mathscr W_{x_0}\backslash\mathscr W)\big)\\
&\ \ls O(\varepsilon_j)\cdot H^n\big( B_{o}(\varepsilon_j)\backslash\mathscr W\big)\\
&\overset{\eqref{equation6.4}}{\ls} O(\varepsilon_j)\cdot o(\varepsilon_j)\cdot H^n\big( B_{o}(\varepsilon_j)\big)\\
&\ =o(\varepsilon^{n+2}_j).
\end{split}
\end{equation}
The combination of equations \eqref{equation6.8++} and \eqref{equation6.10} implies the equation \eqref{equation6.3}. Hence we have completed the proof.
\end{proof}

The following two lemmas were stated by Petrunin \cite{petu96}, and their detailed proofs were given in \cite{zz12}.
\begin{lem}
[Petrunin \cite{petu96}, see also Lemma 4.15 in \cite{zz12}]\label{lem6.4}
Let $h$ be the Perelman's concave function given
 in Proposition \ref{per-concave} on a neighborhood  $U\subset M$. Assume that $f$ is a semi-concave function defined on $U $.
  And suppose that   $u\in W^{1,2}(U )\cap C(U )$ satisfies $\mathscr L_u\ls \lambda\cdot\rv$  on $U$ for some constant $\lambda\in\R$.

We assume that point $x^*\in U$ is a minimal point of function $u+f+h$, then $x^*$ has to be regular.
\end{lem}

The second lemma is Petrunin's perturbation in \cite{petu96}. We need some notations.
Let $u\in W^{1,2}(D)\cap C(\overline{D})$
  satisfy $\mathscr L_u\ls \lambda\cdot\rv$  on a bounded domain $D$. Suppose that $x_0$ is the unique minimum point of $u$ on $D$ and
      $$u(x_0)<\min_{x\in \partial D}u.$$
     Suppose also that $x_0$ is regular and
   $g=(g_1,\ g_2,\ \cdots\ g_n): D\to \R^n$
   is a coordinate system around $x_0$ such that $g
$ satisfies the following:\\
 \indent (i) $g$ is an almost isometry from $D$ to $g(D)\subset\R^n$ (see \cite{bgp92}).
  Namely, there exists a sufficiently small number $\delta_0>0$ such that
   $$\Big|\frac{\|g(x)-g(y)\|}{|xy|}-1\Big|\le \delta_0,\qquad {\rm for\ all}\quad x,y \in D, \ x\not=y;$$
  \indent (ii) all of the coordinate functions $g_j,\ 1\ls j\ls n,$ are concave (\cite{per-morse}).\\
  Then there exists $\epsilon_0>0$ such that, for each vector $V=(v^1,v^2,\cdots,v^n)\in \R^{n}$ with  $|v^j|\ls\epsilon_0$
   for all $1\ls j\ls n$, the function $$G(V,x):=u(x)+V\cdot g(x)$$ has a minimum point in the interior of $D$, where $\cdot$
    is the Euclidean inner product of $\R^n$ and    $V\cdot g(x)=\sum_{j=1}^nv^{j}g_j(x)$.

  Let  $$\mathscr U=\{V\in \R^{n}:\ |v^j|<\epsilon_0,\ 1\ls j\ls n\}
 \subset\R^n.$$
  We define $ \rho: \mathscr U\to D$ by setting
  \begin{center}
  $\rho(V)$ to be one of minimum  point of $G(V,x)$.
  \end{center}
 Note that the map $\rho$ might not be uniquely defined.
\begin{lem}[Petrunin \cite{petu96}, see also Lemma 4.16 in \cite{zz12}]\label{lem6.5}
Let $u,\ x_0,$ $\{g_j\}_{j=1}^n$ and $\rho$ be as above. There exists some $\epsilon\in(0,\epsilon_0)$ such that
 for arbitrary  $\epsilon'\in(0,\epsilon)$, the image  $\rho (\mathscr U^+_{\epsilon'})$ has nonzero Hausdorff measure,
 where
 $$ \mathscr U_{\epsilon'}^+:=\{ V=(v_1,v_2,\cdots,v_n)\in \R^n: \  0<v^j<\epsilon' \ \ {\rm for\ all}\ \ 1\ls j\ls n\}.$$
 Consequently, given any set $A\subset D$ with full measure, then for any $\epsilon'<\epsilon$, there exists $V\in \mathscr U_{\epsilon'}^+$
  such that the function $u(x)+V\cdot g(x)$ has a minimum point in $A$.
 \end{lem}
\begin{proof}The first assertion is the result of Lemma 4.16 in \cite{zz12}. The second assertion is implied obviously by the first one.
\end{proof}

The following lemma is the key for us to prove that $f_t(x,\lambda)$ is a super-solution of the heat equation.
\begin{lem}\label{lem6.6}
Given any point $p\in \Omega'$, there exits  a neighborhood $U_p (=B_p(R_p))$ of $p$  and a constant $t_p>0$ such that,
 for each $t\in(0,t_p)$ and each $\lambda\in[0,1]$, the function $x\mapsto f_t(x,\lambda)$ is a
super-solution of the Poisson equation
\begin{equation}\label{eq6.9}
\mathscr L_{f_t(x,\lambda)}= -e^{-2nk\lambda}\cdot\frac{nk}{t}L^2_{t,\lambda}(x)\cdot\rv
\end{equation}
  on  $U_p.$\end{lem}
\begin{proof}
Let $U_p=B_p(R_p)\subset\subset\Omega'$ be a neighborhood of $p$ such that $U=B_p(2R_p)$ supports a Perelman's concave function $h$ (see Proposition \ref{per-concave}).
Suppose that  $t_p= R_p^2/(2C_*)$, where $C_*$ is given in Lemma \ref{lem6.1}.
Now,  for each $t\in(0,t_p)$, we have  $\varnothing\not=S_t(x,\lambda)\subset\subset U$ for any $(x,\lambda)\in U_p\times[0,1]$, by Lemma \ref{lem6.1}(i).

To prove the lemma, it suffices to prove the following claim.\\
\textbf{Claim:} \emph{For each $t\in(0,t_p)$ and each $\lambda\in[0,1]$, the function $x\mapsto f_t(x,\lambda)$ is a super-solution of the Poisson equation
\begin{equation*}
\mathscr L_{f_t(x,\lambda)}= \Big(-e^{-2nk\lambda}\cdot\frac{nk}{t}L^2_{t,\lambda}(x)+\theta\Big)\cdot\rv\quad   on \ \ U_p
\end{equation*}
for and any $\theta>0$.}

We will divide the argument into four steps, as we did in the proof of Proposition 5.3 in \cite{zz12}.
However, the method is used in the crucial fourth step there,
 is not available for our auxiliary functions $f_t(x,\lambda)$ in this paper.
 Here we will use a new idea in the fourth step via the previous mean value inequalities given in Lemma \ref{lemma6.3}.\\

\noindent\emph{Step 1. Setting up a contradiction argument.}

 Suppose that the \textbf{Claim} fails for some $t\in (0,t_p)$, $\lambda\in[0,1]$ and some  $\theta_0>0$.  According to Corollary \ref{cor3.5},
 there exists a domain  $B\subset\subset U_p$
 such that the function $f_{t}(\cdot,\lambda)-v(\cdot)$ satisfies
$$\min_{x\in B}\big(f_{t}(x,\lambda)-v(x)\big)<0=\min_{x\in\partial B}\big(f_{t}(x,\lambda)-v(x)\big), $$
where $v$ is the (unique) solution of the Dirichlet problem
\begin{equation*}
\begin{cases}
\mathscr L_{v}&=\Big(-e^{-2nk\lambda}\cdot\frac{nk}{t}L^2_{t,\lambda}+\theta_0\Big)\cdot\rv\quad  {\rm in}\ \ B\\
v&=f_{t}(\cdot,\lambda) \quad {\rm on}\ \ \partial B.
\end{cases}
\end{equation*}
In this case we say   that $f_{t}(\cdot,\lambda)-v(\cdot)$ has a \emph{strict minimum} in the interior of $B$.

Let us define a function $H(x,y)$ on $B\times U$, similar as in \cite{petu96,zz12}, by
$$H(x,y):=\frac{e^{-2nk\lambda}}{2t}\cdot|xy|^2-d_Y\big(u(x),u(y)\big)-v(x).$$
Let $\bar x\in B$ be a minimum of $f_{t}(\cdot,\lambda)-v$ on $B$, and let $\bar y\in S_t(\bar x,\lambda)$  $(\subset\subset U)$ such that
\begin{equation}\label{eq6.10}
|\bar x\bar y|=L_{t,\lambda}(\bar x).
\end{equation}
By the definition of $S_t(\bar x,\lambda)$, $H(x,y)$ has a minimum at $(\bar x,\bar y)$.

Let us fix a real number $\delta_0$ with
\begin{equation}\label{eq6.11}
0<\delta_0\ls\frac{\theta_0}{8n(1+\sqrt{-k}\cdot{\rm diam}U)},
\end{equation}
and consider the function
$$H_0(x,y):=H(x,y)+\delta_0 |\bar xx|^2+\delta_0|\bar yy|^2,\quad (x,y)\in B\times U.$$
Since $(\bar x,\bar y)$ is one of the minimal points of $H(x,y)$, we conclude that it is the \emph{unique} minimal point of $H_0(x,y).$\\

\noindent\emph{Step 2. Petrunin's argument of perturbation.}

 In this step, we will perturb the above function $H_0$ to achieve some minimum at a smooth point.

Recall the Perelman's concave function $h$ is $2$-Lipschitz on $U$ (see Proposition \ref{per-concave}). Then, for any sufficiently small
 number $\delta_1>0$, the function
 $$H_1(x,y):=H_0(x,y)+\delta_1h(x)+\delta_1h(y)$$
  also achieves its a strict minimum in the  interior of $B\times U$.
   Let  $(x^*,y^*)$ denote one of minimal points of $H_1(x,y)$.

(i) We first claim that both points $x^*$ and $y^*$ are regular.

To justify this, we consider the function on $B$
\begin{equation*}
\begin{split}
H_1(x,y^*)&=H_0(x,y^*)+\delta_1h(x)+\delta_1h(y^*)\\
&=e^{-2nk\lambda}\cdot\frac{|xy^*|^2}{2t}-d_Y\big(u(x),u(y^*)\big)-v(x)+\delta_0 |\bar xx|^2+\delta_0|\bar yy^*|^2\\
&\quad +\delta_1h(x)+\delta_1h(y^*).
\end{split}
\end{equation*}
From the first paragraph of the proof of Proposition \ref{prop5.4}, we have
$$\mathscr L_{d_Y\big(u(x),u(y^*)\big) }\gs0.$$
Notice that $\mathscr L_v=-nk\cdot e^{-2nk\lambda}\cdot L^2_{t,\lambda}/t+\theta_0\in L^\infty(B)$ (since Lemma \ref{lem6.2}(ii)) and $|\bar x x|^2, |xy^*|^2/(2t)$ is semi-concave on $B$. Notice also that $x^*$ is a minimun of $H_1(x,y^*)$.   We can use Lemma \ref{lem6.4} to conclude
 that $x^*$ is regular. Using the same argument to function $H_1(x^*,y)$, we can get that $y^*$ is also regular.

Consider the function
$$H_2(x,y):=H_1(x,y)+ \delta_1\cdot|xx^*|^2+ \delta_1\cdot |yy^*|^2$$
on $B\times U$. It has the \emph{unique} minimal point at $(x^*,y^*)$.

(ii) We will use  Lemma \ref{lem6.5} to perturb the function $H_2$ to achieve some minimum at a smooth point.

Firstly, we want to show that
\begin{equation}\label{eq6.12}
\mathscr L^{(2)}_{H_2}\ls C(M,t,\lambda,\delta_1,\delta_0,\|L_{t,\lambda}\|_{L^\infty(B)})
\end{equation}
 for some constant $C(M,t,\delta_1,\delta_0,\|L_{t,\lambda}\|_{L^\infty(B)})$, where $\mathscr L^{(2)}$ is the Laplacian on $B\times U.$

Note that
$$|xy|^2=2\cdot {\rm dist}^2_{D_M}(x,y),$$
where ${\rm dist}_{D_M}(\cdot)$ is the distance function from the diagonal set $D_M :=\{(x,x):\ x\in M\}$ on $M\times M.$ Thus
 we know that $|xy|^2$ is a semi-concave function on $M\times M$. The function $|\bar xx|^2+|\bar yy|^2$ is also semi-concave on $M\times M$, because
$$|\bar xx|^2+|\bar yy|^2=|(x,y)(\bar x,\bar y)|^2_{M\times M}.$$
The function $|xx^*|^2+|yy^*|^2$ is semi-concave on $M\times M$ too.
By combining these with the concavity of $h(x)+h(y)$ on $U\times U$ and the sub-harmonicity of $d_Y\big(u(x),u(y)\big) $ on $U\times U$ (see Proposition \ref{prop5.4}),
and that $\mathscr L_v=-nk\cdot e^{-2nk\lambda}\cdot L^2_{t,\lambda}/t+\theta_0\in L^\infty(B)$  (since Lemma \ref{lem6.2}(ii)),
 we obtain \eqref{eq6.12}.

 Since $(x^*,y^*)$ is regular in $M\times M$, by \cite{bgp92} and \cite {per-dc},
 we can choose a nearly orthogonal coordinate system near $x^*$ by concave functions $g_1,g_2,\cdots, g_n$ and another nearly orthogonal coordinate system
 near $y^*$ by concave functions $g_{n+1},g_{n+2},\cdots, g_{2n}.$
Now,
  the point $(x^*,y^*)$, the function $H_2$ and system $\{g_i\}_{1\ls i\ls 2n}$ meet all of conditions in Lemma \ref{lem6.5}.

Meanwhile, according to Lemma \ref{lemma6.3}, there exists
 a sequence $\{\varepsilon_j\}_j$ converging to $0$ and a set $\mathscr N$ with $\rv(\mathscr N)=0$
 such that for all points $(x_0,y_0)\in (\Omega\backslash\mathscr N)\times (\Omega\backslash\mathscr N)$, the mean value inequalities \eqref{equation6.3} hold for functions $w_{x_0,P}$ and $w_{y_0,Q}$
for any $P, Q\in Y$ and any corresponding sets satisfying \eqref{equation6.4}.
(Please see Lemma \ref{lemma6.3} for the definition of functions $w_{x_0,P}$ and $w_{y_0,Q}$.) From now on, fixed such a sequence $\{\varepsilon_j\}_j$.

Hence, by applying Lemma \ref{lem6.5}, there exist arbitrarily small positive numbers $b_1,b_2,\cdots,b_{2n}$ such that the function
 $$H_3(x,y):=H_2(x,y)+\sum^n_{i=1}b_ig_i(x)+\sum^{2n}_{i=n+1}b_ig_i(y)$$
achieves a minimal point $(x^o,y^o)\in B\times U$, which satisfies the following properties:\\
\indent 1) $x^o\not=y^o$;\\
\indent 2) both $x^o$ and $y^o$ are smooth;\\
\indent 3) geodesic $x^oy^o$ can be extended beyond $x^o$ and $y^o$;\\
\indent 4) point $x^o$ is a Lebesgue point of $ e^{-2nk\lambda}\cdot\frac{-nk}{t}L^2_{t,\lambda}+\theta_0$;\\
\indent 5) the mean value inequalities \eqref{equation6.3} hold for functions $w_{x^o,P}$ and $w_{y^o,Q}$
for any $P,Q\in Y$ and any corresponding sets satisfying \eqref{equation6.4}.

Indeed, according to Lemma \ref{lemma6.3} and noting that the set of smooth points has full measure,
it is clear that the set of points satisfying the above 1)-5) has full measure on $B\times U.$\\

\noindent \emph{Step 3. Second variation of arc-length.}

In this step, we will study the second variation of the length of geodesics near the geodesic $x^oy^o$.

Since $M$ has curvature $\gs k$ and the geodesic $x^oy^o$ can be extended beyond $x^o$ and $y^o$,
by the Petrunin's second variation (Proposition \ref{para}),
there exists an isometry $T: T_{x^o}\to T_{y^o}$ and a subsequence of $\{\varepsilon_j\}_j$ given in Step 2,
 denoted by $\{\varepsilon_j\}_j$ again, such that
\begin{equation}\label{eq6.13}
\mathscr F_j(\eta)\ls -k|\eta|^2\cdot |x^oy^o|^2+ o(1)
\end{equation}
for any $\eta\in T_{x^o}$, where the function $\mathscr F_j$ is defined by
\begin{equation*}
\mathscr F_j(\eta):=
\frac{|\exp_{x^o}(\varepsilon_j\cdot \eta)\ \exp_{y^o}(\varepsilon_j\cdot T\eta)|^2-|x^oy^o|^2}{\varepsilon_j^2}
\end{equation*}
if $\eta\in T_{x^o}$ such that $\varepsilon_j\cdot \eta\in \mathscr W_{x^o}$ and
$\varepsilon_j\cdot T\eta\in \mathscr W_{y^o}$, and $ \mathscr F_j(\eta)\!:=0$ if otherwise.

Now we claim that
\begin{equation}\label{eq6.14}
\int_{B_o(1)}\mathscr F_j(\eta)dH^n(\eta)\ls \frac{-k\cdot\omega_{n-1}}{n+2}\cdot|x^oy^o|^2+ o(1).
\end{equation}
Indeed, by setting $z$ is the mid-point of $x^o$ and $y^o$ and   using  the semi-concavity of distance
 function ${\rm dist}_z$, we conclude
$$|z\exp_{x^o}(\varepsilon_j\cdot \eta)|\ls |zx^o|+\ip{\uparrow_{x^o}^z}{\eta}\cdot\varepsilon_j+\sigma_1\cdot |\eta|^2\cdot \varepsilon^2_j$$
 and
$$|z\exp_{y^o}(\varepsilon_j\cdot T\eta)|\ls |zy^o|+\ip{\uparrow_{y^o}^z}{T\eta}\cdot\varepsilon_j+\sigma_2\cdot|\eta|^2\cdot \varepsilon^2_j$$
for any $\eta\in T_{x^o}$ such that $\varepsilon_j\cdot\eta\in \mathscr W_{x^o}$ and $\varepsilon_j\cdot T\eta\in \mathscr W_{y^o}$,
 where $\sigma_1,\sigma_2$ are some positive
constants depending only on $|x^oz|,|y^oz|$ and $k$.
 By applying the triangle inequality
and $\uparrow_{y^o}^z\!=\!-\!T(\uparrow_{x^o}^z),$
we get (note that $|x^oz|\!=\!|y^oz|\!=\!|x^oy^o|/2$,)
\begin{equation*}
\begin{split}
\mathscr F_j(\eta)&\ls  \frac{\big(|z\exp_{x^o}(\varepsilon_j\cdot \eta)|+
|z\exp_{y^o}(\varepsilon_j\cdot T\eta)|\big)^2-|x^oy^o|^2}{\varepsilon^2_j}\\
  &\ls 2(\sigma_1+\sigma_2)\cdot|\eta|^2\cdot |x^oy^o| +(\sigma_1+\sigma_2)^2\cdot|\eta|^4\cdot \varepsilon_j^2\\
  &\ls \sigma_3
\end{split}
\end{equation*}
for any $\eta\in B_o(1)\subset T_{x^o}$, where $\sigma_3$ is some positive
constant depending only on $|x^oz|,|y^oz|$ and $k$. That is, $\mathscr F_j$ is bounded from above in  $B_o(1)$ uniformly.
According to Fatou's Lemma,   \eqref{eq6.13} implies
{\small $$\limsup_{j\to\infty}\int_{B_o(1)}\mathscr F_j(\eta)dH^n(\eta)\ls (-k)\int_{B_o(1)}|x^oy^o|^2|\eta|^2dH^n(\eta)=\frac{-k\cdot\omega_{n-1}}{n+2}\cdot|x^oy^o|^2.$$}
This is the desired \eqref{eq6.14}.
Therefore, by the definition of function $\mathscr F_j$, we have
\begin{equation}\label{eq6.15}
\begin{split}
\int_{B_o(\varepsilon_j)\cap\mathscr W}&\Big(|\exp_{x^o}(\hat \eta)\ \exp_{y^o}( T\hat\eta)|^2-|x^oy^o|^2\Big)dH^n(\hat\eta)\\
&   \overset{\hat\eta=\varepsilon_j\cdot \eta}{\doueq}\varepsilon^n_j\cdot \int_{B_o(1) }\varepsilon^2_j\cdot\mathscr F_j(\eta)dH^n(\eta)\\
& \ \ \ \ls\  \frac{-k\cdot \omega_{n-1}}{n+2} \cdot|x^oy^o|^2\cdot\varepsilon_j^{n+2} +\ o(\varepsilon_j^{n+2}),
\end{split}
\end{equation}
where $\mathscr W:=\mathscr W_{x^o}\cap T^{-1}(\mathscr W_{y^o})=\big\{v\in T_{x^o}:\ v\in \mathscr W_{x^o}\ \ {\rm and}\ \ Tv \in \mathscr W_{y^o}\big\}.$\\

\noindent {\it Step 4. Maximum principle via mean value inequalities.}

Let us fix the sequence of numbers $\{\varepsilon_j\}_j$ as in the above Step 2 and Step 3, and fix the isometry $T:\! T_{x^o}\!\to T_{y^o}$
and the set $\mathscr W\!:=\mathscr W_{x^o}\cap T^{-1}(\mathscr W_{y^o})$ as in Step 3.

Recall that in Step 2, we have proved that the function
$$H_3(x,y)=\frac{e^{-2nk\lambda}}{2t}\cdot|xy|^2-d_Y\big(u(x),u(y)\big)-v(x)+\widetilde{\gamma}_1(x)+\widetilde{\gamma}_2(y)$$
has a minimal point $(x^o,y^o)$ in the interior of $ B\times U$, where both $x^o$ and $y^o$ are smooth points, and the functions
\begin{equation*}
\begin{split}
 \widetilde{\gamma}_1(x)&:=\delta_0\cdot|\bar xx|^2+\delta_1\cdot h(x)+\frac{\delta_1}{8}|x^*x|^2+ \sum_{i=1}^nb_i\cdot g_i(x),\\
{\rm and}\quad\qquad \widetilde{\gamma}_2(y)&:=\delta_0\cdot|\bar yy|^2+\delta_1\cdot h(y)+\frac{\delta_1}{8}|y^*y|^2+ \sum_{i=n+1}^{2n}b_i\cdot g_i(y).\qquad\qquad
 \end{split}
 \end{equation*}

Consider the mean value
\begin{equation}\label{eq6.16}
\begin{split}
I(\varepsilon_j):&=\!\!\int_{B_o(\varepsilon_j)\cap\mathscr W}\Big[H_3\big(\exp_{x^o}(\eta),\exp_{y^o}(T\eta)\big)-H_3(x^o,y^o)\Big]dH^n(\eta)\\
&=I_1(\varepsilon_j)-I_2(\varepsilon_j)-I_3(\varepsilon_j)+I_4(\varepsilon_j)+I_5(\varepsilon_j),
\end{split}
\end{equation}
where
\begin{equation*}
\begin{split}
I_1(\varepsilon_j)&:=\frac{e^{-2nk\lambda}}{2t}\cdot
\int_{B_o(\varepsilon_j)\cap\mathscr W}\Big(|\exp_{x^o}( \eta)\ \exp_{y^o}( T\eta)|^2-|x^oy^o|^2\Big)dH^n(\eta),\\
I_2(\varepsilon_j)&:=\!
\int_{B_o(\varepsilon_j)\cap\mathscr W}\!\Big(d_Y\big(u(\exp_{x^o}( \eta)),u(\exp_{y^o}( T\eta)\big)\!-\!d_Y\big(u(x^o),u(y^o)\big)\!\Big)dH^n\!(\eta),
 \end{split}
 \end{equation*}
 \begin{equation*}
\begin{split}
I_3(\varepsilon_j)&:=
\int_{B_o(\varepsilon_j)\cap\mathscr W}\Big(v(\exp_{x^o}( \eta))-v(x^o)\Big)dH^n(\eta),\\
I_4(\varepsilon_j)&:=
\int_{B_o(\varepsilon_j)\cap\mathscr W}\Big(\widetilde{\gamma}_1(\exp_{x^o}( \eta))-\widetilde{\gamma}_1(x^o)\Big)dH^n(\eta),\\
I_5(\varepsilon_j)&:=
\int_{B_o(\varepsilon_j)\cap\mathscr W}\!\Big(\widetilde{\gamma}_2(\exp_{y^o}(T \eta))-\widetilde{\gamma}_2(y^o)\Big)dH^n(\eta).\qquad\qquad\qquad\qquad\ \ \
 \end{split}
 \end{equation*}

The minimal property of point $(x^o,y^o)$ implies that
\begin{equation}\label{eq6.17}
 I(\varepsilon_j)\gs0.
\end{equation}
We need to estimate $I_1,I_2,I_3, I_4$ and $I_5$. Recall that
the integration $I_1$ has been estimated by \eqref{eq6.15}.\\

\noindent (i) \emph{The estimate of} $I_2$.

By applying Lemma \ref{lem5.2} for points
$$ P=u(\exp_{x^o}(\eta)),\ \ Q=u(x^o),\ \ R=u(y^o)\ \ {\rm and}\ \ S=u(\exp_{y^o}(T\eta)),\quad $$
we get
\begin{equation}\label{equation6.18}
\begin{split}
\Big(d_Y&\big(u(\exp_{x^o}( \eta)),u(\exp_{y^o}( T\eta)\big)\!-\!d_Y\big(u(x^o),u(y^o)\big)\Big)\cdot d_Y\big(u(x^o),u(y^o)\big)\\
&\gs \big(d^2_{PQ_m}-d^2_{PQ}-d^2_{Q_mQ}\big)+\big(d^2_{SQ_m}-d^2_{SR}-d^2_{Q_mR}\big)\\
&=-w_{x^o,Q_m}\big(\exp_{x^o}(\eta)\big)-w_{y^o,Q_m}\big(\exp_{y^o}(T\eta)\big),
\end{split}
\end{equation}
where $Q_m$ the mid-point of $u(x^o)$ and $u(y^o)$, and the function $w_{z,Q_m}$ is defined in Lemma \ref{lemma6.3}, namely,
$$w_{z,Q_m}(\cdot):= d^2_Y\big(u(\cdot),u(z)\big) -d^2_Y\big(u(\cdot),Q_m\big)+ d^2_Y\big(Q_m,u(z)\big).$$

 Now we want to show that the set $\mathscr W\!:=\mathscr W_{x^o}\cap T^{-1}(\mathscr W_{y^o})$ satisfies   \eqref{equation6.4}.
 Since both points $x^o$ and $y^o$ are smooth, by   \eqref{eq2.3} in Lemma \ref{smooth}, we have
$$\frac{H^n\big(\mathscr W_{x^o}\!\cap\! B_o(s)\big)}{H^n\big(B_o(s)\!\subset\! T_{x^o}\big)}\gs1\!-\!o(s)\quad{\rm and}\quad
 \frac{H^n\big(\mathscr W_{y^o}\!\cap\! B_o(s)\big)}{H^n\big(B_o(s)\!\subset\! T_{y^o}\big)} \gs 1\!-\!o(s).$$
Note that $T: T_{x^o}\to T_{y^o}$ is an isometry (with $T(o)=o$). We can get
\begin{equation}
\label{equation6++}
\frac{H^n\big(\mathscr W \!\cap\! B_o(s)\big)}{H^n\big(B_o(s)\!\subset\! T_{x^o}\big)}=\frac{H^n\big(\mathscr W_{x^o}\!\cap\! T^{-1}(\mathscr W_{y^o})\!\cap\! B_o(s)\big)}{H^n\big(B_o(s)\!\subset\! T_{x^o}\big)}\gs1\!-\!o(s).
\end{equation}
In particular, by taking $s=\varepsilon_j$, we have that the set $\mathscr W$ satisfies   \eqref{equation6.4}.

Now by integrating equation \eqref{equation6.18} on $B_o(\varepsilon_j)\cap \mathscr W$
and using Lemma \ref{lemma6.3}, we have
\begin{equation*}
\begin{split}
 d_Y\big(u(x^o),u(y^o)\big)\cdot I_2(\varepsilon_j)&\gs-\int_{B_o(\varepsilon_j)\cap\mathscr W}\!\!w_{x^o,Q_m}\big(\exp_{x^o}(\eta)\big)dH^n(\eta)\\
 &\quad-\int_{B_o(\varepsilon_j)\cap\mathscr W}\!\!w_{y^o,Q_m}\big(\exp_{y^o}(T\eta)\big)dH^n(\eta)\\
 &\gs -o(\varepsilon^{n+2}_j).
\end{split}
\end{equation*}
 Here the last inequality comes from Lemma \ref{lemma6.3}.
If $d_Y\big(u(x^o),u(y^o)\big)\not=0$, then this inequality implies that
\begin{equation}\label{equation6.19}
I_2(\varepsilon_j)\gs -o(\varepsilon^{n+2}_j).
\end{equation}
If $d_Y\big(u(x^o),u(y^o)\big)=0$, then it is simply implied by the definition of $I_2$ that $I_2(\varepsilon_j)\gs0$ for all $j\in \mathbb N$.
Hence, the estimate \eqref{equation6.19} always holds.\\

\noindent (ii) \emph{The estimate of} $I_3$.

 By setting the  function
$$g(x):=v(x^o)-v(x)$$
 on $B$, we have $g(x^o)=0$ and
 $$\mathscr L_g=-\mathscr L_v=\Big(e^{-2nk\lambda}\cdot\frac{nk}{t}L^2_{t,\lambda}-\theta_0\Big)\cdot\rv \quad \ {\rm on}\ \ B.$$
Recall $L_{t,\lambda}\in L^\infty(B)$ (see Lemma \ref{lem6.2}(ii)). By Lemma \ref{lem3.1}, we know that
 $g$ is locally Lipschitz on $B$. Fix some $r_0>0$ such that $B_{x^o}(r_0)\subset\subset B$, and denote by $c_0$
 the Lipschitz constant of $g$ on $B_{x^o}(r_0)$.

Take any $s<r_0.$ Noticing that $g(x^o)=0$, we have that $g(x)+c_0s\gs0$ in $ B_{x^o}(s)$. By using Proposition \ref{mean}, we have
\begin{equation*}
\begin{split}
\frac{1}{H^{n-1}(\partial B_o(s)\subset T^{k}_{x^o})}&\int_{\partial B_{x^o}(s)}\big(g(x)+c_0s\big)d\rv\\
\ls\big(&g(x^o)+c_0s \big) +\frac{e^{-2nk\lambda}\cdot\frac{nk}{t}L^2_{t,\lambda}(x^o)-\theta_0}{2n}s^2+o(s^{2}).
\end{split}
\end{equation*}
So, we get (notice that $g(x^o)=0$ )
\begin{equation*}
\begin{split}
\int_{\partial  B_{x^o}(s)} g(x) d\rv\ls &c_0s\cdot\Big( H^{n-1}(\partial B_o(s)\!\subset\! T^k_{x^o})-\rv(\partial  B_{x^o}(s))\Big)\\
&  +\Big(e^{-2nk\lambda}\cdot\frac{k}{2t}L^2_{t,\lambda}(x^o)-\frac{\theta_0}{2n}\Big)s^2\cdot H^{n-1}(\partial B_o(s)\!\subset\! T^k_{x^o})\! +\!o(s^{n+1}).
\end{split}
\end{equation*}
Notice that Bishop volume comparison theorem implies  $\rv(\partial  B_{x^o}(s))\ls H^{n-1}(\partial B_o(s)\!\subset\! T^k_{x^o})$. We can use  co-area formula to obtain
\begin{equation}\label{eq6.19}
\begin{split}
\int_{B_{x^o}(s)}\! g(x) d\rv\ls& c_0s\cdot\Big( H^{n}( B_o(s)\subset T^k_{x^o})-\rv(  B_{x^o}(s))\Big)\\
 & +\Big(e^{-2nk\lambda}\cdot\frac{k}{2t}L^2_{t,\lambda}(x^o)-\frac{\theta_0}{2n}\Big)\! \int^s_0\!\tau^2\cdot H^{n-1}\big(\partial B_o(\tau)\!\subset\! T^k_{x^o}\big)d\tau\\
 & +o(s^{n+2}).
 \end{split}
\end{equation}
Because that $x^o$ is a smooth point, we can apply Lemma \ref{smooth} to conclude
\begin{equation}\label{eq6.20}
\big|H^{n}\big( B_o(s) \subset T_{x_0}\big)-\rv\big( B_{x_0}(s)\big)\big|\ls o(s)\cdot H^{n}\big( B_o(s) \subset T_{x_0}\big)=o(s^{n+1}).
\end{equation}
On the other hand,  the fact that $x^o$ is smooth also implies that $T^k_{x^o}$ is isometric to $\mathbb M_k^n$, and hence
$$ \big|H^{n}\big( B_o(s) \subset T^k_{x_0}\big)- H^{n}\big( B_o(s) \subset T_{x_0}\big)\big|=O(s^{n+2})$$
and
$$H^{n-1}\big(\partial B_o(\tau)\!\subset\! T^k_{x^o}\big)=\omega_{n-1}\cdot \Big(\frac{\sinh(\sqrt{-k}\tau)}{\sqrt{-k}}\Big)^{n-1}=\omega_{n-1}\cdot \tau^{n-1}+O(\tau^{n+1}). $$
Thus, by substituting this and \eqref{eq6.20} into \eqref{eq6.19}, we can get
\begin{equation}\label{eq6.21}
\int_{B_{x^o}(s)} g(x) d\rv\ls
 \Big(e^{-2nk\lambda}\cdot\frac{k}{2t}L^2_{t,\lambda}(x^o)-\frac{\theta_0}{2n}\Big)\cdot\frac{\omega_{n-1}}{n+2}\cdot s^{n+2}  +o(s^{n+2}).
\end{equation}

Next we want to show that
\begin{equation}\label{eq6.22}
\begin{split}
 &\int_{B_{o}(s)\cap\mathscr W}\!\!g(\exp_{x^o}(\eta))dH^n(\eta)\ls \int_{B_{x^o}(s)}\!g(x)d\rv(x)+ o(s^{n+2})
\end{split}
 \end{equation}
 for all $0<s<r_0$.

Since $x^o$ is a smooth point, we can use Lemma \ref{smooth} to obtain
\begin{equation}\label{eq6.23}
\begin{split}
 &\int_{B_{o}(s)\cap\mathscr W_{x^o}}\!\!g(\exp_{x^o}(\eta))dH^n(\eta)\\
&\qquad=\int_{B_{x^o}(s)\cap W_{x^o}}\!\!g(x)(1+o(s))d\rv(x) \\
 &\qquad\ls \int_{B_{x^o}(s)}\!\!g(x)d\rv(x)+\int_{B_{x^o}(s)}\!\!|g(x)|\cdot o(s)d\rv(x)\\
 &\qquad\ls \int_{B_{x^o}(s)}\!\!g(x)d\rv(x)+\int_{B_{x^o}(s)}\!\!O(s)\cdot o(s)d\rv(x)\\
 &\qquad\qquad\ \ ({\rm since}\ g(x)\ {\rm is\ Lipschitz\ continuous\ in}\ B_{x^o}(s)\ {\rm and} \ g(x^o)=0.)\\
 &\qquad=\int_{B_{x^o}(s)}\!\!g(x)d\rv(x)+ o(s^{n+2})
\end{split}
 \end{equation}
 for all $0<s<r_0$, where we have used that $W_{x^o}$ has full measure (please see \S2.2).
\begin{equation}\label{eq6.24}
\begin{split}
 \int_{B_{o}(s)\cap\mathscr W}\!\!g(\exp_{x^o}(\eta))dH^n(\eta)&- \int_{B_{o}(s)\cap\mathscr W_{x^o}}\!\!g(\exp_{x^o}(\eta))dH^n(\eta)\\
   &\ls  \int_{B_{o}(s)\cap(\mathscr W_{x^o}\backslash\mathscr W)}\!\!|g(\exp_{x^o}(\eta))|dH^n(\eta)\\
 &\ls O(s)\cdot \rv\big(B_{o}(s)\cap(\mathscr W_{x^o}\backslash\mathscr W)\big)
\end{split}
 \end{equation}
 for all $0<s<r_0$. Here we have used the fact that $g$ is Lipschitz continuous in $ B_{x^o}(s)$ and $g(x^o)=0$ again.
 Recall   \eqref{equation6++} in the previous estimate for $I_2$. We have
\begin{equation*}\begin{split}
\rv\big(B_{o}(s)\cap(\mathscr W_{x^o}\backslash\mathscr W)\big)&\ls\rv\big(B_{o}(s) \backslash\mathscr W\big)
\overset{\eqref{equation6++}}{\ls} o(s)\cdot \rv\big(B_{o}(s)\!\subset\! T_{x^o}\big)\\
&\ls o(s^{n+1}).
\end{split}
\end{equation*}
By combining this with    \eqref{eq6.23}--\eqref{eq6.24}, we conclude  the desired estimate \eqref{eq6.22}.

By taking $s=\varepsilon_j$ and  using \eqref{eq6.21}--\eqref{eq6.22}, we obtain the estimate of $I_3$
\begin{equation} \label{eq6.27}
\begin{split}
-I_3(\varepsilon_j)&= \int_{B_{o}(\varepsilon_j)\cap\mathscr W}\!\!g(\exp_{x^o}(\eta))dH^n(\eta)\\
&\ls \Big(e^{-2nk\lambda}\cdot\frac{k}{2t}L^2_{t,\lambda}(x^o)-\frac{\theta_0}{2n}\Big)\cdot\frac{\omega_{n-1}}{n+2} \cdot \varepsilon_j^{n+2}  -o(\varepsilon_j^{n+2}),
\quad \forall j\in \mathbb N.
\end{split}
\end{equation}
$\ $

\noindent (iii) \emph{The estimate of} $I_4$ \emph{and} $I_5$.

Because all of the integrated functions in $I_4$ and $I_5$ are semi-concave,  we consider the following sublemma.
\begin{slem}\label{slem6.7}
Let $\sigma\in \mathbb R$  and let $f$ be a $\sigma$-concave function near a smooth point $z$. Then
$$\int_{(B_o(s)\cap\mathscr W_1)\subset T_z}\big(f(\exp_z(\eta))-f(z)\big)dH^n(\eta)\ls\frac{\omega_{n-1}}{2(n+2)}\cdot\sigma\cdot s^{n+2} + o(s^{n+2}) $$
for any subset $\mathscr W_1\subset \mathscr W_z\subset T_z$ with $H^n(B_o(s)\backslash\mathscr W_1 )\ls o(s^{n+1})$.
\end{slem}
\begin{proof}
Since $f$ is $\sigma$-concave near $z$, we have
$$f(\exp_z(\eta))-f(z)\ls d_zf(\eta)+\frac{\sigma}{2}|\eta|^2$$
for all $\eta\in   \mathscr W_z$.
The integration on $B_o(s)\cap \mathscr W_1$ tells us
\begin{equation}\label{eq6.28}
\int_{B_o(s)\cap\mathscr W_1 }  \!\!\big(f(\exp_z(\eta))-f(z)\big)dH^n \ls \int_{B_o(s)\cap\mathscr W_1 }  \!\!
\big( d_zf(\eta)+\frac{\sigma}{2}|\eta|^2\big)dH^n.
\end{equation}
Because $f$ is semi-concave function, we have
 $\int_{B_o(s)}d_zf(\eta)dH^n\ls0$ (see Proposition 3.1 of \cite{zz12}).
Thus,
\begin{equation*}
\begin{split}
\int_{B_o(s)\cap\mathscr W_1 } \!\!
  d_zf(\eta)dH^n&\ls -\int_{B_o(s)\backslash\mathscr W_1 } \!\!
  d_zf(\eta)dH^n\ls \max_{B_o(s)}|d_zf(\eta)|\cdot H^n(B_o(s)\backslash\mathscr W_1)\\
  &\ls O(s)\cdot o(s^{n+1})=o(s^{n+2}).
\end{split}
\end{equation*}
Similarly, we have
\begin{equation*}
\begin{split}
\int_{B_o(s)\cap\mathscr W_1 }\!
  |\eta|^2dH^n&=\int_{B_o(s)}\!
  |\eta|^2dH^n-\int_{B_o(s)\backslash\mathscr W _1}\!
  |\eta|^2dH^n\\
  &=\int_0^s t^2\cdot\omega_{n-1}\cdot t^{n-1}dt-\int_{B_o(s)\backslash\mathscr W_1 } \!
  |\eta|^2dH^n\\
  &\qquad\qquad\quad ({\rm because}\ \ z \ \ {\rm is\ smooth} )\\
  &=\frac{\omega_{n-1}\cdot s^{n+2}}{n+2}+O(s^2)\cdot o(s^{n+1})\\
 &\qquad\qquad\quad ({\rm because}\ \  0\ls H^n(B_o(s)\backslash\mathscr W_1 )\ls o(s^{n+1}) ).
\end{split}
\end{equation*}
Substituting the above two inequalities into equation \eqref{eq6.28}, we have
$$\int_{B_o(s)\cap\mathscr W_1 } \!\!\big(f(\exp_z(\eta))-f(z)\big)dH^n
 \ls \frac{\omega_{n-1}\cdot\sigma}{2(n+2)}\cdot s^{n+2} + o(s^{n+2}).$$
This completes the proof of the sublemma.
\end{proof}

Now let us use the sublemma to estimate $I_4$ and $I_5$.

Note that $M$ has curvature $\gs k$ implies that the function
 ${\rm dist}^2_q(x):=|qx|^2$  is $2(\sqrt{-k}|qx|\cdot\coth(\sqrt{-k}|qx|))$-concave for all $q\in M$. For all $q,x\in U$, we have
 $$ 2\sqrt{- k}|qx|\cdot\coth(\sqrt{-k}|qx|)\ls 2(1+\sqrt{-k}|qx|)\ls 2+2\sqrt{-k}\cdot{\rm diam}(U):=C_{k,U}.$$
By combining with that $h$ is (-1)-concave  and that $g_i(x)$ is concave for any
$1\ls i\ls n$, we know that the function $\widetilde{\gamma}_1$ is $(\delta_0\cdot C_{k,U}-\delta_1+\delta_1\cdot C_{k,U}/8)$-concave.
Recall that the equation \eqref{equation6++} implies
$$ H^n(B_o(s)\backslash\mathscr W)\ls o(s)\cdot  \rv\big(B_{o}(s)\!\subset\! T^k_{x^o}\big)=o(s^{n+1}).$$
According to Sublemma \ref{slem6.7}, we obtain (by setting $s=\varepsilon_j$)
\begin{equation}\label{eq6.29}
I_4(\varepsilon_j)\ls\kappa(\delta_0,\delta_1)\cdot \frac{\omega_{n-1}}{2(n+2)}\cdot \varepsilon_j^{n+2}+ o(\varepsilon_j^{n+2}), \quad \forall j\in \mathbb N,
\end{equation}
where
$$\kappa(\delta_0,\delta_1):=(\delta_0\cdot C_{k,U}-\delta_1+\delta_1\cdot C_{k,U}).$$
Since the map $T$ is an isometry, the same estimate holds for $I_5$. Namely,
\begin{equation}\label{eq6.30}
I_5(\varepsilon_j)\ls\kappa(\delta_0,\delta_1)\cdot \frac{\omega_{n-1}}{2(n+2)}\cdot \varepsilon_j^{n+2}+ o(\varepsilon_j^{n+2}), \quad \forall j\in \mathbb N.
\end{equation}

Let us recall the equation \eqref{eq6.16}, \eqref{eq6.17} and combine all of estimates from $I_1$ to $I_5$.
 That is, the equations \eqref{eq6.15}, \eqref{equation6.19}, \eqref{eq6.27}, \eqref{eq6.29} and \eqref{eq6.30}.  We obtain
 \begin{equation*}
 \begin{split}
 0\ls&\bigg[\frac{-k\!\cdot\! e^{-2nk\lambda}}{t} |x^oy^o|^2+\frac{e^{-2nk\lambda}\!\cdot\! k}{t}L^2_{t,\lambda}(x^o)-\frac{\theta_0}{n}+2\kappa(\delta_0,\delta_1) \bigg] \frac{\omega_{n-1}}{2(n+2)}\cdot\! \varepsilon_j^{n+2}\\
 &+  o(\varepsilon_j^{n+2}).
 \end{split}
 \end{equation*}
 Thus,
\begin{equation}\label{eq6.31}
\frac{-k\!\cdot\! e^{-2nk\lambda}}{t} \Big(|x^oy^o|^2-L^2_{t,\lambda}(x^o)\Big)-\frac{\theta_0}{n}+2\kappa(\delta_0,\delta_1)\gs0.
\end{equation}

Recall that in Step 2, we have $H_3(x,y)$ converges to $H_0(x,y)$ as $\delta_1$ and $b_i$ tends to $0^+$, $1\ls i\ls 2n$.  Notice that the  point $(\bar x,\bar y)$ is the \emph{unique}
minimum of $H_0$, we  conclude that $(x^o,y^o)$ converges to $(\bar x,\bar y)$ as $\delta_1\!\to\!0^+$ and $b_i\!\to\!0^+$, $1\ls i\ls 2n$. Hence,
letting $\delta_1\!\to\!0^+$ and $b_i\!\to\!0^+$, $1\ls i\ls 2n$, in \eqref{eq6.31}, we obtain
\begin{equation}\label{eq6.32}
\frac{-k\!\cdot\! e^{-2nk\lambda}}{t}\Big(|\bar x\bar y|^2- \liminf_{\delta_1\to0^+,\ b_i\to0^+}L^2_{t,\lambda}(x^o)\Big)-\frac{\theta_0}{n}+2\cdot \delta_0\cdot C_{k,U}\gs0.
\end{equation}
On the other hand,
by the lower semi-continuity of
$L_{t,\lambda}$ (from Lemma \ref{lem6.2}(i)), we have
$$\liminf_{\delta_1\to0^+,\ b_i\to0^+}L_{t,\lambda}(x^o)\gs L_{t,\lambda}(\bar x).$$
Therefore, by combining with \eqref{eq6.32}, \eqref{eq6.10} and the fact $-k\gs0$, we have
 \begin{equation*}
0\ls -\frac{\theta_0}{n}+2\cdot \delta_0\cdot C_{k,U}=-\frac{\theta_0}{n}+4\cdot \delta_0\cdot \big(1+\sqrt{-k}\cdot{\rm diam}(U)\big).
\end{equation*}
This contradicts with \eqref{eq6.11} and completes the proof of the \textbf{Claim}, and hence that of  the lemma.
\end{proof}

\begin{cor}\label{cor6.8}
Given any domain $\Omega''\subset\subset\Omega'$, there exits a constant $t_1>0$ such that,
 for each $t\in(0,t_1)$ and each $\lambda\in[0,1]$, the function  $x\mapsto f_t(x,\lambda)$ is a
super-solution of the Poisson equation \eqref{eq6.9} on $\Omega''$.
\end{cor}
\begin{proof}
For any $p\in \Omega'$, by Lemma \ref{lem6.6}, there exists  a neighborhood $B_p(R_p)$ and a number $t_p>0$ such that the function $f_t(\cdot,\lambda)$ is a
super-solution of the Poisson equation \eqref{eq6.9} on $B_p(R_p)$, for each $t\in(0,t_p)$ and $\lambda\in[0,1]$.

Given any  $\Omega'' \subset\subset\Omega'$, we have $\overline{\Omega''}\subset \cup_{p\in\Omega'} B_p(R_p/2).$ Since $\overline{\Omega''}$ is compact, there exist finite
$p_1,p_2,\cdots, p_N$ such that $\overline{\Omega''}\subset \cup_{1\ls j\ls N} B_{p_j}(R_{p_j}/2).$ By the standard construction for partition of unity, there exist Lipschitz functions $0\ls \chi_j\ls 1$ on $\Omega'$ with ${\rm supp}\chi_j\subset B_{p_j}(R_{p_j})$ for each $j=1,2,\cdots, N$ and $\sum_{j=1}^N\chi(x)=1$ on $\Omega''$.

Take any nonnegative $\phi\in Lip_0(\Omega'')$. Then $\chi_j\phi\in Lip_0(B_{p_j}(R_{p_j}))$ for each $j=1,2,\cdot, N.$  We thus obtain
{\small \begin{equation*}
\begin{split}
\int_{\Omega''}\ip{\nabla f_t(\cdot,\lambda)}{\nabla \phi}&\rv=\mathscr L_{f_t(\cdot,\lambda)}(\sum_{j=1}^N\chi_j\cdot\phi)=\sum_{j=1}^N\mathscr L_{f_t(\cdot,\lambda)}(\chi_j\cdot\phi)\\
\ls \sum_{j=1}^N\int_{U_{p_j}}&e^{-2nk\lambda}\cdot\frac{-nk}{t}L^2_{t,\lambda}\cdot(\chi_j\cdot \phi)\rv
=\int_{\Omega''}e^{-2nk\lambda}\cdot\frac{-nk}{t}L^2_{t,\lambda} \cdot \phi\rv.
\end{split}
\end{equation*}
}
This completes the proof of the corollary.
\end{proof}

In the following we want to show that the function $f_t(\cdot,\cdot)$ satisfies a parabolic differential inequality $\mathscr L_{f_t(x,\lambda)}\ls \partial f_t/\partial\lambda$.

Given a domain $G\subset M$ and an interval $I=(a,b)$, then $Q=G\times I$ is called a \emph{parabolic cylinder} in space-time $M\times \mathbb R$.
For a parabolic cylinder $Q$, we equip with the product measure
$$\underline{\nu}:=\rv\times\mathcal L^1.$$
When $G=B_{x_0}(r)$ and $I=I_{\lambda_0}(r^2)\!:=(\lambda_0-r^2,\lambda_0+r^2)$, we denote by the cylinder
\begin{equation*}
Q_r(x_0,\lambda_0):= B_{x_0}(r)\times I_{\lambda_0}(r^2).
\end{equation*}
If without confusion arises, we shall write it as $Q_r$.

The theory for \emph{local weak solution} of the heat equation  on metric spaces has been developed by  Sturm in \cite{st95} and, recently,
 by Kinnunen-Masson \cite{km14}, Marola-Masson \cite{mm13}.
 According to Lemma \ref{lem6.1}(iv), our auxiliary functions $f_t(x,\lambda)$ are in $W^{1,2}(\Omega''\times(0,1))$.
  So we  consider only the weak solution in $W^{1,2}_{\rm loc}(Q)$. In such a case, the definition of weak solution
 of the heat equation can be simplified as follows.
\begin{defn}\label{weak-heat}
Let $Q=G\times I$ be a cylinder.  A function $g(x,\lambda)\in W^{1,2}_{\rm loc}(Q)$ is said a (weak) \emph{super-solution} of the heat equation
\begin{equation}\label{heat}
\mathscr L_{g}= \frac{\partial g}{\partial \lambda}\quad {\rm on}\ \ Q,
\end{equation}
if it satisfies
$$-\int_{Q}\!\!\ip{\nabla g}{\nabla \phi}d\underline{\nu}(x,\lambda)\ls \int_{Q}\frac{\partial g}{\partial \lambda}\cdot\phi d\underline{\nu}(x,\lambda) $$
for all nonnegative function  $\phi\in Lip_0(Q)$.

 A function $g(x,\lambda)$ is said a \emph{sub-solution} of the equation \eqref{heat} on $Q$ if $-g(x,\lambda)$ is a super-solution
 on $Q$. A function $g(x,\lambda)$ is said a \emph{local weak solution} of the equation \eqref{heat} on $Q$ if it is both sub-solution and super-solution
 on $Q$.
\end{defn}
\begin{rem}
The test functions $\phi$ in the above Definition \ref{weak-heat} also can be chosen in $Lip(Q)$ such that, for each $\lambda\in I$,
the function $  \phi(\cdot,\lambda)$ is in $Lip_0(G).$ That is, it vanishes only on the lateral boundary $\partial G\times I$.
\end{rem}

\begin{lem}\label{lem6.11}
Let $Q=G\times I$ be a cylinder. Suppose a function $g(x,\lambda)\in W^{1,2}_{\rm loc}(Q)$.
 If, for almost all $\lambda\in I$, the function $x\mapsto g(x,\lambda)$ is a super-solution of the Poisson equation
 \begin{equation}\label{eq6.34}
\mathscr L_{g}=\frac{\partial g}{\partial \lambda}\cdot \rv \quad {\rm on}\ \ G.
\end{equation}
Then $g(x,\lambda)$ is a super-solution  of the heat equation
 \begin{equation*}
\mathscr L_{g}=\frac{\partial g}{\partial \lambda}  \quad {\rm on}\ \ Q.
\end{equation*}
\end{lem}
\begin{proof}
 Take any nonnegative function $\phi(x,\lambda)\in Lip_0(Q)$. Then, for each $\lambda\in I$, the function $\phi(\cdot,\lambda)$ is in $Lip_0(G).$
  For almost all $\lambda\in I$, since the function $g(\cdot,\lambda)$ is a super-solution of the Poisson equation \eqref{eq6.34} on  $G$, we have
 \begin{equation}\label{eq6.35}
 -\int_{G}\ip{\nabla g}{\nabla \phi}d\rv=\int_{G}\phi d\mathscr L_{g} \ls \int_{G}\phi\cdot \frac{\partial g}{\partial \lambda}d\rv.
 \end{equation}

Notice that $g(x,\lambda)\in W_{\rm loc}^{1,2}(Q)$ and $\phi(x,\lambda)\in Lip_0(Q)$, we know that $|\ip{\nabla g}{\nabla \phi}|\in L^2(Q)$ and that $\phi\cdot\frac{\partial g}{\partial\lambda}\in L^2(Q).$  By using Fubini Theorem, we obtain
\begin{equation*}
\begin{split}
-\int_{G\times I}\ip{\nabla g(x,\lambda)}{\nabla \phi(x,\lambda)}d\underline{\nu}(x,\lambda)&
=-\int_I\int_{G}\ip{\nabla g}{\nabla \phi}d\rv d\lambda
\\\overset{\eqref{eq6.35}}{\ls}\int_I \int_{G}\phi \cdot \frac{\partial g}{\partial \lambda} d\rv d\lambda&=\int_{G\times I}\phi \cdot \frac{\partial g}{\partial \lambda} d\underline{\nu}(x,\lambda).
\end{split}
\end{equation*}
Thus, $g(x,\lambda)$ is a super-solution of the heat equation
$\mathscr L_{g}=\frac{\partial g}{\partial \lambda}$  on $Q$.
\end{proof}

Now we are ready to show that the function $(x,\lambda)\mapsto f_t(x,\lambda)$ is a super-solution of the heat equation.
\begin{prop}\label{prop6.12}
Given any  $\Omega''\subset\subset \Omega'$, and let $t_*:=\min\{t_0,t_1\}$, where $t_0$ is given in Lemma \ref{lem6.1}, and $t_1$ is given in Corollary \ref{cor6.8}.
Then, for each $t\in(0,t_*)$,
 the function $(x,\lambda)\mapsto f_t(x,\lambda)$ is a super-solution of
\begin{equation}\label{eq6.36}
\mathscr L_{f_t(x,\lambda)}  = \frac{\partial f_t(x,\lambda)}{\partial\lambda}
\end{equation}
 on the cylinder $\Omega''\times(0,1)$.
 \end{prop}
\begin{proof}
From Lemma \ref{lem6.1}(iv), we know that $f_t(x,\lambda)\in W^{1,2}(\Omega''\times(0,1))$ for all $t\in(0,t_*).$
According to Corollary \ref{cor6.8}, for each $\lambda\in[0,1]$, the function $f_{t}(\cdot,\lambda)$ is a super-solution of the Poisson equation
$$\mathscr L_{f_t(\cdot,\lambda)}=-e^{-2nk\lambda}\cdot\frac{nk}{t}L^2_{t,\lambda}\cdot\rv\qquad {\rm on }\ \ \Omega''. $$
On the other hand, by Lemma \ref{lem6.3}, we have
\begin{equation}\label{eq6.37}
\frac{\partial f_t(x,\lambda)}{\partial \lambda}\gs- e^{-2nk\lambda}\cdot\frac{nk}{t}L^2_{t,\lambda}(x)\qquad\qquad
\end{equation}
for $\underline{\nu}$-a.e.  $(x,\lambda)\in \Omega''\times(0,1)$. We know that $\frac{\partial f_t}{\partial \lambda}\in L^2(\Omega''\times(0,1))$ from Lemma \ref{lem6.1}(iv). By Fubini's theorem, we get that, for almost all $\lambda\in(0,1)$,  the equation \eqref{eq6.37} holds for almost all $x\in \Omega''.$  Hence, for almost all $\lambda\in(0,1)$, we have $$\mathscr L_{f_t(\cdot,\lambda)}\ls\frac{\partial f_t(x,\lambda)}{\partial \lambda}\cdot\rv\qquad {\rm on }\ \ \Omega''. $$
Therefore, the proposition follows from Lemma \ref{lem6.11}.
\end{proof}

\subsection{Lipschitz continuity of harmonic maps}$\ $

In this subsection, we will prove our main Theorem  \ref{main-thm}.

 We need the following weak Harnack inequality for sub-solutions of the heat equation (See Theorem 2.1 \cite{st95} or Lemma 4.2 \cite{mm13})
\begin{lem}[\cite{st95,mm13}]\label{lem6.13}
Let $G\times I$ be a parabolic cylinder in $M\times \mathbb R$, and let $g(x,\lambda)$ be a nonnegative, local bounded sub-solution of the heat equation $\mathscr L_{g}=\frac{\partial g}{\partial \lambda}$  on $Q_r\subset  G\times I$. Then there exists a constant $C=C(n,k,{\rm diam}G)$, depending only on $n,k$ and ${\rm diam}G$,  such that we have
\begin{equation}\label{eq6.38}
{\rm ess}\sup_{Q_{r/2}}g\ls \frac{C}{r^2\cdot \rv\big(B_x(r)\big)}\int_{Q_{r}}gd\underline{\nu}.
\end{equation}
\end{lem}

 Fix any domain $\Omega'\subset\subset\Omega$.
 For any $t>0$ and any $0\ls \lambda\ls 1$,  the function $f_t(x,\lambda)$  is given in \eqref{eq6.1}. 
Notice that
\begin{equation}\label{eq6.39}
0\ls -f_t(x,\lambda)\ls{\rm osc}_{\overline{\Omega'}}u.
\end{equation}

The following lemma is essentially a consequence of the above weak Harnack inequality.
\begin{lem}\label{lem6.14} Let $R\ls1$ and let ball $B_q(2R)\subset\subset \Omega'$. Suppose that  $t_*$ is given in Proposition \ref{prop6.12} for $\Omega''=B_q(2R)$.
For each $t\in(0,t_*)$ and $\lambda\in(0,1)$, we define the function $x\to |\nabla^-f_t(x,\lambda)|$ on $B_q(2R)$ by
\begin{equation}\label{eq6.40}
|\nabla^-f_t(x,\lambda)|\!:=\limsup_{r\to0}\sup_{y\in B_x(r)}\frac{\big(f_t(x,\lambda)-f_t(y,\lambda)\big)_+}{r}\qquad \forall x\in B_q(2R),
\end{equation}
where $a_+=\max\{a,0\}.$

Then, there exists a constant $C_1(n,k,R)$  such that
\begin{equation}\label{eq6.41}
\frac{1}{\rv\big(B_q(R)\big)}\int_{B_q(R)\times(\frac{1}{4},\frac{3}{4})}|\nabla^-f_t(x,\lambda)|^2d\underline{\nu}\ls C_1(n,k,R)\cdot{\rm osc}_{\overline{\Omega'}}^2 u
\end{equation}
holds for all $t\in(0,t_*)$.
\end{lem}
\begin{proof}
\textbf{1.} First, let us consider   an arbitrary function $h\in W^{1,2}\big(B_q(R)\big)$. Take any $\Omega_1\subset\subset B_q(R)$.
According to the Theorem 3.2 of \cite{hk00},
 there exists a constant $\bar C=\bar C(\Omega_1,B_q(R))$ such that for almost all
 $x,y\in \Omega_1$ with $|xy|\ls {\rm dist}(\Omega_1,\partial B_q(R))/\bar C$,
  we have
$$|h(x)-h(y)|\ls |xy|\cdot\Big(M(|\nabla h|)(x)+M(|\nabla h|)(y)\Big),$$
where $Mw$ is the Hardy-Littlewood maximal function for the function $w\in L^1_{\rm loc}(B_q(R))$
$$ Mw(x)=\sup_{s>0}\frac{1}{\rv(B_x(s))}\int_{B_x(s)\cap B_q(R)}|w|d\rv.$$
 Hence, for almost all $x\in \Omega_1$, we have
\begin{equation}\label{eq6.42}
\begin{split}
 \fint_{B_{x}(r)}|h(x)-h(y)|d\rv(y)
\ls& r \cdot \fint_{B_{x}(r)}\Big(M(|\nabla h|)(x)+M(|\nabla h|)(y)\Big)d\rv(y)\\
\ls& r \cdot \Big(M(|\nabla h|)(x)+M[(M(|\nabla h|)](x)\Big)
 \end{split}
 \end{equation}
for any  $r<{\rm dist}(\Omega_1,\partial B_q(R))/\bar C$.\\

\noindent\textbf{2.} Fix any $t\in(0,t_*)$. We first introduce a function $F(x,\lambda)$ on  $ B_q(R)\times(0,1)$ as
\begin{equation*}
F(x,\lambda):=\limsup_{r\to0}\frac{1}{r}\cdot \fint_{I_\lambda(r^2)}\fint_{B_x(r)}\big|f_t(x,\lambda)-f_t(x',\lambda')\big|d\rv(x')d\lambda'
\end{equation*}
for any $(x,\lambda)\in\ B_q(R)\times(0,1)$, where $I_\lambda(r^2)=(\lambda-r^2,\lambda+r^2)$. We claim that there exists a constant $C_2(n,k,R)$
such that
\begin{equation}\label{eq6.43}
\int_{ B_q(R)}F^2(x,\lambda)d\rv(x)\ls C_2(n,k,R)\cdot\int_{ B_q(R)}|\nabla f_t(x,\lambda)|^2d\rv(x)
\end{equation}
holds for all $\lambda\in(0,1)$.

To justify this, let us fix any $\lambda\in(0,1)$.
According to Lemma \ref{lem6.1}(ii), we have
$f_t(\cdot,\lambda) \in W^{1,2}\big(B_q(R)\big).$
Take any $\Omega_1\subset\subset  B_q(R)$.
By using \eqref{eq6.42}  to the function $f_t(\cdot,\lambda) $, we obtain that, for almost all $x\in \Omega_1$,
\begin{equation}\label{eq6.44}
\fint_{B_{x}(r)}\!|f_t(x,\lambda) -f_t(x',\lambda)| d\rv(x')
 \ls r \cdot \Big(M(|\nabla f_t(\cdot,\lambda) |)(x)+M[(M(|\nabla f_t(\cdot,\lambda) |)](x)\Big)
  \end{equation}
for all $r<{\rm dist}(\Omega_1,\partial  B_q(R))/\bar C(\Omega_1, B_q(R))$. Thus, for almost all $x\in \Omega_1$, we can use Lemma \ref{lem6.1}(iii) to conclude
{\small \begin{equation*}
\begin{split}
G_r(x,\lambda):&= \frac{1}{r}\cdot \fint_{I_\lambda(r^2)}\fint_{B_x(r)}\big|f_t(x,\lambda)-f_t(x',\lambda')\big|d\rv(x')d\lambda'\\
&\ls \frac{1}{r}\cdot \fint_{I_\lambda(r^2)}\fint_{B_x(r)}\!\Big(\big|f_t(x',\lambda')-f_t(x',\lambda)\big|+\big|f_t(x',\lambda)-f_t(x,\lambda)\big|\Big)d\rv(x')d\lambda'\\
&\!\!\overset{\eqref{eq6.4}}{\ls} \frac{ e^{-2nk}\cdot C_*}{r}\cdot \fint_{I_\lambda(r^2)}\!|\lambda-\lambda'|d\lambda'+ \frac{1}{r} \fint_{B_x(r)}\fint_{I_\lambda(r^2)}\!\big|f_t(x',\lambda)-f_t(x,\lambda)\big|d\lambda'd\rv(x')\\
&\!\!\!\overset{\eqref{eq6.44}}{\ls}  e^{-2nk}\cdot C_*\cdot r +   M(|\nabla f_t(\cdot,\lambda) |)(x)+M[(M(|\nabla f_t(\cdot,\lambda) |)](x),\\
  \end{split}
  \end{equation*}}
for all sufficiently small $r>0,$ where we have used  $|\lambda'-\lambda|\ls r^2$.
By the definition of $F(x,\lambda)$, we have
 \begin{equation}\label{eq6.45}
F(x,\lambda)=\limsup_{r\to0}G_r(x,\lambda)\ls M(|\nabla f_t(\cdot,\lambda) |)(x)+M[(M(|\nabla f_t(\cdot,\lambda) |)](x)
  \end{equation}
 for almost all $x\in \Omega_1$. By the arbitrariness of $\Omega_1\subset\subset B_q(R)$,
 we know that \eqref{eq6.45} holds for  almost all $x\in B_q(R).$
 Now the desired estimate \eqref{eq6.43} is implied by the $L^2$-boundedness of maximal operator (see, for example, Theorem 14.13 in \cite{hk00}).
Notice that the norm $\|M\|_{L^2\to L^2}$ of maximal operator depends only on the doubling constant of $B_q(R)$; and hence, it depends only on $n,k$ and $R.$

According to Proposition \ref{prop6.12}, the function $(x,\lambda)\mapsto -f_t(x,\lambda)$ is a nonnegative sub-solution of the heat equation on $B_q(2R)\times(0,1)$. By using the parabolic version of Caccioppoli inequality (Lemma 4.1 in \cite{mm13}), we can get
$$\sup_{\frac{1}{4}\ls \lambda\ls\frac{3}{4}}\int_{B_q(R)}f_t^2(\cdot,\lambda)d\rv+\int_{B_q(R)\times(\frac{1}{4}, \frac{3}{4})}|\nabla f_t|^2d\underline{\nu}\ls C_3(n,k,R)\cdot \int_{B_q(2R)\times(0,1)}f_t^2d\underline{\nu},$$
 where we have used that $R\ls1$.
In particular, by combining with \eqref{eq6.39}, we have
\begin{equation}\label{eq6.47}
\int_{B_q(R)\times(\frac{1}{4}, \frac{3}{4})}|\nabla f_t|^2d\underline{\nu}\ls C_3(n,k,R)\cdot \rv\big(B_q(2R)\big)\cdot {\rm osc}_{\overline{\Omega'}}^2u.
\end{equation}

On the other hand, fix any $(x,\lambda)\in B_q(R)\times(0, 1)$. From Proposition \ref{prop6.12}, we know that the function $\big( f_t(x,\lambda)-f_t(\cdot,\cdot)\big)_+$
is a sub-solution of the heat equation on $B_q(R)\times(0,1).$ According to Lemma \ref{lem6.13} (noticing that $f_t$ is continuous),
there exists a constant $C_4(n,k,R)$ such that
\begin{equation*}
\sup_{Q_{r/2}(x,\lambda)}\!\big(f_t(x,\lambda)-f_t(x',\lambda')\big)_+\ls \frac{C_4(n,k,R)}{r^2\!\cdot\! \rv\big(B_x(r)\big)}\!\int_{Q_r(x,\lambda)}\big|f_t(x,\lambda)-f_t(x',\lambda')\big|d\underline{\nu}(x',\lambda')
\end{equation*}
for all  $Q_r(x,\lambda)=B_x(r)\times I_\lambda(r^2)\subset\subset B_q(R)\times(0,1).$
 Hence, by the definition of $|\nabla^-f_t|$ and $F$, we have
\begin{equation}\label{eq6.46}
|\nabla^-f_t(x,\lambda)|\ls 2C_4(n,k,R)\cdot F(x,\lambda),\qquad \forall (x,\lambda)\in B_q(R)\times(0,1).
\end{equation}
By integrating \eqref{eq6.46} on $B_q(R)\times(\frac{1}{4},\frac{3}{4})$ and combining with \eqref{eq6.43}, \eqref{eq6.47}, we have
\begin{equation*}
\int_{B_q(R)\times(\frac{1}{4},\frac{3}{4})}|\nabla^-f_t(x,\lambda)|^2d\underline{\nu}\ls 4C^2_4\cdot C_2\cdot C_3\cdot  \rv\big(B_q(2R)\big)\cdot{\rm osc}_{\overline{\Omega'}}^2 u.
\end{equation*}
By combining this with  $\rv\big(B_q(2R)\big)\ls C_5(n,k,R)\cdot  \rv\big(B_q(R)\big)$, we get  the desired estimate \eqref{eq6.41}.
\end{proof}

Now we are in the position to prove the main theorem.
\begin{proof}[Proof of the Theorem \ref{main-thm}] Let us fix a ball $B_q(R)$ with $B_q(2R)\subset \Omega$ and denote by $\Omega'=B_q(R)$.
   Let  $\bar t=\min\{t_*,R^2/(64+64{\rm osc}_{\overline{\Omega'}}u)\}$, where $t_*$ is given in Proposition \ref{prop6.12} for  $\Omega''=B_q(R/2)$.
Denote by
$$v(t,x,\lambda):=-f_t(x,\lambda),\qquad (t,x,\lambda)\in (0,\bar t)\times B_q(R/2)\times[0,1].$$
 According to Proposition \ref{prop6.12}, for each $t\in(0,\bar t)$, the function $v(t,\cdot,\cdot)$ is a sub-solution of the heat equation on the cylinder $B_q(R/2)\times(0,1)$.

 Next, we want to estimate $\frac{\partial^+}{\partial t}v(t,x,\lambda)$.
\begin{slem}\label{slem6.16}
 For any $t\in(0,\bar t)$ and any $(x,\lambda)\in B_q(R/4)\times(0,1)$, we have
\begin{equation}\label{eq6.57}
\begin{split}
\frac{\partial^+}{\partial t}v(t,x,\lambda):&=\limsup_{s\to0^+}\frac{v(t+s,x,\lambda)-v(t,x,\lambda)}{s}\\
&\ls   {\rm Lip}^2u(x)+|\nabla^-f_t(x,\lambda)|^2
\end{split}
\end{equation}
\end{slem}
 \begin{proof}
 For the convenience, we denote by
$$\rho(x,y)\!:=d_Y\big(u(x),u(y)\big)$$
 in the proof of this Sublemma.

 Fix any $(x,\lambda)\in  B_q(R/4)\times[0,1]$ and $t+s\ls \bar t$. We can apply Lemma \ref{lem6.1}(i) to conclude
\begin{equation*}
v(t+s,x,\lambda)=\! \sup_{y\in B_q(R/2)}\Big\{\rho(x,y)-e^{-2nk\lambda}\cdot\frac{|xy|^2}{2(t+s)}\Big\}.
\end{equation*}
We claim firstly that
\begin{equation*}
\frac{|xy|^2}{2(t+s)}=\inf_{z\in\Omega'}\big\{\frac{|xz|^2}{2s}+\frac{|yz|^2}{2t}\big\}.
\end{equation*}

 To justify this, we notice that, by the triangle inequality, any minimal geodesic $\gamma$ between $x$ and $y$ is in $ B_q(R)$.
  By taking $z\in\gamma$ with $|xz|=\frac{s}{s+t}|xy|$, we conclude that the left hand side of the above is greater than the right hand side.
The converse is implied by the triangle inequality.

Thus, we have
\begin{equation*}
\begin{split}
v(t+s,x,\lambda)& =\!\sup_{y\in B_q(R/2)}\Big\{\rho(x,y)-e^{-2nk\lambda}\cdot\inf_{z\in\Omega'}\big\{\frac{|xz|^2}{2s}\!+\!\frac{|yz|^2}{2t}\big\}\Big\}\\
&=\!\!\sup_{y\in B_q(R/2)}\sup_{z\in \Omega'}\Big\{\rho(x,y)-e^{-2nk\lambda}\cdot\frac{|xz|^2}{2s}-e^{-2nk\lambda}\cdot\frac{|yz|^2}{2t}\Big\}\\
&\ls \sup_{z\in \Omega'}\sup_{y\in \Omega'}\Big\{\rho(x,z)+\rho(y,z)-e^{-2nk\lambda}\cdot\frac{|xz|^2}{2s}-e^{-2nk\lambda}\cdot\frac{|yz|^2}{2t}\Big\}\\
&\qquad\qquad\quad ({\rm by\ the\ triangle\ inequality})\\
&= \sup_{z\in \Omega'}\Big\{\rho(x,z)-e^{-2nk\lambda}\cdot\frac{|xz|^2}{2s}+v(t,z,\lambda)\Big\}.
\end{split}
\end{equation*}
Hence, we can get
\begin{equation}\label{eq6.58}
\begin{split}
&\frac{v(t+s,x,\lambda)-v(t,x,\lambda)}{s}\\
&\qquad \ls \sup_{z\in \Omega'}\Big\{\frac{\rho(x,z)+v(t,z,\lambda)-v(t,x,\lambda)}{s}-e^{-2nk\lambda}\cdot\frac{|xz|^2}{2s^2}\Big\}\\
&\qquad \ls \sup_{z\in \Omega'}\Big\{\frac{\rho(x,z)+v(t,z,\lambda)-v(t,x,\lambda)}{s}- \frac{|xz|^2}{2s^2}\Big\}:=RHS,
\end{split}
\end{equation}
where we have used that $k\ls0$.
It is clear that $RHS\gs0$ (by taking $z=x$). On the other hand, if $|xz|\gs s^{1/4},$
then
$$\frac{\rho(x,z)+v(t,z,\lambda)-v(t,x,\lambda)}{s}- \frac{|xz|^2}{2s^2}\ls \frac{3\cdot{\rm osc}_{\overline{\Omega'}}u}{s}-\frac{s^{2/4}}{2s^2}\ls\frac{6{\rm osc}_{\overline{\Omega'}}u-s^{-1/2}}{2s}<0$$
for any  $0<s< (6{\rm osc}_{\overline{\Omega'}}u)^{-2}$. Hence,
$$RHS=\sup_{|xz|<s^{1/4}}\Big\{\frac{\rho(x,z)+v(t,z,\lambda)-v(t,x,\lambda)}{s}-\frac{|xz|^2}{2s^2}\Big\}$$
for all sufficiently small $s>0$.
Now let us continue the calculation of   \eqref{eq6.58}.  By using Cauchy-Schwarz inequality,  we have
\begin{equation*}\label{eq+}
\begin{split}
\frac{v(t+s,x,\lambda)-v(t,x,\lambda)}{s}&\ls \sup_{|xz|<s^{1/4}}\Big\{\frac{\rho(x,z)+v(t,z,\lambda)-v(t,x,\lambda)}{s}-\frac{|xz|^2}{2s^2}\Big\}\\
&\ls \sup_{|xz|<s^{1/4}}\Big\{ \Big(\frac{\rho(x,z)}{|xz|}+\frac{[v(t,z,\lambda)-v(t,x,\lambda)]_+}{|xz|}\Big)\cdot\frac{|xz|}{s}\!-\!\frac{|xz|^2}{2s^2}\Big\}\\
&\ls \frac{1}{2}\sup_{|xz|<s^{1/4}}\Big(\frac{\rho(x,z)}{|xz|}+\frac{[f_t(x,\lambda)-f_t(z,\lambda)]_+}{|xz|}\Big)^2
\end{split}
\end{equation*}
for all sufficiently small $s>0$.
Letting $s\to0^+$, we get the desired equation \eqref{eq6.57}. This completes the proof the sublemma.
\end{proof}

\begin{slem}\label{slem6.17}
 We define a function $\mathscr H(t)$ on $(0,\bar t)$ by
$$\mathscr H(t):=\frac{1}{\rv\big(B_q(R/4)\big)}\int_{B_q(R/4)\times(\frac{1}{4}, \frac{3}{4})}v(t,x,\lambda)d\underline{\nu}(x,\lambda),\qquad t\in(0,\bar t).$$
Then $\mathscr H(t)$ is locally Lipschitz in $(0,\bar t)$.
\end{slem}
\begin{proof}
For the convenience, we continue to denote by
$\rho(x,y)\!:=d_Y\big(u(x),u(y)\big)$
 in the proof of this Sublemma. Given any interval $[a,b]\subset(0,\bar t)$, we have to show that $\mathscr H(t)$ is Lipschitz continuous in $[a,b].$

Let us fix any $t,t'\in [a,b]$.  Take any $(x,\lambda)\in B_q(R/4)\times(0,1)$ and let $y\in \Omega'$ achieve the maximum in the definition of $v(t',x,\lambda)$. Then we have
\begin{equation*}
\begin{split}
v(t',x,\lambda)-v(t,x,\lambda)&=\rho(x,y)-e^{-2nk\lambda}\frac{|xy|^2}{2t'}-\sup_{z\in\Omega'}\big\{\rho(x,z)-e^{-2nk\lambda}\frac{|xz|^2}{2t}\big\}\\
&\ls e^{-2nk\lambda}\cdot\frac{|xy|^2}{2}  \cdot (\frac{1}{t}-\frac{1}{t'})\ls e^{-2nk}\cdot\frac{{\rm diam}^2(\Omega')}{2}\cdot \frac{|t'-t|}{a^2},
\end{split}
\end{equation*}
where we have used that  $k\ls 0$, $\lambda\ls1$ and $t',t\gs a$.
By the symmetry of $t$ and $t'$, we have
\begin{equation*}\label{eq6.59}
|v(t',x,\lambda)-v(t,x,\lambda)|\ls e^{-2nk}\cdot\frac{{\rm diam}^2(\Omega')}{2a^2}\cdot |t'-t|.
\end{equation*}
The integration of this on $B_q(R/4)\times(\frac{1}{4},\frac{3}{4})$ implies the Lipschitz continuity of $\mathscr H(t)$ on $[a,b]$. Therefore, the proof of sublemma is complete.
\end{proof}

Now let us continue to prove the proof of Theorem \ref{main-thm}.

Fixed every $t>0$, from the Sublemma \ref{slem6.16} and Sublemma \ref{slem6.17}, we can apply dominated convergence theorem to conclude
 \begin{equation}\label{eq6.60}
 \begin{split}
\frac{d^+}{dt}\mathscr H(t)&=\limsup_{s\to0^+}\frac{1}{\rv\big(B_q(R/4)\big)}\!\int_{B_q(R/4)\times(\frac{1}{4}, \frac{3}{4})}\frac{v(t+s,x,\lambda)-v(t,x,\lambda)}{s}d\underline{\nu} \\
&\ls \frac{1}{\rv\big(B_q(R/4)\big)}\int_{B_q(R/4)\times(\frac{1}{4}, \frac{3}{4})}\limsup_{s\to0^+}\frac{v(t+s,x,\lambda)-v(t,x,\lambda)}{s}d\underline{\nu}\\
&\ls\frac{1}{\rv\big(B_q(R/4)\big)}\int_{B_q(R/4)\times(\frac{1}{4}, \frac{3}{4})}\Big( {\rm Lip}^2u(x)+|\nabla^- f_t(x,\lambda)|^2(x)\Big)d\underline{\nu}.
\end{split}
\end{equation}
Since $B_q(3R/2)\subset\subset \Omega$, we can use Theorem \ref{thm5.5} to obtain
\begin{equation*}
\int_{B_q(R/4)}{\rm Lip}^2u(x)d\rv(x)\ls C_1\cdot\int_{B_q(R/4)}|\nabla u|_2(x)d\rv(x)\ls C_1\cdot E^u_2\big(B_q(R/4)\big).
\end{equation*}
Here and in the following of the proof, all of constants $C_1,C_2, \cdots, $ depend only on $n,k$ and $R.$ By combining with Lemma \ref{lem6.14} and
\eqref{eq6.60}, we have
\begin{equation*}\label{eq6.61}
  \frac{d^+}{dt}\mathscr H(t) \ls\frac{C_1}{2}\cdot\frac{ E^u_2\big(B_q(R/4)\big)}{\rv\big(B_q(R/4)\big)}+ C_2 \cdot{\rm osc}_{\overline{\Omega'}}^2u\ls C_3\bigg(\frac{ E^u_2\big(B_q(R)\big)}{\rv\big(B_q(R)\big)}+ {\rm osc}_{\overline{\Omega'}}^2u\bigg),
 \end{equation*}
where we have used that $\rv\big(B_q(R)\big)\ls C(n,k,R)\cdot \rv\big(B_q(R/4)\big)$. Denoting by
$$\mathscr A_{u,R}:=\bigg(\frac{ E^u_2\big(B_q(R)\big)}{\rv\big(B_q(R)\big)}\bigg)^{\frac{1}{2}}+{\rm osc}_{\overline{B_q(R)}}u,$$
we have  $\frac{d^+}{dt}\mathscr H(t) \ls 2C_3\cdot \mathscr A^2_{u,R}.$

We notice that $\lim_{t\to0^+}v(t,x,\lambda)=0$ for each given $(x,\lambda)\in B_q(R/4)\times(0,1)$. Indeed, from Lemma 6.1(i),
$$v(t,x,\lambda)=\max_{\overline{B_x(\sqrt{C_*t})}}\Big\{d_Y(u(x),u(y))-e^{-2nk\lambda}\frac{|xy|^2}{2t}\Big\}\ls \max_{\overline{B_x(\sqrt{C_*t})}} d_Y(u(x),u(y)).$$
By combining this with the continuity of $u$, we deduce that   $\lim_{t\to0^+}v(t,x,\lambda)=0$.
Since $v(t,\cdot,\cdot)$ is bounded  from \eqref{eq6.39}, we can use dominated convergence theorem to conclude that
$\lim_{t\to0^+}\mathscr H(t)=0$.  By combining this with Sublemma \ref{slem6.17} and $\frac{d^+}{dt}\mathscr H(t) \ls 2C_3\cdot \mathscr A^2_{u,R}$, we have
\begin{equation}\label{eq6.62}
\mathscr H(t)\ls 2C_3\cdot t\cdot  \mathscr A^2_{u,R}.
\end{equation}
for any $t\in(0,\bar t)$,

Let us recall Proposition \ref{prop6.12} that, for each $t\in(0,\bar t)$, the function $v(t,\cdot,\cdot)$ is nonnegative and a sub-solution of the heat equation on the cylinder $B_q(R/2)\times(0,1)$, hence so is the function $\frac{v(t,\cdot,\cdot)}{t}.$ By using Lemma \ref{lem6.13} and $R\ls1$, we obtain
{\small\begin{equation}\label{eq6.63}
\begin{split}
\sup_{B_q(R/8)\times(\frac{3}{8},\frac{5}{8})}\frac{v(t,x,\lambda)}{t}&\ls   \frac{C_4}{R^2\cdot \rv\big(B_q(R/4)\big)}\int_{B_q(R/4)\times(\frac{1}{4},\frac{3}{4}) }\!\!\frac{v(t,x,\lambda)}{t}d\underline{\nu}(x,\lambda)  \\&= \frac{C_4}{R^2}\cdot\!\frac{\mathscr H(t)}{t}
 \overset{\eqref{eq6.62}}{\ls} \frac{C_4}{R^2}\cdot 2 C_3\cdot \mathscr A^2_{u,R}:=  C_5\cdot \mathscr A^2_{u,R}.
\end{split}
\end{equation}}
Given any $x,y\in B_q(R/8)$, from the definition of $v(t,x,\lambda)$, we can apply \eqref{eq6.63} to $v(t,x,\frac{1}{2})$ and deduce
\begin{equation}\label{eq6.64}
\frac{d_Y\big(u(x),u(y)\big)}{t}-e^{-nk}\frac{|xy|^2}{2t^2} \ls \frac{v(t,x,\frac{1}{2})}{t}\ls  C_5\cdot \mathscr A^2_{u,R}
\end{equation}
for  all $t\in(0,\bar t).$ Now, if $|xy|< e^{nk/2}\cdot \mathscr A_{u,R}\cdot \bar t$, by choosing $t=\frac{|xy|}{\mathscr A_{u,R}\cdot e^{nk/2}}$ in \eqref{eq6.64}, we have
\begin{equation}\label{eq6.65}
 \frac{d_Y\big(u(x),u(y)\big)}{|xy|} \ls  \Big(C_5+\frac{1}{2}\Big)\cdot e^{-nk/2} \mathscr A_{u,R}:=C_6\cdot\mathscr A_{u,R}.
 \end{equation}

 At last, let $x,y\in B_q(R/16).$ If $|xy|< e^{nk/2}\cdot \mathscr A_{u,R}\cdot \bar t$, then \eqref{eq6.65} holds. If
 $|xy|\gs e^{nk/2}\cdot \mathscr A_{u,R}\cdot \bar t$, we can take some minimal geodesic $\gamma$ between $x$ and $y$. The triangle inequality implies that $\gamma\subset B_q(R/8)$.  By choosing points $x_1,x_2,\cdots, x_{N+1}$ in $\gamma$ with
 $x_1=x,\ x_{N+1}=y$ and $|x_ix_{i+1}|< e^{nk/2}\cdot \mathscr A_{u,R}\cdot \bar t$ for each $i=1,2,\cdots,N$  and by using the triangle inequality and \eqref{eq6.65}, we have
 $$d_Y\big(u(x),u(y)\big)\ls \sum_{i=1}^{N}d_Y\big(u(x_i),u(x_{i+1})\big)
 \ls C_6\cdot\mathscr A_{u,R}\cdot  \sum_{i=1}^{N}|x_i x_{i+1}|=C_6\cdot\mathscr A_{u,R}\cdot |xy|.$$
That is, \eqref{eq6.65} still holds.
Therefore the proof of Theorem \ref{main-thm} is complete.
\end{proof}

\end{document}